\documentclass[a4paper]{article}
\usepackage{fullpage}
\usepackage{setspace}
\usepackage{mathtools}
\usepackage{bbm}
\usepackage{amsmath}
\usepackage{amsfonts}
\usepackage{graphicx}
\usepackage{cancel}
\usepackage{hyperref}
\usepackage{courier}
\usepackage{subfigure}
\usepackage{color}
\usepackage[english]{babel}
\usepackage{mathbbol}
\usepackage{multirow}
\usepackage{subfigure}
\usepackage{latexsym}
\usepackage{pdfsync}
\usepackage{epsfig}
\usepackage{subfigure}
\usepackage{empheq}
\definecolor{Blue}{rgb}{0.3,0.3,0.9}
\usepackage{amsmath}
\usepackage{amsthm,amsbsy}
\usepackage{mathcomp}
\usepackage{textcomp}

\usepackage{amssymb}
\usepackage{graphicx}
\usepackage{graphicx}
\usepackage{dcolumn}
\usepackage{bm}
\usepackage{amsfonts}
\usepackage{latexsym}
\usepackage{pdfsync}
    \usepackage{fullpage} 
\usepackage{colordvi}
\usepackage{color}
\usepackage{booktabs}
\usepackage{graphicx}
\usepackage{subfigure}
\usepackage{stmaryrd}
\usepackage{authblk}
\usepackage{cite}
\usepackage[linesnumbered,commentsnumbered,ruled]{algorithm2e}
\usepackage{epstopdf}
\input xy
\xyoption{all}

\newtheorem{thm}{Theorem}[section]

\newtheorem{prop}[thm]{Proposition}

\newtheorem{rem}[thm]{Remark}
\newtheorem{alg}[thm]{Algorithm}

\newcommand{\commentout}[1]{}

\newcommand{\nwc}{\newcommand}

\renewcommand{\1}{\mathbb{1}}

\nwc{\bR}{\mb R}
\nwc{\bH}{{\mb H}}
\nwc{\bxp}{{{\mathbf x}}}
\nwc{\bap}{{{\mathbf y}}}

\nwc{\bPhi}{\mathbf{\Phi}}
\nwc{\bPsi}{\mathbf{\Psi}}
\nwc{\bh}{\mathbf h}

\nwc{\bI}{\mathbf I}
\nwc{\bP}{\mathbf P}
\nwc{\bs}{\mathbf s}
\nwc{\bd}{\mathbf{d}}
\nwc{\bX}{\mathbf X}

\nwc{\om}{\beta'}

\nwc{\nwt}{\newtheorem}
\nwc{\xp}{{x^{\perp}}}
\nwc{\yp}{{y^{\perp}}}
\nwt{remark}{Remark}
\nwt{definition}{Definition}
\nwt{corollary}{Corollary} 

\nwc{\ba}{{\mb a}}
\nwc{\bal}{\begin{align}}
\nwc{\ben}{\begin{equation*}}
\nwc{\beqq}{\begin{equation}}
\nwc{\bea}{\begin{eqnarray}}
\nwc{\beq}{\begin{eqnarray}}
\nwc{\bean}{\begin{eqnarray*}}
\nwc{\beqn}{\begin{eqnarray*}}
\nwc{\beqast}{\begin{eqnarray*}}

\nwc{\eal}{\end{align}}
\nwc{\een}{\end{equation*}}
\nwc{\eeqq}{\end{equation}}
\nwc{\eea}{\end{eqnarray}}
\nwc{\eeq}{\end{eqnarray}}
\nwc{\eean}{\end{eqnarray*}}
\nwc{\eeqn}{\end{eqnarray*}}
\nwc{\eeqast}{\end{eqnarray*}}

\nwc{\vep}{\varepsilon}
\nwc{\ep}{\epsilon}
\nwc{\ept}{\epsilon}
\nwc{\vrho}{\varrho}
\nwc{\orho}{\bar\varrho}
\nwc{\ou}{\bar u}
\nwc{\vpsi}{\varpsi}
\nwc{\lamb}{\lambda}
\nwc{\Var}{{\rm Var}}

\nwt{proposition}{Proposition}
\nwt{theorem}{Theorem}
\nwt{summary}{Summary}
\nwt{lemma}{Lemma}
\nwc{\nn}{\nonumber}
\nwc{\mf}{\mathbf}
\nwc{\mb}{\mathbf}
\nwc{\ml}{\mathcal}

\nwc{\IA}{\mathbb{A}} 
\nwc{\bi}{\mathbf i}
\nwc{\bo}{\mathbf o}
\nwc{\IB}{\mathbb{B}}
\nwc{\IC}{\mathbb{C}} 
\nwc{\ID}{\mathbb{D}} 
\nwc{\IM}{\mathbb{M}} 
\nwc{\IP}{\mathbb{P}} 
\nwc{\II}{\mathbb{I}} 
\nwc{\IE}{\mathbb{E}} 
\nwc{\IF}{\mathbb{F}} 
\nwc{\IG}{\mathbb{G}} 
\nwc{\IN}{\mathbb{N}} 
\nwc{\IQ}{\mathbb{Q}} 
\nwc{\IR}{\mathbb{R}} 
\nwc{\IT}{\mathbb{T}} 
\nwc{\IZ}{\mathbb{Z}} 

\nwc{\cE}{{\ml E}}
\nwc{\cP}{{\ml P}}
\nwc{\cQ}{{\ml Q}}
\nwc{\cL}{{\ml L}}
\nwc{\cX}{{\ml X}}
\nwc{\cW}{{\ml W}}
\nwc{\cZ}{{\ml Z}}
\nwc{\cR}{{\ml R}}
\nwc{ \cV}{{\ml V}}
\nwc{\cT}{{\ml T}}
\nwc{\crV}{{\ml L}_{(\delta,\rho)}}
\nwc{\cC}{{\ml C}}
\nwc{\cO}{{\ml O}}
\nwc{\cA}{{\ml A}}
\nwc{\cK}{{\ml K}}
\nwc{\cB}{{\ml B}}
\nwc{\cD}{{\ml D}}
\nwc{\cF}{{\ml F}}
\nwc{\cJ}{{\ml J}}
\nwc{\cS}{{\ml S}}
\nwc{\cM}{{\ml M}}
\nwc{\cN}{{\ml N}}
\nwc{\cG}{{\ml G}}
\nwc{\cH}{{\ml H}}
\nwc{\bk}{{\mb k}}
\nwc{\bn}{{\mb n}}
\nwc{\cbz}{\overline{\cB}_z}
\nwc{\supp}{{\hbox{supp}}}
\nwc{\fR}{\Re}
\nwc{\bY}{\mathbf Y}

\nwc{\pft}{\cF^{-1}_2}
\nwc{\bU}{{\mb U}}
\nwc{\bG}{{\mb G}}
\nwc{\bg}{\mathbf{g}}
\nwc{\mbB}{\mathbf{B}}
\nwc{\mbc}{\mathbf{c}}
\nwc{\mbm}{\mathbf{m}}
\nwc{\mbM}{\mathbf{M}}
\nwc{\mbf}{\mathbf{f}}
\nwc{\mbg}{\mathbf{g}}
\nwc{\mbgt}{\widetilde{\mathbf{g}}}
\nwc{\mbh}{\mathbf{h}}
\nwc{\mbP}{\mathbf{P}}
\nwc{\mbe}{\mathbf{e}}
\nwc{\be}{\mathbf{e}}
\nwc{\Om}{\beta'}
\nwc{\ind}{\operatorname{I}}
\nwc{\mbx}{\mathbf{f}}
\nwc{\bb}{\mathbf{g}}
\nwc{\xmax}{f_{\rm max}}
\nwc{\xmin}{f_{\rm min}}
\nwc{\suppx}{\hbox{\rm supp} (\mbf)}
\nwc{\by}{\mathbf{h}}
\nwc{\bZ}{\mathbf{Z}}
\nwc{\bF}{\mathbf{F}}
\nwc{\bE}{\mathbf{E}}
\nwc{\bV}{\mathbf{V}}
\nwc{\cI}{\IZ^2_N}
\nwc{\chis}{{\chi^{\rm s}}}
\nwc{\chii}{{\chi^{\rm i}}}
\nwc{\pdfi}{{f^{\rm i}}}
\nwc{\pdfs}{{f^{\rm s}}}
\nwc{\pdfii}{{f_1^{\rm i}}}
\nwc{\pdfsi}{{f_1^{\rm s}}}
\nwc{\thetatil}{{\tilde\theta}}
\nwc{\red}{\color{red}}
\nwc{\prox}{\hbox{prox}} 
\nwc{\diag}{\hbox{\rm diag}}
\nwc{\sloc}{J_{\rm f}}
\nwc{\bu}{\xi}
\nwc{\bv}{\beta}
\nwc{\cU}{\mathcal{U}}
\nwc{\bN}{\mathbf{N}}
\nwc{\bw}{\mathbf{w}}
\nwc{\im}{i}
\nwc{\bom}{\mathbf{w}}
\nwc{\bt}{\mathbf{t}}
\nwc{\z}{y}
\nwc{\cY}{\mathcal{Y}}

\title{Matrix balancing based interior point methods for point set matching problems\thanks{Submitted to the editors DATE.
Funding: The research  of the  author is supported  by grant 110-2115-M-005-007-
MY3 from the Ministry of Science and Technology, Taiwan.
}
}

\author[1]{Janith Wijesinghe\footnote{Email address: janithuop@gmail.com}}
\author[2]{Pengwen Chen\footnote{Corresponding author. Email address: pengwen@nchu.edu.tw, pengwen@email.nchu.edu.tw}}
\affil[1,2]{Applied Mathematics, National Chung Hsing University, Taichung, 402,  Taiwan.
}

\begin{document}

\maketitle

\begin{abstract} 
Point sets matching problems can be handled  by optimal  transport. The mechanism behind it is that  
optimal transport recovers  the  point-to-point correspondence  associated with  the  least curl deformation.  Optimal transport is a special form of linear programming with dense constraints. Linear programming can be handled by interior point methods, provided that  the involved  ill-conditioned Hessians can be computed accurately. 
During the decade,   matrix balancing has been employed to  compute optimal transport under  entropy regularization approaches. 
The solution quality  relies on two factors: the accuracy of matrix balancing and the boundedness of the dual vector. 
High accurate matrix balancing is achieved by the application of  Newton methods on a sequence of matrices along  a central path.
 In this work, we  apply
   sparse support constraints to
   matrix-balancing based   interior point methods, in which   the sparse  set fulfilling total support   is iteratively updated to truncate the domain of the transport plan.   Total support  condition  is one crucial condition, which  guarantees the existence of matrix balancing as well as the boundedness of the dual vector. 
\end{abstract}



\section{Introduction}
Registration aims to  match two or more sets of image data altered by geometric transforms, taken
at different times,  or from different sensors. 
Point set representing  image data is commonly employed  to reduce the computational load in computer vision. The associated point-set matching problem(registration) 
 is to establish a consistent
point-to-point correspondence between two point sets and to estimate the
spatial alignment transformation. The quality of correspondence   plays a crucial role in estimating followup transformations in registration. 
The iterative closest point (ICP)
algorithm is one classic and  popular  approach in feature-based image registration \textcolor{black}{ problems},
because of its simplicity~\cite{RefWorks:McKay}. For correspondence correctness,
the ICP algorithm  requires sufficient overlap between the point
sets. Its vulnerability in performance also includes the proneness to outliers.
 To alleviates these difficulties,
researchers describe  the correspondence by a permutation matrix, which minimizes some ``distance" of the point-sets, typically consisting of  one regularization term for transformations and one assignment term for correspondence.
 For instance,   Chui and Rangarjan proposed a robust point matching  method(RPM), which
estimates non-rigid transformation and correspondence simultaneously, where the point-to-point correspondence is  enforced  by   Sinkhorn matrix balancing\cite{RefWorks:ChuiCVPR}.  This can be viewed as one early application of optimal transport in registration. 
Comprehensive surveys  of  traditional registration methods can be found in~\cite{Maintz} and \cite{RefWorks:Zitova}.


Two unlabeled point sets can be regarded as two histograms, whose distance can be fast evaluated by 
various information divergences, e.g., Hellinger distances,  Kullback-Leibler divergences and Jensen Shannon divergences. From a perspective of  correspondence retrieval, 
 a natural  choice is  Wasserstein distance (also known as the earth mover's distance \cite{Rubner2000}).
Wasserstein  distance  quantifies the minimal cost of moving the probability
mass from one distribution to the other distribution.
In the 1780's,   Monge described a problem of
transporting a pile of soil with the least amount of work. 
In the
1940's,   Kantorovich \cite{RefWorks:Kant} employed a dual variation
principle to convert  the original nonlinear problem into a linear programming problem and to
study the optimal solutions.  A
survey of theoretical works  on  this problem can be found in
\cite{RefWorks:Evans}\cite{Vill1} or \cite{Vill2}. 
Nowadays, optimal transport has been applied  in various tasks, including   image retrieval,   image registration, image morphing, shape matching and maching learning, see~\cite{WP}, \cite{RefWorks:Kaijser}, \cite{RTG}, \cite{ZYHT}, \cite{RefWorks:Rabin},\cite{Peyre2019},\cite{Schmitzer2015}, \cite{RefWorks:Museyko},\cite{Chen2013},\cite{Kolouri2017},\cite{Cuturi2014}.

  In the application of point-set registration, 
 we can incorporate optimal transport  in  feature based methods to estimate the transforms and the correspondence in the existence of  outliers,  for instance,   Hellinger distances based point set matching model (HD)\cite{JMIV}.
The HD model  can be regarded as an approximation of  optimal transport,  when the kernel scale tends to infinity. With a finite kernel scale,  the measure preserving constraint is relaxed to   tolerate  the existence of outliers.   
The effectiveness  of this application  generally depends on the hypothesis of  geometric  transforms.  A fundamental question is, for which class of transformations the underlying  point correspondence can be reconstructed correctly?
  Impressively, when the transformation  can be expressed as  the gradient of some convex function,   the underlying correspondence can be recovered correctly by solving  the $L^2$ optimal  transport problem. The set of transformations   includes  scalings,  translations,  positive definite affine transforms and other curl-free maps.
    This property makes optimal transport  models suitable and robust in certain
applications. For instance, 
 \cite{Chen2013} applies the optimal mass transport model  to match  lung vessel branch points,  which  are   extracted from  two computed tomography(CT) lung images
acquired during breath-holds.  Although the physical deformation field is rather large and complex, the correspondence reconstruction is surprisingly almost perfect, which  verifies the superiority of the optimal transport model. 

Despite of the theoretical advantage, 
 optimal transport is   limited by its heavy computational  requirement in practical applications.
 Briefly, as one member of linear programming,  optimal transport can be solved by various algorithms in linear programming. 
Standard  algorithms  include the simplex method and the interior point method~\cite{Robere2012}\cite{Luenberger2016}\cite{Gondzio2012}.  
Nowadays  the primal-dual interior method is  an efficient interior point method in solving linear programming, when the problem size is moderate~\cite{Wright1997}. 
Thanks to  second-order convergence in each sub-problem,  an interior point method can quickly generate accurate solutions from proper matrix-free algorithms.  For instance, 
in the community of machine learning,  Wasserstein barycenter  is one average of multiple discrete probability measures in terms of Wasserstein distance\cite{YangLST21, Ge2019},
where accurate solutions can be computed by interior point methods~\cite{Ge2019}. 
  In considering the flexibility of handling transformation and correspondence simultaneously,   we  focus on 
    the negative entropy function as a regularizer to handle optimal transport in the registration problem. This regularization   elegantly converts  optimal transport  to one matrix balancing task. Actually,  matrix balancing algorithms are known as an  effective tool  to produce one approximation of the optimal transport plan~\cite{Cuturi2013}\cite{Benamou2015}\cite{Karlsson2017}\cite{Schmitzer2019}. The major numerical tool is the Sinkhorn balancing algorithm\cite{Sinkhorn}\cite{Knopp1967}. To improve the convergence speed of Sinkhorn algorithm,  the $\epsilon$-scaling heuristic  and the kernel truncation are introduced to reduce the number of iterations and the number of  variables to reduce the computational load\cite{Schmitzer2019}.

 \subsection{Contributions}
 This paper is concerned with  the application of this matrix balancing based interior point methods in solving point-set matching problems. 
 The main question  is whether we can  develop  a proper central path for  discrete optimal transport approximations with small regularization parameters? 
The contribution can be summarized as follows.  
 First, we investigate the application of Newton methods in 
  matrix balancing based interior point methods for optimal transport.   
  Although Sinkhorn  balancing algorithm is popular and widely used in balancing  matrices, it is generally  difficult  to produce   an accurate result quickly for  our application. 
  In this paper,  
 we propose Sinkhorn-Newton Negative entropy  interior point methods(SNNE) in~\ref{SNNE_sec}, where 
 Newton directions is computed  by  matrix-free conjugate gradient methods.
  One underlying challenging is that   as the central path heads toward an optimal permutation solution,  
  the rank of the associated Schur complement matrices  reduces to $n$, where
      $n$ is  the point cardinality in each point-set.
During the rank-reduction process,    it is numerically  challenging to maintain the accuracy of Newton iterates.
To overcome this, we adopt the techniques proposed in the stabilized scaling algorithm~\cite{Schmitzer2019}, including computations in the Log-domain and the translation of scaling vectors.  See section~\ref{sec3.4}. 
  
Second, we revisit  a few matrix balancing algorithms, including  the  Knight-Ruiz(KR) fixed point method~\cite{Knight2012}.
 Our matrix-balancing  experiments confirm the  excellent performance of KR algorithm, although its  global convergence is unclear.
   To reveal the connection between KR and other Newton  methods, we 
 introduce  one  convex  function for matrix balancing task and propose 
  a novel modified Newton method, called LB algorithm.
  The  KR algorithm
  is  the modified Newton method with step size $1$. Theorem~\ref{LBstepsize} indicates that  when LB is applied to  a matrix with total support, 
     the step size  will be  $1$, as the iterates  get close to  an optimal solution. 

%

Third, as in the kernel truncation method\cite{Schmitzer2019},
   sparse support sets can be  imposed to   reduce the memory requirement  in the application of  the interior point methods  to large-scale problems.
   However, the truncated kernel matrix does not always have total support, which is crucial 
    to guarantee the quality of  matrix balancing computation and the boundedness of the scaling vectors.
      In Prop.~\ref{total}, we propose one  simple method to construct  one sparse support set with total support, and 
      propose SNNE-sparse in Alg.~\ref{SNNEalg}, which are cable of   handling  large-scale  matching problems.
        To evaluate sparse support matrix balancing  methods,
        Theorem~\ref{thm1} gives one  error bound estimate, which relates the boundedness of the dual vector to the duality measure estimate.
         According to Remark~\ref{3.6}, the boundedness of the dual vector can be ensured, if  the truncated matrix satisfies the  total support  condition.

This paper is organized as follows.
In section 2,  we describe  the application of   optimal transport in point-set registration.
Discrete optimal transport can be solved by matrix balancing based  interior point methods, including SNNE and SNNE-sparse. 
In section 3, we describe a few matrix balancing schemes, including Sinkhorn-Knopp balancing, Knight-Ruiz scheme and other Newton methods. Matrix balancing can be achieved through minimizing  a convex function. 
In section 4, we present a few numerical simulations, which \textcolor{black}{demonstrate} the effectiveness of the proposed algorithms SNNE and SNNE-sparse. 

\subsection{Notations}\label{notation}
 In this paper,  let $\langle x, y \rangle$ denote the inner product between $x,y$ in $\IR^n$.  
  For  a vector $x\in \IR^n$ and a scalar $\epsilon\in \IR$, let $y=(x>\epsilon)$ denote a zero-one vector, i.e., for $i=1,\ldots, n$, set $y_i=1$ if $x_i>\epsilon$,  and set $y_i=0$ otherwise. \textcolor{black}{For simplicity of notation, the functions $\exp$ and $\log$ are extended to vector spaces $\IR^n$ by componentwise application to all components: $(\exp( x))_i=\exp( x_i)$, $(\log x)_i=\log x_i$, $i=1,\ldots, n$. Likewise, let  $x^{-1}$ be the vector whose entries are  $x_i^{-1}$. Let the operator $\odot$ denote entrywise multiplication, e.g., $x\odot y\in \IR^n$ and $(x\odot y)_i=x_i y_i$.}
     Let $\1_n=[1,1,\ldots, 1]^\top\in \IR^n$  be the vector whose entries are all one.
Let $[x;y]$ denote the stacked vector $[x^\top, y^\top]^\top$ for any two vectors $x,y$. The norm $\|\cdot \|$ represents the 2-norm.  
Let $\IT$ be the reshape operator $x\in \IR^{n^2}\to \IR^{n\times n}$, $\IT(x)\in \IR^{n\times n}$, $\IT(x)_{i,j}=x_{in+j}$ for $i,j\in \{1,\ldots, n\}$. In addition, for the sake of simplicity,  $x_{i,j}$ stands for $\IT(x)_{i,j}$ if no confusion occurs. 
Let $\Pi_n$ denote the set of doubly stochastic matrices, i.e., row stochastic and column stochastic $\IT(x) \1_n=\1_n=\IT(x)^\top \1_n$ for \textcolor{black}{ each $\IT(x)\in \Pi_n$}. Finally,   $A^\dagger$ stands for the pseudo inverse of  a matrix $A$.

\section{Optimal transport}

\subsection{Matching point-sets  under deformations}
We first review the deformation characterization of optimal transport  applied on the point-set matching problems  in  the previous work\cite{Chen2013}. 
 The primary  focus of the point set matching  is  the reconstruction of 
 the correspondence between   two unlabeled point-sets
 $\{z_i\}_{i=1}^n\subset\Omega$ and $\{y_i\}_{i=1}^n\subset  T(\Omega)$, where
 $T$ is some  injective and orientation-preserving deformation
   on a bounded open connected subset 
$\Omega$ of $\IR^3$.  
The correspondence  can be described by a permutation 
 $\tau$  such that  $y_i=T(z_{\tau(i)} )$ and some  optimal condition hold for $\tau$. One natural criterion is   the minimization problem: 
\begin{equation} \label{tau} 
\min_{\tau} \sum_{i=1} ^n \|y_i- z_{\tau (i)} \|^2.
\end{equation} 
This is a discrete combinatorial optimization problem,  because  $n!$ possibilities must be evaluated.
 This difficulty can be alleviated,  if we consider the  relaxed continuous  problem, 
\begin{equation} \label{MK1} 
\min_{X_{i, j} } \sum_{i=1} ^n \|y_i-z_j\|^2 X_{i, j}, 
\end{equation} 
 subject to the unit mass constraints  $\sum_{i=1} ^n X_{i, j} =1=\sum_{j=1} ^n X_{i, j}$ and $  X_{i, j} \ge 0.$ The problem
   is known as the $L^2$ Monge-Kantorovich mass transport problem. 
      The relaxed problem described by   Eq.~(\ref{MK1}) is 
    a convex (in fact,  linear) minimization  problem,  which  has   an optimal  permutation matrix 
(the existence of this is guaranteed by Birkhoff's theorem) 
  and     can be solved  by interior point methods\cite{Boyd}  or 
 primal-dual algorithms~\cite{RefWorks:Kaijser}  (see  chapter 4 in~\cite{Assignment2009}). 
 
 In the context of { (\ref{tau})},    the  permutation $\tau$ corresponding to the  permutation 
$X$ is optimal,  if and only if  $\{( z_{\tau(i)}, y_i)\}_{i=1}^n$ is  cyclically monotone.  
Consider   a transform $T:\IR^d\to\IR^d$ between two point sets $\{z_i\}_{i=1}^n,  \{y_i\}_{i=1}^n$ in $\mathbb{R}^d$ with $y_i=T(z_i)$.
  \emph{When  a (unknown) transform  between these  point sets is the gradient of some convex function,  then the correspondence can be recovered correctly by solving mass transport problems.} The set of  transforms  includes  scalings,  translations,  and other curl-free maps.
 Point  correspondence  can be reconstructed  correctly 
from optimizing   transport objectives, 
 if the  transform $T$ between point-sets is   the gradient of some convex function. 
In general, for    a  point-set with finite cardinality  $n$ sampled from $\Omega\subset \IR^3$,  
   when the  curl  of the transform 
  is sufficiently small, 
    then
   the underlying correspondence 
coincides with  a minimizer 
$\{X_{i,j} \} _{i,j=1} ^n$  of Eq.~(\ref{MK1}).
Empirical studies show the outstanding performance of optimal transport in recovering the point-to-point correspondence under a small curl deformation~\cite{Chen2013}.

\subsection{Discrete optimal  transport}
To solve (\ref{MK1}), introduce a vector $c\in \IR^{n^2}$ and its associated (reshaped)  matrix $\IT(c)$ with $\IT(c)_{i,j}=\|y_i-z_j\|^2$.
We can express (\ref{MK1}) as  the primal problem (transportation):  searching for the optimal solution $x\in \IR^{n^2}$ in 
\beqq\label{P0}
\min_ {\IT(x)\in \Pi_n } \langle c,  x\rangle, 
\eeqq
where  $\Pi_n$ is the set 
 \beqq\label{M}
\{\IT(x):  M x:=[\IT(x) \1_n;\IT(x)^\top \1_n]=\1_{2n}, x\ge 0\}.
 \eeqq 
 The matrix $X=\IT(x)$ represents a coupling  matrix $X=[X_{i,j}\ge 0: i,j=1,\ldots n]$, whose entry $X_{i,j}$ describes the amount of mass flowing from bin $i$ toward bin $j$. 
  The problem in  (\ref{P0}) is also  known as the assignment problem with assignment matrix $\IT(c)$. For each  feasible  solution $x$,  at most  $n$ entries can reach  the value $1$, i.e.,   $\IT(x)$ is a permutation matrix. By Birkhorff theorem, the extreme points of the set of doubly stochastic matrices are the permutation matrices.

The action of the adjoint operator  $M^\top$ on a vector $\nu=[\nu^{(1)}; \nu^{(2)} ]\in \IR^{2n}$
 is given by 
\beqq\label{Mt}
M^\top \nu=\IT^{-1}(\nu^{(1)} \1_n^\top+\1_n {\nu^{(2)}}^\top).
\eeqq
Its dual problem to (\ref{P0}) is the maximization problem with respect to a dual variable $\nu\in \IR^{2n}$,
\beqq\label{Pd}
\max_{\nu}\{ \1_{2n}^\top \nu: M^\top \nu\le c\}.
\eeqq
The optimal condition of the primal and dual problem is characterized by 
the Karush-Kuhn-Tucker(KKT) conditions, i.e.,  the nonnegativeness of a slack vector in (\ref{Pd}),
\beqq\label{KKT}
s:=c-M^\top \nu\ge 0
\eeqq
holds and 
   $\IT(s)_{i,j}>0$ occurs only for those indices   $(i,j)$ with  $X_{i,j}=0$.
The slackness condition actually implies zero duality gap,
\beqq\label{duality}
\langle c, x\rangle-\langle \nu, \1_{2n}\rangle=
\langle c, x\rangle-\langle M^\top\nu, x\rangle=\langle s, x\rangle= 0.
\eeqq

\subsection{Interior point methods}
Here we quickly illustrate the application of  interior point methods to   (\ref{P0}). More details can be found in textbooks\cite{Boyd} and \cite{Luenberger2016}. 
We start with  log-barrier functions for a basic conceptual introduction of interior point methods, which motivates the
 negative entropy  barrier functions   in our interior point methods.

  
To reach one optimal solution of  (\ref{P0}), 
path-following methods~\cite{Fiacco1968} solve the associated logarithmic barrier function  with larger and larger values of $t\in \{t_j: 0<t_0<t_1<t_2<\ldots\}$, 
\beqq\label{IP_MB}
\min_ {\IT(x)\in \Pi_n } \{  c^\top  x-t^{-1}\langle  \1_{n^2},  \log x\rangle\}. 
\eeqq
For each $t=t_j>0$, let  $x=x^{(t)}$ be
the critical point of
 the Lagrangian function, 
\beqq\label{eq_37}
\min_x \{  f(x,\nu):=c^\top x-t^{-1}\langle \1_{n^2}, \log x\rangle-\nu^\top  (Mx-\1_{2n})\}.
\eeqq
We compute the central point $x^{(t_j)}$ starting from the previously computed central point $x^{(t_{j-1})}$. 
The following  proposition shows   the KKT condition of (\ref{IP_MB}).  The proof can be given by  the direct calculus.  
\begin{prop}\label{prop3.1}Consider (\ref{IP_MB}) with $t>0$.
Introducing a multiplier vector $\nu$ for the constraint $Mx=\1_{2n}$,
 we have 
the Lagrangian function
\beqq
c^\top x-t^{-1}\langle  \1_{n^2},  \log x\rangle-\nu^\top  (Mx-\textcolor{black}{\1_{2n}}).
\eeqq 
The optimal condition of $x$ is 
\beqq\label{KKTt}
 c\odot x- \textcolor{black}{ t^{-1}  \1_{n^2}}=\diag(x) M^\top \nu, \; i.e.,  t x=(c-M^\top \nu)^{-1},
\eeqq
\textcolor{black}{
where  thanks to the constraint $Mx=\1_{2n}$,  $\nu$ is a root of  the nonlinear equation,
\beqq
M(c-M^\top \nu)^{-1}- t\1_{2n}=0, \textrm{ subject to } M^\top \nu<c.
\eeqq }
\end{prop}
The condition in (\ref{KKTt}) states that   $c\odot x-\textcolor{black}{ t^{-1} \1_{n^2}}$ lies in the range of $\diag(x) M$ for the optimal interior point $x>0$ in $\Pi_n$. 
Taking the product (\ref{KKTt}) with $x$ yields the duality gap $t^{-1}n^2$ associated with  finite $t$, i.e., 
\beqq\label{eq_14}
c^\top  x-\1_{2n}^\top \nu=t^{-1}n^2\ge 0,
\eeqq
which provides a measure of  closeness to optimality. The optimal solution of (\ref{P0}) can be obtained from  a limit of $x^{(t)}$ as $t\to \infty$.

 \subsubsection{  Matrix-free conjugate gradient methods for  central path}\label{sec_Newton}

We illustrate the matrix-free computation of $x^{(t)}$. The argument is standard,  for instance, see~\cite{Luenberger2016}.  
We start with one initial point $x^{(t_0)}$ in $\Pi_n$.
To approximate the critical point $x^{(t)}$ in (\ref{eq_37}), we 
generate  a minimizing sequence $\{(x_k,\nu_k): k=1,2,3,\ldots\}$ of (\ref{eq_37}) with step size $\alpha>0$,
 \beqq \label{eq_20}x_{k+1}=x_k+\alpha d_k\in \IT^{-1}(\Pi_n),\;  \nu_{k+1}=\nu_k+\alpha y_k,\eeqq
where  $z_k:=(d_k,y_k)$ satisfies the linearization of (\ref{eq_37})
\beqq\label{eq27}
\nabla f(x_k+d_k ,\nu_k+y_k)\approx 
\nabla f(x_k,\nu_k)+\langle \nabla f^2(x_k ,\nu_k), z_k\rangle=0.
\eeqq
Introduce the residual vector,   \beqq r_k=-(c-(x_k t)^{-1}-M^\top \nu_k).\eeqq Together with   $Mx_k=\1_{2n}$, 
 (\ref{eq27}) gives
 \beqq \label{eq36'}
\nabla^2 f(x_k,\nu_k)  z_k=
\left(\begin{array}{cc}
t^{-1}\diag(x_k)^{-2}, & -M^\top \\
-M,  & 0
\end{array}\right)
\left(\begin{array}{c}
d_k\\
y_k\end{array}\right)=-\nabla f(x_k,\nu_k)=
\left(\begin{array}{c}
r_k\\
0\end{array}\right).
\eeqq



The first part of (\ref{eq36'}) implies
 \beqq\label{eq_24}
 t^{-1}  d_k=\diag(x_k^2) (M^\top y_k+r_k).
 \eeqq
 Together with the second part  of (\ref{eq36'}), we have the normal equation for $y_k$,
 \beqq\label{normal}
 M \diag(x_k^2) (  M^\top y_k+r_k)=0,\; 
 i.e., 
y_k=-(M \diag(x_k^2) M^\top)^\dagger (M \diag(x_k^2) r_k).
 \eeqq
The well-poshness of (\ref{normal}) is  given in the appendix. 
 We can employ   Krylov subspace methods, e.g.,   matrix-free conjugate gradient methods to  solve $y_k$ from (\ref{normal}) and  then compute $d_k$ from (\ref{eq_24}).

In solving (\ref{eq36'}),  we shall avoid forming those big matrices $M$ and $\diag(x_k^{-2})$.
We demonstrate the matrix-vector product in the conjugate gradient method in solving $y_k$.
With  $y_k :=[{y^{(1)}}; {y^{(2)}}]$ and
\beqq\label{Sk}
\widetilde M_k :=
M \diag(x^2) M^\top=
\left(\begin{array}{cc}
\diag(\IT( x_k^2)\1_n), &\IT( x_k^2) \\
\IT(x_k^2)^\top,  & \diag(\IT( x_k^2)^\top \1_n)
\end{array}\right),
\eeqq
 we  implement the matrix-vector product in the conjugate gradient method,
\beqq
\widetilde M_k y_k
=
\left(\begin{array}{c}
y^{(1)}\odot (\IT(x_k^2) \1_n)+\IT(x_k^2) y^{(2)}\\
y^{(2)}\odot (\IT(x_k^2)^\top  \1_n)+\IT(x_k^2)^\top  y^{(1)} 
\end{array}\right).\label{eq_38}
\eeqq
To further enhance the convergence speed, 
we can adopt some preconditioners for the conjugate gradient method, e.g., modified Cholesky preconditioners\cite{RN3}. 
%

\begin{rem}[Rank reduction]\label{2.2}

Note that the matrix $\widetilde M_k :=M \diag(x_k^2) M^\top$ can be regarded as the \textbf{Schur complement} of the first block in the Hessian matrix in (\ref{eq36'}), after ignoring the scaling factor $t$. (This matrix   also appears in    the Hessian computation  in (\ref{eq_80}) and (\ref{eq_14''}) for  matrix balancing algorithms. )
Each $x^{(t)}$ is computed based on the Newton direction $d_k$, whose calculation is essentially   the application of a projection $P_k$.
 The calculation could be inaccurate, if the involved \textbf{Schur complement} $M\diag(x_k^2) M^\top\in \IR^{2n\times 2n}$ has serious rank deficiency due to the limitation of finite precision.
Since the null space of $M^\top$ has dimension $1$, the rank of $M\diag(x_k^2) M^\top$ is $2n-1$ for $x_k$ with all entries away from $0$ (See the appendix). When  the optimal solution $\IT(x)$ of (\ref{P0}) is a  permutation matrix,  $M\diag(x^2) M^\top$, which is the sum of $n$ rank one matrices,  has rank only $n$. Hence, as $x^{(t)}$ tends to $x$,   many  entries in $(x^{(t)})^2$ (though nonzero) will be rounded to zero in the matrix-vector-product calculation. The inaccuracy is always inevitable for $t$ sufficiently large.
To reduce numerical errors caused by the singularity, the matrix $\widetilde M_k$  should be replaced with a regularized  matrix \beqq
\widetilde M_k+\epsilon I_{2n\times 2n}
\textrm{ for some positive small $\epsilon$}.\label{M_eps}
\eeqq 
In addition, to ensure the feasibility of $x_k$, we can apply matrix balancing to project $x_k$ on $\Pi_n$.
 Another manner to alleviate the rank deficiency is that   we can employ 
 some early termination condition stated  in Prop.~\ref{early}
  to produce  an optimal solution fulfilling the KKT condition, if it is applicable. 
 
\end{rem}

 The aforementioned  log-barrier 
  interior point method
   only serves for the purpose of illustrating the overall  algorithmic  framework, and  motivating the negative-entropy based interior point methods.
Computational experiments show that primal-dual methods  can perform much better than this pure primal barrier methods on practical problems.  For instance, Mehrotra predictor-corrector method\cite{doi:10.1137/0802028} is one 
popular primal-dual method,  whose iterates follow  a path  with duality measure tending to $0$ to reach one  point fulfilling the KKT condition  in the space of  $x$,  $\nu$ and $s$\cite{Wright1997}. 

\subsection{Optimal transport  by matrix balancing}\label{SNNE_sec}
Recently,   optimal transport has been approximated  by 
 an  entropic regularized optimal transport problem~\cite{Cuturi2013}\cite{Chizat2018}.
 Using the negative entropy function $x\log x-x$,  we obtain
a regularized problem with $t>0$,
\beqq\label{eq_76'}
\min_{Mx=\1_{2n}} \left\{ \IF_t(x):=\langle c,x\rangle+t^{-1} \langle \1_{n^2}, x\odot \log x-x\rangle\right\}.
\eeqq
The strict convexity of $x\odot \log x$ implies the uniqueness of  the minimizer in $\IR_+^{n^2}$. 
The first-order optimal  condition suggests that  the optimal solution can be computed by 
matrix scaling algorithms. 
Introducing   a multiplier vector $\nu$ for the constraint,  we have
 the { problem } 
\beqq\label{eq_29'}
\min_x  \left\{\langle c,x\rangle+t^{-1} \langle \1_{n^2}, x\odot \log x-x\rangle-\langle \nu, Mx-\1_{2n}\rangle\right\}.
\eeqq
The gradient computation gives the optimal condition of  $x$,
\beqq\label{eq_69}
c+t^{-1}\log x-M^\top \nu=0, \textrm{ i.e.,
$
x=\exp(-t(c-M^\top \nu))$.}
\eeqq
The multiplier vector $\nu:=[\nu^{(1)}; \nu^{(2)}]$ in (\ref{eq_29'})   can be  determined  in matrix balancing of $\exp(-t \IT(c))$. Indeed, 
since $\IT(M^\top \nu)=\nu^{(1)} \1_n^\top +\1_n {\nu^{(2)}}^\top $, then (\ref{eq_69}) yields that $\IT(x)\in \Pi_n$ is obtained under
   proper scaling matrices,
   \beqq\label{xlogx}
\IT(x)=\IT(\exp(-t(c-M^\top \nu)))=\diag(\exp(t\nu^{(1)} ))\exp(-t \IT(c))\diag(\exp(t\nu^{(2)} ))\in \Pi_n.
\eeqq
Various Newton methods  can be employed to perform matrix balancing in (\ref{xlogx}).
 The details of   matrix balancing algorithms will be presented  in next section.

 Under large  $t$,
  the solution  in (\ref{eq_76'})
can provide a better  approximation
 to the original optimal transport in (\ref{P0}). 
 However,  
  problems with large $t$ are generally very ill-conditioned and   hard to solve. 
 To alleviate the ill-condition issue, 
 with $\eta>1$,  we  solve $\nu$ in a sequence of subproblems associated with   $t=t^{(0)}, t^{(1)},\ldots, t_{\max}$. \textcolor{black}{This method is known as   $\epsilon$-scaling heuristic\cite{Schmitzer2019} with   $1/t$ replaced with $\epsilon\to 0$. } 
  To emphasize the usage of Newton methods, we  
 call the interior point method in solving (\ref{eq_76'}) with $t\to \infty$ as  the \textbf{Sinkhorn-Newton-negative-entropy method(SNNE)}.
\begin{itemize}
\item Initialize $t=t_0$ and $\nu=\nu_{ini}$.  Repeat the following two steps until $t=t_{\max}$. 
\item Employ Newton based matrix balancing algorithms to update  $\nu$, i.e.,   $\exp(-t \IT(c-M\nu))$ is doubly stochastic.\item If $t<t_{\max}$, update $t\to t\eta$.
\end{itemize}
 The convergence of SNNE consists of two parts:  the duality gap and the slackness condition.  
 The convergence of duality gap requires   the boundedness of $\nu$, which is related to  the total support condition.  We postpone  the discussion to Theorem~\ref{thm1}. 
 Here,  we give a few words on the convergence of $s\odot x\to 0$ as $t\to \infty$.
With $s=c-M^\top \nu$, the optimal condition in (\ref{eq_69}) can be expressed as   $s=-t^{-1}\log x\ge 0$.
Fixing $\gamma'\in (0,1)$ and $\gamma''\in (1,\infty)$,
we can compute  an approximate   solution $x$ with 
\beqq
\gamma' t^{-1} (-x\odot \log x)\le s\odot x\le \gamma'' t^{-1} (-x\odot \log x)
\eeqq 
 As $t\to \infty$, we reach the KKT condition in (\ref{KKT}), \beqq
0\le x\odot s=-t^{-1} x\odot \log x\le (et)^{-1}\to 0.
\eeqq
 As $t$ gets sufficiently large, a solution satisfying  the slackness condition can  be reached  with the aid of  early termination in Prop.~\ref{early}.
Empirically,  the  convergence for large $t$ does require fast convergence and high accuracy of  matrix balancing algorithms. 

\subsubsection{ Interior point methods with total support constraints} Although an optimal solution $x$ could be  sparse,
interior point methods  require   memory storage $O(n^2)$ for $x$, which could be prohibited  
in  large-scale point-sets.  As  column generation solves large linear programming, we shall use dual variables to  reduce the memory storage by   imposing 
 (and dynamically  updating)  the \textbf{sparse support constraint} on $x$. 
  For instance,   in \cite{Schmitzer2019}  sparse support sets are introduced 
   to form  approximate problems with truncated  sparse kernels to reduce the memory storage requirement. 
 Actually,  introducing these constraints to remove those inactive components can  also  improve the  quality of solutions $x^{(t)}$.

Let $supp(x)$ be the index set of all the positive entries in $x$. 
We say that  the index set $\Sigma$ is one \textit{support} of $X=\IT(x)$, if $\Sigma$ consists of all indices of nonzero entries in $X$, i.e., $X_{i,j}=0$ holds for all $(i,j)\notin \Sigma$.    
 We say that \textit{$X$ is a solution to optimal transport with respect to  the support constraint $\Sigma$, 
 } if $\Sigma$ is a support  of 
 $X=\IT(x)$
and $x$
 is one optimal solution to   \beqq\label{F20}
 \min_{Mx=\1_{2n}} \left\{ \IF_t(x,\Sigma):=\langle c,x\rangle+t^{-1} \langle \1_{n^2}, x\odot \log x-x\rangle\right\}, \; 
  supp(\IT(x))\subset\Sigma
\}.\eeqq


 To reach one optimal transport  approximation,  we    shall  generate a sequence of \textrm{ supports }\beqq 
 \{\Sigma_1,\ldots, \Sigma_\xi, \ldots\}, \eeqq  and apply matrix balancing algorithms to get 
an   approximate solution   $X_{\xi}\in \Pi_n$  with respect to the  support $\Sigma_{\xi}$ for each $\xi\in \{1,2,\ldots\}$.
By updating $x$ and $\Sigma$ alternately, we can reach a good approximation of the optimal solution for $\IF_t(x)$ in (\ref{eq_76'}).
if  the selection rule of $\Sigma_{\xi+1}$ is given by  (\ref{sigma1}) to fulfill  two conditions:
 the  total support condition (see Definition ~\ref{def1}) and the inclusion of  the index set
\beqq
\label{Sigma_ep}
\Sigma'':=\{(i,j): s_{i,j} :=c_{i,j}-(\nu^{(1)}(i)+\nu^{(2)}(j))\le \epsilon\}\subset \Sigma_{\xi+1}.
\eeqq
Here, $\epsilon$  is
  some positive parameter 
to ensure the sparsity of the support.
 
\subsubsection{Total support condition}
 
\begin{definition}\label{def1}Let $X$ be  an $n\times n$ matrix and $\sigma$ be a permutation of $\{1,2,\ldots, n\}$. Then the sequence $\{X_{1,\sigma(1)},X_{2,\sigma(2)}, \ldots, X_{n,\sigma(n)}\}$ is a diagonal of $X$ (corresponding to $\sigma$). Then  a nonnegative square matrix $X$ is said to have support if $X$ contains one positive diagonal. Also, $X$ has  total support if $X\neq 0$ and if every positive entry of $X$ lies on a positive diagonal   \cite{Knopp1967}.  Let $\1_\Sigma$ denote the indicator matrix, whose $(i,j)$-entry is $1$ for each $(i,j)\in \Sigma$. We say that an index set $\Sigma$ satisfies total support condition, if the associated indicator matrix $\1_\Sigma$ has total support. 
 \end{definition}  
When $\1_\Sigma$ has no support, then  $\1_\Sigma$ can not \textcolor{black}{be}  scaled to a doubly stochastic matrix. Actually,   by Birkhorff theorem, 
any  doubly stochastic matrix is convex combination of  permutation matrices.  Since the support of one nonnegative matrix remains invariant under the product of positive diagonal matrices,   having total support is one necessary condition for matrix balancing.  Indeed, 
Theorem~\ref{them_xi} states that  total support is the crucial condition to ensure  the existence of a doubly stochastic matrix from  a sparse nonnegative matrix $X$. 
 \begin{theorem} \label{them_xi}\cite{Knopp1967} Let $X$ be a nonnegative squared matrix. A necessary and sufficient condition that $B=\diag(y) X \diag(z)$ is double stochastic for two positive vectors $y,z$ is that $X$ has total support. 
 \end{theorem}
To illustrate the importance of total support, consider the following example. Let  $X=\left(\begin{array}{ccc}1& 0& 0\\ 2& 3& 0\\0& 0 &4 \end{array}\right)$. 
 Since the entry $2$  is not contained in a positive diagonal,
 $X$  cannot be scaled to a doubly stochastic matrix. However, when    $X=\left(\begin{array}{ccc}1& .05& 0\\ 2& 3& 0\\0& 0 &4 \end{array}\right)$,  the entry $2$ is contained in the positive diagonal $[0.05, 2,4]$ and thus
 the matrix $X$ can be balanced. On the other hand, let $X=\left(\begin{array}{cc}1& \epsilon\\ 1& 1\end{array}\right)$ with $\epsilon>0$. Even though $X$  
   can be scaled to a doubly stochastic matrix, 
   \beqq\label{eq_41n}
\diag([1, t] )   \left(\begin{array}{cc}1& \epsilon\\ 1& 1\end{array}\right)\diag((1+t)^{-1}[1,t^{-1}])=\frac{1}{1+t}
   \left(\begin{array}{cc} 1& \epsilon t^{-1} \\ t& 1\end{array}\right)\; \textrm{ with  $t=\epsilon^{1/2}$}, 
   \eeqq
  the relative magnitude of  entries of the scaling vectors tend to $\infty$ as $\epsilon\to 0$.

We illustrate the construction of a set with total support. Let $\Sigma''$ denote the index set,
\beqq
\label{Sigma_ep''}
\Sigma'':=\{(i,j): c_{i,j}-(\nu^{(1)}(i)+\nu^{(2)}(j))\le \epsilon\}.
\eeqq
In general, the set $\Sigma''$ does not automatically meet the total support condition. 
Here is  one simple  construction of a total support set $\Sigma$ containing  the prescribed index set $\Sigma''$.  
 \begin{prop}\label{total}
 Let $\Sigma''$ be some prescribed  index set. Let 
 $\sigma$ be a permutation of $\{1,2,\ldots, n\}$ and 
 let $\Sigma':=\{(i,\sigma(i)): i=1,\ldots, n\}$. Then 
 the union set
 \beqq
 \Sigma:=\Sigma' \cup
 \Sigma''\cup \Sigma''',\; \Sigma''':= \{(\sigma^{-1}(j), \sigma(i)): (i,j)\in \Sigma''\}\label{sigma1}
 \eeqq
  has total support. 
 \end{prop}
 \begin{proof}
 For each  $(i,j)\in \Sigma''$, we shall point out one diagonal in $\Sigma$. 
 Since $\sigma$ is a permutation, then $\{(k, \sigma(k)): k=1,2,3,\ldots, n\}$ is one diagonal. Express the diagonal sequence as
$
\{ (i, \sigma(i)), (\sigma^{-1}(j),j), \widehat\Sigma\}$, i.e., $\widehat \Sigma$ is  the set consisting the remaining $n-2$ indices. 
Note that $\widehat \Sigma$ does not consist of any entries in row-$i$, row-$\sigma^{-1}(j)$, column-$j$ and column-$\sigma(i)$.
Then $\{(i,j), (\sigma^{-1}(j), \sigma(i)), \widehat \Sigma\}$ is
 a diagonal
 for this $(i,j)$. 
 \end{proof}

 \begin{rem}[The choice of $\sigma$]\label{rem2.6} 
 The set $\Sigma'''$ can be regarded as one ``reflection" of $\Sigma''$ with respect to the diagonal $\Sigma'$.
 For simplicity,  one can  consider the fixed choice: let $\sigma$ to be the identity and 
  $\Sigma$ in (\ref{sigma1}) is the index set corresponding to the positive entries of $I+\1_{\Sigma''}+\1_{\Sigma''}^\top$.
  Empirically,  we suggest that the permutation $\sigma$  should be chosen dynamically,  so that  the corresponding entries $\{ X_{i, \sigma(i)}: i=1,\ldots, n\}$ are  large entries in $X$, away from zero. 
 
 \end{rem}

 \subsubsection{Index set $\Sigma''$}
The inclusion  of $\Sigma''$ is to provide one tight approximation to $\IF_t(x)$ in (\ref{eq_76'}).
Substitute  the optimal vector   $x$ in (\ref{eq_69}) to (\ref{eq_29'}).  The Lagrange dual of (\ref{eq_76'}) is given by 
\beqq\label{eq_dual}
\max_{\nu}
\left\{\IG_t(\nu):=- t^{-1} \langle \exp(t\nu^{(1)} ), \exp(-t \IT(c)) \exp(t\nu^{(2)} ) \rangle+\langle \nu, \1_{2n}\rangle\right\}.
\eeqq
Introduce   a sparse support set $\Sigma$  as the support of $x$ and solve $x$ from the problem
\beqq\label{eq_29Sig}
\min_{M x=\1_{2n}, x\ge 0} \left\{  \IF_t(x, \Sigma) :=\langle c,x\rangle+t^{-1} \langle \1_{\Sigma}, x\odot \log x-x\rangle\right\}.
\eeqq
Introduce a multiplier vector $\nu$ for the constraint and form the Lagrangian function,
\beqq
\langle c,x\rangle+t^{-1} \langle \1_{\Sigma}, x\odot \log x-x\rangle-\langle \nu, Mx-\1_{2n}\rangle.
\eeqq
The   optimal solution  is given by 
\beqq\label{eq_49'}
x=\IT^{-1}(\1_{\Sigma})\odot \exp(-t(c-M^\top \nu)),\;\textrm{ i.e., }\;
\IT(x)=\diag(\nu^{(1)})(\1_{\Sigma}\odot \exp(-t \IT(c))) \diag(\nu^{(2)}),
\eeqq where $\nu$ is chosen to ensure  $ \IT(x)\in\Pi_n$.
Using  (\ref{eq_49'}),   we have  the Lagrange dual of (\ref{eq_29Sig}),
\beqq\label{k_bound}
\max_\nu \left\{\IG_t(\nu,\Sigma):=- t^{-1} \langle \1_\Sigma, \exp(-t \IT(c- \textcolor{black}{M^\top \nu}))  \rangle+\langle \nu, \1_{2n}\rangle\right\}.\eeqq

Let $\Pi_0$ denote the whole index set $\{(i,j): 1\le i, j\le n\}$ and   let  $\Sigma^c$ denote the complement set  of $\Sigma$. According to duality,   \beqq
\max_\nu \IG_t(x,\Pi_0)=\min_{x\in \Pi_n} \IF_t(x)=\min_{x\in \Pi_n} \IF_t(x,\Pi_0)\le \min_{x\in \Pi_n} \IF_t(x,\Sigma)=\max_\nu \IG_t(\nu,\Sigma).\eeqq
Hence,  $\max_\nu \IG_t(\nu,\Sigma)$ is one upper estimate for $\min_x \IF_t(x)$ and the gap can be estimated by
\beqq
\max_\nu \IG_t(x,\Sigma)-\max_\nu \IG_t(x,\Pi_0)\le 
\max_\nu (\IG_t(x,\Sigma)- \IG_t(x,\Pi_0))\le 
\max_\nu\{ t^{-1} \langle \1_{\Sigma^c}, \exp(-t \IT(c- \textcolor{black}{M^\top \nu}))  \rangle\}.
\eeqq
 For a tight estimate  to $\min_{x\in \Pi_0} \IF_t(x)$,  the support set $\Sigma$ should be chosen to  include    the index set $\{ (i,j): (c-\textcolor{black}{M^\top \nu})_{i,j}<\epsilon\}$ for some constant  $\epsilon>0$.

In summary, we have  the following  SNNE-sparse algorithm.
 As pointed in Theorem~\ref{them_xi},  the support set must satisfy  total support condition  to ensure the existence of scaling vectors $\nu^{(1)}$ and $\nu^{(2)}$ for matrix balancing.


 \begin{alg}[SNNE with {sparse} support]\label{SNNEalg}
   Input:  parameters $\epsilon>0$, $\xi_{\max}>0$,  $t_{\max}>0$, $\eta>1$  and the assignment matrix $c$. 
   Initialize $t=t_0$ and $\nu$.
   Generate
one initial support set   $\Sigma_1$ fulfilling the total support condition.
Repeat the following steps for $t=t_0, t_1, \ldots, t_{\max}$, so that  $\nu_\xi$ gives  a solution for $x$ in (\ref{eq_49'}).

\begin{itemize}
\item  For 
$\xi=1,2,3,\ldots, \xi_{\max}$, iterate the following two steps to get approximation solutions for $\nu, \Sigma$.
\begin{enumerate}
\item Employ Newton method based matrix balancing algorithms  in section~\ref{NE_sec}  to update  $\nu$, i.e.,   \beqq\label{SNNEs} 
\1_\Sigma\odot \exp(-t \IT(c-M\nu))\eeqq is doubly stochastic.
\item Let  $\nu_\xi=[\nu^{(1)}; \nu^{(2)}]$ and construct $\Sigma''$ by (\ref{Sigma_ep''}).
Let  $\Sigma_{\xi+1}$ be the total support  set in (\ref{sigma1}). 
\end{enumerate}
\item If $t<t_{\max}$, update $t\to t\eta$.
\end{itemize}

   \end{alg} 

 \begin{rem}[Convergence] We give a few comments on  the convergence of SNNE-sparse.  Suppose we fix the cardinality $|\Sigma_\xi|$ for each $\xi$. The sequence of $(x_\xi, \Sigma_\xi)$ is  actually constructed  to minimize $\IF_t(x,\Sigma)$ alternately, where $x_\xi$ is given by (\ref{eq_49'}). Since the function $\IF_t(x,\Sigma)$ is bounded below, the sequence will eventually  stop at some $\xi$. Indeed,   the  optimality of  $x_\xi$ is ensured if $\IT(x_\xi)$ in (\ref{eq_49'}) is balanced by some $\nu_\xi$.   From (\ref{eq_29Sig}), the optimality of $\Sigma_\xi$ is ensured,  if $\Sigma_\xi$ contains  the index set associated with the smallest entries of $x\log x-x$, equivalently,  the smallest entries of $c-M^\top \nu$. (Thanks to  the monotonic decrease of $x\log x-x$ for $x\in [0,1]$, $\Sigma_\xi$ actually contains the index set   associated with the largest entries of $x$.) Here,  we ignore  the total support requirement on each $\Sigma_{\xi}$.
 
\end{rem}

\subsubsection{Error estimate of SNNE-sparse}

Error estimates of SNNE-sparse can be examined by duality measure $\langle c, x\rangle-\langle \nu, \1_{2n}\rangle$.
  The following result indicates how   the duality measure    under $t\to \infty$ can be improved  by
  the accuracy of  matrix balancing on $\IT(x)$ and
  the boundedness assumption  on $\nu$. 
  \begin{theorem}\label{thm1} Consider an approximate optimal solution $x$ of (\ref{F20}), constructed from matrix balancing 
 \beqq 
 x=\IT^{-1}(\1_{\Sigma})\odot \exp(-t(c-M^\top \nu))
 \eeqq for some dual vector $\nu$. Let $\cN$ be the null space  of $\diag(\IT^{-1}(\1_\Sigma)) M^\top$ and
 let $P$ be the projection with kernel $\cN$.
 Suppose that 
 $\|P \nu\|_2 \le \delta$ holds for some positive constant $\delta >0$ and 
  $\IT(x)$ is nearly doubly stochastic, i.e., $\| Mx-\1_{2n}\|_2\le \epsilon_{MB}$ for some $\epsilon_{MB}>0$. 
  Then we have error estimates,
  \beqq
|\langle c, x\rangle-\langle \1_{2n}, \nu\rangle  |\le \epsilon\delta + (et)^{-1}  |\Sigma|,
\eeqq 
where
$|\Sigma|$ is the cardinality of the index set $\Sigma$.

\end{theorem}
\begin{proof}
Since $\Sigma$ has total support, then $\1_\Sigma$ can be balanced by some scaling vectors $\zeta^{(1)}, \zeta^{(2)}\in \IR^n$, i.e.,  
$\1_\Sigma\odot (\zeta^{(1)} {\zeta^{(2)}}^\top )$ is doubly stochastic, 
$ M \IT^{-1}(\1_\Sigma\odot (\zeta^{(1)} {\zeta^{(2)}}^\top ))=\1_{2n}.$
Hence, 
\beqq
M (\IT^{-1}(\1_\Sigma)\odot x)-\1_{2n}=M (\IT^{-1}(\1_\Sigma)\odot (x-\IT^{-1} (\zeta^{(1)} {\zeta^{(2)}}^\top ) ) )
\eeqq
lies in the range of $M\diag(\IT^{-1}(\1_\Sigma))$,  and also lies in the range of $P$,  which implies
 $P(Mx-\1_{2n})=Mx-\1_{2n}$ from the definition of $P$. 
Computation shows
\begin{eqnarray}
&& |c^\top x-\nu^\top \1_{2n}|
=|c^\top x-\nu^\top Mx+\nu^\top (Mx-\1_{2n})|\\
&\le& |(c-M^\top \nu)^\top x|+\|P\nu\|_2 \|M x-\1_{2n}\|_2\\
&= &  -t^{-1} \langle \IT^{-1}( \1_{\Sigma}), x\odot \log x\rangle+ \|P \nu\|_2 \|M x-\1_{2n}\|_2\\
&\le &  (et)^{-1}|\Sigma|  +\delta \epsilon,
\end{eqnarray}
where  the last inequality is derived from $x\log x\ge -e^{-1}$.

%
 \end{proof}
This result is consistent with   empirical studies, where  solving a negative entropy regularized  optimal transport could be a  challenging problem, if the norm of the associated dual  vector  is  large.
  Later, we shall prove that   the required norm  bound can be  obtained  under  the total support condition. See Prop.~\ref{prop5.5} and Remark~\ref{3.6}.
%
%
%
%
%
%
%
%

%
%
%

   \begin{rem}[Parameters in  SNNE-sparse]\label{rem2.10}
 It could be  not easy to choose a proper parameter $\epsilon>0$ to meet the
desired sparsity. One practicable  manner is to select a parameter $k>0$ and let $\Sigma$ consist of those $(i,j)$ corresponding to (at most) $k$ smallest entries $(c-\textcolor{black}{M^\top \nu})_{i,j}$ for each row and each column. In this manner, $\Sigma_{\xi}$ consists of at most $(2k+1)n$ entries. 
 In section~\ref{sec_4.3},
we shall present  numerical experiments
 under a proper value $k$ to demonstrate the effectiveness.
\end{rem}

\section{Matrix balancing}
 Let $A$ denote  a positive  matrix  in $ \IR^{n\times n}$. 
Matrix balancing~\cite{Sinkhorn} aims to find a pair of positive \textit{scaling}  vectors $\{\zeta^{(1)}, \zeta^{(2)}\}$, so that the matrix balancing projection 
\[ A':=\diag(\zeta^{(1)})A\diag(\zeta^{(2)})\] is doubly stochastic, i.e.,  
\begin{eqnarray}
&&A' \1_n=\diag(\zeta^{(1)})A  \zeta^{(2)}=\diag(\zeta^{(1)})A\diag( \zeta^{(2)})\1_n=\1_n,\label{eq_75} \\ 
&&{A'}^\top  \1_n=\diag(  \zeta^{(2)} )A^\top   \zeta^{(1)} =(\diag(\zeta^{(1)})A\diag(\zeta^{(2)}))^\top \1_n=\1_n.\label{eq_76}
\end{eqnarray}  The existence of $\{\zeta^{(1)},\zeta^{(2)}\}$ is proved in~\cite{Sinkhorn},\cite{Knopp1967} for any positive matrix and any nonnegative matrix with total support, respectively.
Matrix scaling methods and its various applications in scientific computing, statistics and engineering can be found in the extensive survey~\cite{Idel2016} and the references therein. In general, the prescribed row sums and column sums do not have to be restricted to $\1_n$. See~\cite{Kalantari2008} and\cite{AllenZhu2017}.  
In the  section, we shall list a few  matrix scaling algorithms and their variants. 

\subsection{Sinkhorn-Knopp balancing(SK) and Knight-Ruiz(KR) method}
For the conditions in  (\ref{eq_75},\ref{eq_76}), the Sinkhorn-Knopp balancing(SK) (also known as the RAS  or biproportional problem\cite{Bacharach1970}) is
one well-known method  to carry out  matrix balancing on $A$,  consisting of iterates $\{(\zeta^{(1)}_k, \zeta^{(2)}_k): k=1,2,3,\ldots\}$,
\beqq\label{SKb}
\zeta^{(2)}_{k+1}=(A^\top \zeta^{(1)}_k)^{-1},\; 
\zeta^{(1)}_{k+1}=(A \zeta^{(2)}_k)^{-1}.
\eeqq
We can express (\ref{SKb}) in a symmetric manner\cite{knight:261}. Form  one symmetric matrix $\widetilde A$ from $A$, \beqq\label{eq_12}
\widetilde A:=\left(
\begin{array}{cc}
A_{1,1}& A_{1,2} \\
A_{2,1} & A_{2,2}\
\end{array}\right)=\left(
\begin{array}{cc}
0& A \\
A^\top & 0\
\end{array}\right).
\eeqq 
Let $\zeta_k:=[ \zeta^{(1)}_k; \zeta^{(2)}_k]$ be a sequence of the scaling vectors. When \beqq\label{SKbi}
\zeta^{(2)}_1=(A^\top \zeta^{(1)}_1)^{-1},\eeqq the SK algorithm in (\ref{eq_12}) can be expressed in a compact form,
\[
\zeta_{k+1}=(\widetilde A\zeta_k)^{-1},
\]
whose limit $\zeta =\lim_{k\to \infty} \zeta_k$ is actually a root  of \beqq\label{eq_8}
\mbg (\zeta):=\zeta\odot ( \widetilde  A\zeta)-\1_{2n}=0.
\eeqq
\begin{rem}[$(\rho^{(1)},\rho^{(2)})$-balancing] In this paper,  we focus on the application of point-set matching problems and thus consider the matrix balancing with $(\1_n, \1_n)$-balancing, i.e., the row sum and the column sum both $1$. In literatures, e.g., section 3 in \cite{Idel2016}, SK algorithms can be applied to reach a matrix with  row sum $\rho^{(1)}$ and  column sum $\rho^{(2)}$, where $(\rho^{(1)},\rho^{(2)})$ is not necessarily restricted to $(\1_n,\1_n)$. 
\end{rem}
 To solve the roots of $\mbg(\zeta)=0$, Knight and Ruiz \cite{Knight2012} proposed one Newton   method,
 \begin{eqnarray}\label{eq_7'}
\zeta_{k+1}&=&\zeta_k-(\diag(\zeta_k) \widetilde  A+\diag( \widetilde  A\zeta_k))^{\dagger} (\zeta_k\odot ( \widetilde A \zeta_k)-\1_{2n})\\
&=&\zeta_k\odot \left\{\1_{2n}-\left(B_k+\diag(B_k \1_{2n})\right)^{\dagger} (B_k \1_{2n}-\1_{2n})\right\}
\end{eqnarray}
 to alleviate   slow convergence of SK, where $B_k:=
 \diag(\zeta_k) \widetilde  A\diag(\zeta_k)
 $ is used. 
Compared with the SK algorithm, the Newton approach exhibits fast convergence. However, as mentioned in \cite{Knight2012},  the global convergence  property of (\ref{eq_7'}) is theoretically unclear. 

\subsection{ Negative entropy(NE) based matrix balancing}\label{NE_sec}
We describe one algorithm proposed in \cite{Cohen2017,Brauer2017}, which   implements Newton's method for matrix balancing in  (\ref{xlogx}) or in (\ref{eq_49'}).  To simplify the notation,  consider 
  \beqq\label{defA} A=\exp(- t\IT(c))\odot \1_\Sigma \in \IR^{n\times n},
\textrm{ with } t=1 
\eeqq
where $\Sigma$ is the support set used in SNNE-sparse.
Introduce the symmetric $2n\times 2n$-matrix $\widetilde A$  as in (\ref{eq_12}).
Write   the scaling vector $\exp(\nu)$ of $\widetilde A$  with  $\nu:=[\nu^{(1)};  \nu^{(2)}]$ and $\nu^{(1)}\in \IR^n$ and $\nu^{(2)}\in \IR^n$. 
Matrix balancing on $A$ can be solved by the convex optimization (i.e., \textcolor{black}{ the  problem} in (\ref{eq_dual})),
\beqq\label{eq_58}
\min_{\nu \in \IR^{2n}} \left\{\mbf (\nu )=\frac{1}{2}\langle \exp(\nu ),\widetilde A \exp(\nu )\rangle-\langle \1_{2n},\nu \rangle\right\}.
\eeqq Indeed, reformulate 
 (\ref{eq_58}) as follows:
\beqq
\mbf (\nu )=\langle \exp(\nu^{(1)}), (\exp(- \IT(c)) \odot \1_\Sigma) \exp(\nu^{(2)})\rangle-\langle \1_{n}, \nu^{(1)}\rangle-
\langle \1_{n}, \nu^{(2)}\rangle.
\eeqq
For simplicity, let
 $\mbB(\nu)$ denote the scaled  matrix of $A$, 
  \beqq
\mbB(\nu):=\exp(-\IT(c-M^\top \nu))\odot \1_\Sigma=\diag(\exp(\nu^{(1)}))(A \diag(\exp(\nu^{(2)})),\textrm{ where  } \IT(M^\top \nu)=
 \nu^{(1)} \1_n^\top+\1_n {\nu^{(2)}}^\top.\eeqq
We can express $\mbf$ as 
\beqq
\mbf(\nu)= \langle \1_n, \mbB(\nu) \1_n \rangle-\langle \1_{2n}, \nu\rangle.
\eeqq
 First,  a scaling vector  $\nu$ with $\nabla \mbf(\nu)=0$ yields  the double stochastic matrix $\mbB(\nu)$.
  Indeed,    \begin{eqnarray}
&&\nabla \mbf(\nu)= M (\exp(-(\IT(c -M^\top \nu) )  ) \odot \1_\Sigma)\1_n -\1_{2n}
=
\left(\begin{array}{c}
\exp(\nu^{(1)})\odot (A\exp(\nu^{(2)}))\\
\exp(\nu^{(2)})\odot (A^\top \exp(\nu^{(1)}))\\
\end{array}\right)-\1_{2n}\\&=&\left(\begin{array}{c}
\mbB(\nu)-I_n\\
(\mbB(\nu)-I_n)^\top 
\end{array}\right)\1_n.\label{gradf}
\end{eqnarray}
Second, the Hessian computation  verifies 
 the convexity of $\mbf$. 
Computation shows 
\begin{eqnarray}
\nabla^2 \mbf(\nu)&=& 
\left(\begin{array}{cc}
\diag(
(\exp(- \IT(c-M^\top \nu) ) \odot \1_\Sigma)\1_n) & \exp(- \IT(c-M^\top \nu)) \odot \1_\Sigma\\
(\exp(- \IT(c-M^\top \nu))\odot \1_\Sigma)^\top  &
 \diag( (\exp(- \IT(c-M^\top \nu) \odot \1_\Sigma) )^\top \1_n) \end{array}\right)\\
&=&\left(\begin{array}{cc}
\diag(\mbB(\nu)\1_n) & \mbB(\nu)\\
\mbB(\nu)^\top  & \diag(\mbB(\nu)^\top \1_n) \end{array}\right).\label{eq_80}
\end{eqnarray}
The following Newton's method, called Negative entropy  method(NE),
employs  step size   given by backtracking line search to compute 
a minimizer of the problem in (\ref{eq_58}), i.e., 
\beqq\label{eq_79}
\nu_{k+1}=\nu_k-\alpha (\nabla^2 \mbf(\nu_k))^{\dagger}\nabla \mbf(\nu_k).
\eeqq
Convergence arguments are standard. See  section 9.5.3~\cite{Boyd}. 
The following shows the consistency analysis.
\begin{prop}\label{Prop 3.2} Suppose the matrix $A$ in (\ref{defA}) is  nonnegative and  has  support.  Then the system 
\beqq \label{eq_cons}\nabla^2 \mbf(\nu_k) w=-\nabla \mbf(\nu_k)\eeqq is consistent for some vector $w\in \IR^{2n}$. \textcolor{black}{In addition,  for nonzero $\nabla \mbf(\nu_k)$, let $u=-(\nabla^2 \mbf(\nu_k))^\dagger \nabla \mbf(\nu_k)$. Then we have the squared \textbf{Newton decrement}
\beqq\label{eq_60'}
\langle u, \nabla^2 \mbf(\nu_k)u \rangle=
\langle \nabla \mbf(\nu_k), (\nabla^2 \mbf(\nu_k))^\dagger \nabla \mbf(\nu_k) \rangle>0.
\eeqq
 }
\end{prop}
\begin{proof}
For each vector $w=[w^{(1)}; w^{(2)}]\in \IR^{2n}$ with $w^{(1)}\in \IR^n$, $w^{(2)}\in \IR^n$, the Hessian $\nabla^2 \mbf(\nu)$ is symmetric diagonally dominant\cite{Cohen2017, AllenZhu2017}, thus
the convexity of $\mbf$ is verified from \beqq\label{eq_78}
\langle w, \nabla^2 \mbf(\nu) w\rangle
=\sum_{i=1}^n\sum_{j=1}^n A_{i,j} e^{\nu^{(1)}_i} e^{\nu^{(2)}_j}(w^{(1)}_i+w^{(2)}_j)^2\ge 0,
\eeqq
 For each  vector  $w$ in the null space of $\nabla^2 \mbf$, 
 from (\ref{eq_78}), $w$ satisfies
 $\langle w, \nabla^2 \mbf(\nu) w\rangle=0$, which  implies \beqq\label{eq_80'}
w^{(1)}_i+w^{(2)}_j=0 \textrm{ for all $A_{i,j}>0$.}
\eeqq Since $A$ has support, then  $\{A_{i,\sigma(i)}: i=1,2,\ldots, n\}$ are all positive for some permutation $\sigma$. 
Since $A_{i, \sigma(i)}>0$, then any vector $w$ in the null space of $\nabla^2 \mbf$ satisfies $w_i^{(1)}=-w_{\sigma(i)}^{(2)}$ and has the form 
\beqq
w:=[w^{(1)}; w^{(2)}]=[w_1, w_2, \ldots, w_n, -w_{\sigma^{-1}(1)},\ldots, -w_{\sigma^{-1}(n)}]^\top.
\eeqq
Clearly,   $\langle w^{(1)}, \1_{n}\rangle+\langle w^{(2)}, \1_{n}\rangle=0$ holds. Thus, we have the orthogonality between $-\nabla \mbf(\nu_k)$ and the null space of $\nabla^2 \mbf(\nu)$. Indeed,  \begin{eqnarray}
&&\langle w, \nabla \mbf(\nu_k)\rangle={w^{(1)}}^\top (\mbB(\nu_k)-I_n) \1_n+ {w^{(2)}}^\top (\mbB(\nu_k)-I_n)^\top \1_n
\\
&=&\sum_{i=1}^n \sum_{j=1}^n (w_i^{(1)}+w_j^{(2)})
 A_{i, j}e^{\nu^{(1)}_i} e^{\nu^{(2)}_j}=0,
\end{eqnarray} where  we used
   (\ref{eq_80'}). 
Hence, $-\nabla \mbf(\nu_k)$ lies in the range of $\nabla^2 \mbf$, which verifies that  the system in (\ref{eq_cons}) is consistent. 
Finally, we obtain  (\ref{eq_60'}) according to the positive semi-definite property in  (\ref{eq_78}) and the following observation. Since  $\nabla \mbf(\nu_k)$ is orthogonal to the null space of $\nabla^2 \mbf(\nu_k)$, then $\nabla \mbf(\nu_k)$ is orthogonal to the null space of $(\nabla^2 \mbf(\nu_k))^\dagger$.
\end{proof}
Since
$\nabla^2 \mbf(\nu_k) (\nu_{k+1}-\nu_k)=-\nabla \mbf(\nu_k)$ is consistent,
the  Newton iterations in (\ref{eq_79})
 can be employed  to find $\nu$ with  $\nabla\mbf (\nu)=0$, e.g., 
  the conjugate gradient method\cite{Brauer2017}. \textcolor{black}{
 Note that  the squared Newton decrement in (\ref{eq_60'}) can be interpreted as the directional derivative of $\mbf$ in the direction of $u$,  
 \beqq
 -\langle u, \nabla^2 \mbf(\nu_k)u \rangle=\langle\nabla \mbf(\nu_k), u \rangle=\frac{d}{d\alpha} \mbf(\nu_k+\alpha u )|_{\alpha=0}.
 \eeqq
Thanks to (\ref{eq_60'}),  when $\nabla\mbf (\nu_k)\neq 0$,  the step size $\alpha>0$ can be  chosen properly to decrease the objective $\mbf$.}

\begin{rem}Consider   the application in SNNE, i.e., the balancing  in  (\ref{xlogx}).  Note that the Hessian $\nabla^2 \mbf$ is exactly the Schur complement matrix $\widetilde M$ described in (\ref{Sk}). When  
 $\IT(x)=\exp(- t\IT(c-M^\top \nu))$ heads to an optimal permutation with $t\to \infty$,
 the Hessian matrix will easily undergo a rank-reduction process.
  Hence, using a regularized Hessian matrix  as in (\ref{M_eps}) is suggested in empirical algorithms for (\ref{eq_79}). 
\end{rem}

\subsection{ Logarithmic barrier functions(LB) based matrix balancing}\label{LB_sec}
We provide another Newton method, called Logarithmic barrier (LB) based matrix balancing,
 to compute scaling vectors of matrix balancing.  The LB iterations will be stated in (\ref{eq_14''}). The introduction can shed light on convergence of  Knight-Ruiz algorithm. 
Consider a nonnegative matrix $A$.  Define $\widetilde  A$ as in (\ref{eq_12}). Consider  the  minimization of $\mbg$, \beqq\label{eq_10}
\min_{\zeta>0} \{\mbg (\zeta)=\frac{1}{2} \zeta^\top  \widetilde  A\zeta-\1_{2n}^\top  \log \zeta\}.
\eeqq
The  objective function in (\ref{eq_10}) is identical to the function in (\ref{eq_58}), except for $\nu$ replaced with $\log \zeta$. 
In~\cite{Marshall1968}, the function $\mbg$ is employed to show the existence of matrix-scaling on a fully indecomposable matrix. In \cite{Khachiyan1992}, authors   proposed  one path-following  Newton algorithm, minimizing a sequence of 
  sub-problems to scale a symmetric positive semi-definite matrix $\widetilde A$, so that    convergence requirement  of   Newton iterates can be met in each sub-problem. 
  Here, we propose  a  modified Newton method for the computation of matrix balancing for one positive matrix $A$. 


Compute 
  the gradient and  the Hessian of $\mbg$,
 \beqq\label{eq_9} \nabla \mbg= \widetilde  A\zeta-\zeta^{-1},\;
\nabla^2 \mbg(\zeta)=\widetilde  A+\diag(\zeta^{-2}), \eeqq
 respectively. 
 First, from (\ref{eq_9}),
 the Sinkhorn-Knopp balancing  is 
 the coordinate descent iteration of $\mbg(\zeta)$ with $\zeta=[\zeta^{(1)}; \zeta^{(2)}]$,
 \beqq
\zeta_{k+1}^{(1)}\leftarrow arg\min_{\zeta^{(1)}} \mbg([\zeta^{(1)}; \zeta^{(2)}_k]),\; 
\zeta_{k+1}^{(2)}\leftarrow arg\min_{\zeta^{(2)}} \mbg([\zeta_{k+1}^{(1)}; \zeta^{(2)}]).
 \eeqq
Thus, SK balancing   decreases the objective $\mbg$ in (\ref{eq_10}). 
 Second,  
suppose a minimizer $\zeta$ is an interior point in $\IR^{2n}_+$. 
Clearly,  $\zeta$  is a root to (\ref{eq_8}), i.e., $\widetilde  A\zeta=\zeta^{-1}$. 
Write $\zeta=\exp(\nu)$  component-wise with some vector $\nu$. From (\ref{eq_78}), $\mbg(\exp(\nu))=\mbf(\nu)$ is convex in $\nu$ and a local minimizer of $\mbg$ is actually the global minimizer of $\mbg$.   
Let us 
employ one damped Newton iteration to reach the global minimizer, where step size $\alpha_k$ is selected to minimize $\mbg(\zeta_k-\alpha_k (\nabla^2 \mbg(\zeta_k))^{-1}\nabla \mbg(\zeta_k))$ in (\ref{eq_10}), for $k=1,2,3,\ldots$,
 \begin{eqnarray}
\zeta_{k+1}&=&\zeta_k-\alpha_k (\nabla^2 \mbg(\zeta_k))^{-1}\nabla \mbg(\zeta_k)=\zeta_k- \alpha_k (\widetilde A+\diag(\zeta_k^{-2}))^{-1}(\widetilde A \zeta_k-\zeta_k^{-1})
 \\
&=&\zeta_k- \alpha_k\diag(\zeta_k) \left(\diag(\zeta_k)(\widetilde A+\diag(\zeta_k^{-2})) \diag(\zeta_k) \right)^{-1} \left(\zeta_k\odot (\widetilde A\zeta_k-\zeta_k^{-1})\right)
 \\
 &=&\zeta_k- \alpha_k \zeta_k\odot \left( I_{2n}+B_k  \right)^{-1} (B_k \1_{2n}- \1_{2n}),\label{eq_13}
 \end{eqnarray}
with \beqq
\label{eq_A} 
 B_k:=\diag(\zeta_k) \widetilde  A\diag(\zeta_k).\eeqq 
Since the matrix $ I_{2n}+B_k$ in (\ref{eq_13}) is not necessarily positive definite, the iteration in (\ref{eq_13}) is not globally convergent.  Instead, 
consider
a  modified  Newton \textcolor{black}{ iteration} (called
LB matrix balancing scheme
)
 \beqq\label{eq_14''}
\zeta_{k+1}=\zeta_k\odot \{ \1_{2n}- \alpha_k ( C_k +B_k  )^{\dagger} ( B_k-I_{2n}) \1_{2n}\},
 \eeqq
  where $I_{2n}$ in (\ref{eq_13}) is replaced with 
the positive diagonal matrix, 
\beqq\label{eq_c} C_k= \diag(\1_{2n}\odot (B_k \1_{2n})).\eeqq

\begin{rem}[Safeguard parameter $\epsilon_+$]\label{safe}

We  implement (\ref{eq_14''}) as follows. 
For each $k$,
compute $B_k$ and $C_k$ from (\ref{eq_A}, \ref{eq_c}),
and
 $u_k:=-
 \zeta_k \odot ( C_k + B_k  )^{\dagger} ( B_k-I_{2n}) \1_{2n}$.
 Use conjugate gradient to solve \beqq
 y_k:=( C_k +B_k )^\dagger ( B_k-I_{2n}) \1_{2n}\eeqq from the consistent system, 
\beqq\label{eq_14'} ( C_k +B_k )y_k=   ( B_k-I_{2n}) \1_{2n}.
\eeqq
The step size $\alpha_k$  is chosen to  ensure the decrease of $\mbg$ and
  $ \zeta_{k+1}=\zeta_k\odot (1-\alpha_k y_k)>0$.
For $\zeta_{k+1}>0$, we introduce a safeguard parameter $\epsilon_+\in (0,1)$ and $\alpha$ is chosen within $(0, y_{\max}^{-1}(1-\epsilon_+)]$, where $y_{\max}$ is the largest positive entry of $y_k$. Indeed, 
 $\zeta_{k+1}=\zeta_k+\alpha u_k=\zeta_k\odot (1-\alpha y_k)\ge  \zeta_k \epsilon_+>0$.

\end{rem}

In the following, we shall discuss the wellposeness of LB and show  the step size  of  LB tending to $1$ near an optimal solution.

\subsubsection{Well-definedness   of LB in (\ref{eq_14''})}
The following proposition shows the well-definedness of $(C_k+B_k)^\dagger (B_k-I_{2n})\1_{2n}$ in (\ref{eq_14''}). 
Also, we  calculate 
 the directional derivative of $\mbg$ in the direction of \beqq\label{eq_88}
 u_k:={-}
 \zeta_k\odot ( C_k +B_k  )^{\dagger} ( B_k-I_{2n}) \1_{2n}
 \eeqq with  $\|u_k\|>0$,
 which sheds some light on the convergence of this Newton method, 
 \begin{eqnarray}
&&\frac{d}{d\alpha}\mbg(\zeta_k+\alpha u_k)|_{\alpha=0}=\langle \nabla \mbg(\zeta_k),u_k\rangle
\\
&=&-\langle ( B_k-I_{2n} )\1_{2n}, ( C_k +B_k  )^{\dagger} ( B_k-I_{2n}) \1_{2n}
\rangle<0.\label{eq_97}
\end{eqnarray}
In the following,  we shall verify the calculation in (\ref{eq_97}).
We introduce $H(\zeta)$ in (\ref{Hdef}) to investigate the null space of $B_k+C_k$. Note that $H(\zeta_{k})=C_k+B_k$.

\begin{prop}\label{prop3.3} Consider   one matrix  $A\in \IR^{n\times n}$, which   is nonnegative and has support. Let $\widetilde A$ be given in (\ref{eq_12}) and $B=\diag(\zeta) \widetilde A \diag(\zeta)$. Let  $C=\diag( B \1_{2n})$. 
Then $C +B$ is symmetric and  positive semi-definite and
the system \beqq
(C +B  ) y =(
B-I_{2n}) \1_{2n}\eeqq
 is  consistent. 
 In addition,  introduce the null space of $C+B$,
 \beqq\label{cN}
\cN:=\{w=[w^{(1)}; w^{(2)} ]:  w_i^{(1)}+w_j^{(2)}=0, \; \forall  (i,j) \textrm{ with }A_{i,j}>0\}.
\eeqq
  For any positive vector $\zeta\in \IR^{2n}$ and  for any null vector  $w\in \cN$, the function  $\mbg$ takes a constant value, as $\zeta\to \zeta\odot \exp(w)$, i.e.,
 \beqq
\mbg(\zeta\odot \exp(w))=\mbg(\zeta).
\eeqq
Introduce \beqq \label{Hdef}
H(\zeta)=\diag(\zeta\odot  (\widetilde A \zeta))+\diag(\zeta) \widetilde A \diag(\zeta).\eeqq 
Then $\cN$ is
 the null space of $H(\zeta)$ for any positive vector $\zeta$.

 \end{prop}
\begin{proof}
By Gershgorin circle theorem, the symmetric matrix  $C +B \succeq 0$ is diagonally dominant and thus is
a positive semi-definite  matrix. Actually, for each vector $w=[w^{(1)}; w^{(2)}]\in \IR^{2n}$,
\beqq
\langle w, (C+B) w\rangle=
\sum_{i=1}^n \sum_{j=1}^n A_{i,j}\zeta_i^{(1)}\zeta_j^{(2)} (w^{(1)}_i+w^{(2)}_j)^2\ge 0.
\eeqq
Hence,  each null vector $w$ of $C +B$ satisfies  \beqq \label{eq_78'} w^{(1)}_i+w^{(2)}_j=0 \textrm{ for all $A_{i,j}>0$, }\eeqq
which justifies (\ref{cN}).
Since $A$ has support, then for some permutation $\sigma$, we have
$A_{i,\sigma(i)}>0$ for $i=1,\ldots, n$. 
Hence, (\ref{eq_78'}) implies \beqq\label{eq_79'}
\sum_{i=1}^n  w_i^{(1)}+\sum_{j=1}^n w_j^{(2)}
=\sum_{i=1}^n  w_i^{(1)}+\sum_{i=1}^n w_{\sigma(i)}^{(2)}
=0.
\eeqq 
 Next, we show that  $(
B-I_{2n}) \1_{2n}$ lies in the range of $(C+B)$. Indeed, for each null vector $w$,  using (\ref{eq_78'}) and (\ref{eq_79'}), we have $(
B-I_{2n}) \1_{2n}$ is orthogonal to the null space of $(C+B)^\top=C+B$, i.e., 
\beqq
\langle w, 
(
B-I_{2n}) \1_{2n}
 \rangle=\sum_{i=1}^n \sum_{j=1}^n 
 (w_i^{(1)}+w_j^{(2)}) \zeta_i^{(1)} \zeta_j^{(2)} A_{i,j}-(\sum_{i=1}^n  w_i^{(1)}+\sum_{j=1}^n w_j^{(2)})=0.
\eeqq 
The  above orthogonality arguments  also implies  (\ref{eq_97}). Again from 
(\ref{eq_78'}) and (\ref{eq_79'}), we have
 \begin{eqnarray}
&& \mbg(\zeta\odot \exp(w))=
 \sum_{i=1}^n \sum_{j=1}^n  A_{i,j} 
\zeta^{(1)}_i \zeta^{(2)}_j  \exp( w^{(1)}_i+w^{(2)}_j )-\langle \1_{2n}, \log \zeta\rangle 
-\langle  \1_{2n}, w \rangle\\
&=& 
 \sum_{i=1}^n \sum_{j=1}^n  A_{i,j} 
\zeta^{(1)}_i \zeta^{(2)}_j  
-\langle \1_{2n}, \log \zeta\rangle 
=\mbg(\zeta).
 \end{eqnarray}
 Finally,  observe  that 
 $
 \langle  w, H(\zeta) w \rangle=\sum_{i=1}^n \sum_{j=1}^n A_{i,j} \zeta^{(1)}_i \zeta^{(2)}_j (w^{(1)}_i+w^{(2)}_j)^2=0
 $
 if and only if  $w\in \cN$. Thus, $\cN$ is  the null space of $H(\zeta)$ for any positive vector $\zeta$.
\end{proof}

%
%
%

\subsubsection{Relation between KR and LB}

First,  we make one observation. 
\begin{rem}[KR method is a special case with $\alpha_k=1$]
Note that the LB method in (\ref{eq_14''}) with $\alpha_k=1$ coincides with the algorithm proposed by Knight and Ruiz in (\ref{eq_7'}).
As $k$ increases, the objective values $\mbg(\zeta_k)$  decrease monotonically.  As $B_k$ tends to \textcolor{black}{be} a doubly stochastic matrix, we have $C_k=\diag(B_k \1_{2n})\to I_{2n}$ and the LB method  reduces to Newton's method in  (\ref{eq_7'}), i.e., KR method.
 \end{rem}
In the following,  we  demonstrate that the step size $\alpha_k$ of LB is $1$ for sufficiently large $k$. 
To proceed, 
we start with  some boundedness   related to  the sequence $\{ \zeta_k: \mbg(\zeta_k)\le c_0, \; k=1,2,\ldots\}$ under  total support assumption on $A$. 
 For notation simplicity, we drop the subscript $k$. 
\begin{prop} \label{prop5.5}Suppose $A\in \IR^{n\times n}$ has total support. Let $\Sigma:=\{(i,j): A_{i,j}>0\}$. 
Let $\delta$ be a positive lower bound for $\{A_{i,j}: (i,j)\in \Sigma\}$.   Fix some $c_0\in \IR$.
Let  $\zeta=[\zeta^{(1)}, \zeta^{(2)}]$ be a positive vector in the $c_0$-sublevel set of $\mbg$,  i.e., $\mbg(\zeta)\le c_0$.
 Then 
$\{ \zeta^{(1)}_i \zeta^{(2)}_j: (i,j)\in \Sigma\}$
 are bounded below by \beqq\label{eq_82'}
\exp(-c_0+(n-1)(1+\log \delta))
\eeqq
and bounded above by 
\beqq\label{eq_82''}
\max(\delta^{-1}  (c_0-(n-1)(1+\log \delta)),1).
\eeqq
In particular, for
any  $\zeta$  with $\mbg(\zeta)\le c_0$,
$\|(\zeta^{(1)} {\zeta^{(2)} }^\top)  \odot \1_\Sigma\|$ is bounded above by  some constant only 
depending on $c_0$ and $\delta$.
\end{prop}
\begin{proof}
Fix one entry $A_{i_1, j_1}>0$. 
By assumption,  $A$ has total support,
and thus   $(i_1, j_1)$ lies on  some diagonal $\{(i, \sigma(i)):i\in\{1,2,\ldots, n\} \}$. 
Then 
\beqq\label{eq_84'}
\sum_{i=1}^n\{ A_{i, \sigma(i)} \zeta^{(1)}_i \zeta^{(2)}_{\sigma(i)}-\log (\zeta^{(1)}_i \zeta^{(2)}_{\sigma(i)})\} \le 
\mbg(\zeta)=
{ \langle \zeta^{(1)}, A\zeta^{(2)}\rangle } -\langle \1_{2n}, \log \zeta\rangle\le c_0.
\eeqq
By convexity, the following inequality holds  for each $a>0$, 
\beqq\label{eq_83'}
\min_{x\ge 0} (ax -\log x)\ge 1+\log a.
\eeqq
Applying  (\ref{eq_83'}) to the right hand side of (\ref{eq_84'}) for those   $i\neq i_1$, 
we have
\beqq\label{eq_111}
\sum_{i=1}^n\{ A_{i, \sigma(i)} \zeta^{(1)}_i \zeta^{(2)}_{\sigma(i)}-\log (\zeta^{(1)}_i \zeta^{(2)}_{\sigma(i)})\}
\ge A_{i_1, j_1}\zeta^{(1)}_{i_1} \zeta^{(2)}_{j_1}-\log (\zeta^{(1)}_{i_1} \zeta^{(2)}_{j_1})+(n-1)(1+\log \delta).
\eeqq
Together with (\ref{eq_84'}), dropping the positive term $A_{i_1, j_1}\zeta^{(1)}_{i_1} \zeta^{(2)}_{j_1}$ in (\ref{eq_111}), we have (\ref{eq_82'}).
Likewise, for an upper bound, when $\zeta^{(1)}_{i_1} \zeta^{(2)}_{j_1}\ge 1$, we can 
  drop $-\log (\zeta^{(1)}_{i_1} \zeta^{(2)}_{j_1})$ in (\ref{eq_111}),
which yields  the upper bound in (\ref{eq_82''}).

\end{proof}

\begin{rem}\label{3.6} Let $\1_\Sigma:=(A>0)$.
When $A$ has total support, then $\zeta_k=[\zeta_k^{(1)} ,\zeta_k^{(2)} ]$ from (\ref{eq_14''})  generates  a bounded matrix $(\zeta_k^{(1)} {\zeta_k^{(2)}}^\top)\odot \1_\Sigma \in \IR^{n\times n}$. Express the $k$-th iterate $\zeta_k$ as $\zeta_k=\exp(\nu_k)$ with $\nu_k=[\nu_k^{(1)}; \nu_k^{(2)} ]$. Introduce a linear transform $\IB$,
\beqq
\IB(\nu_k):= \1_\Sigma\odot 
\log(\zeta_k^{(1)} {\zeta_k^{(2)}}^\top) =\1_\Sigma\odot 
\IT(M^\top \nu_k)=\IT((\IT^{-1}(\1_\Sigma))\odot M^\top \nu_k).
\eeqq 
From Prop.~\ref{prop5.5},  the null space of $\IB$ is the null space $\cN$ in (\ref{cN}), i.e., 
\beqq
\cN=\{w: \IT^{-1}(\1_\Sigma)\odot M^\top w=0\}=\{w:A\odot ( w^{(1)} \1_n^\top +\1_n {w^{(2)}}^\top)=0\}.
\eeqq
Let $P:\IR^{2n}\to \IR^{2n}$ be the orthogonal  projection with kernel  $\cN$.
  Let $m$ be the smallest singular value of $\IB$. Then 
$\| \IB \nu_k\|\ge m \| P \nu_k\|$.
Hence,
the boundedness  $(\zeta_k^{(1)} {\zeta_k^{(2)}}^\top)\odot \1_\Sigma$ actually indicates the boundedness of $\{\|P\nu_k\|: k=1,2,3,\ldots\}$, when  $\{\nu_k\}$ and $\{\zeta_k\}$ are chosen to minimize $\mbf(\nu)$ or $\mbg(\zeta)$, respective. This justifies the norm assumption required in   Theorem~\ref{thm1}.

\end{rem}

The following theorem states  that LB iterates are exactly KR iterations,  when $k$ is sufficiently large. Since the proof is lengthy, we place it in the appendix. 
\begin{theorem}\label{LBstepsize}Suppose that  $A\in \IR^{n\times n}$ has  total support. For $k$ sufficiently large, the step size $\alpha_k$ in the LB iteration  is $1$. 
\end{theorem}

%

\subsection{Stability issues in practical algorithms}\label{sec3.4} When we balance a sequence of matrices with $t$ increasing, the norm of these scaling vectors will increase synchronously.  Without careful  numerical treatment, large numerical errors can easily occur  in KR, NE and LB algorithms.   
Two techniques  proposed in the  Stabilized Scaling algorithms\cite{Schmitzer2019} will be employed in our simulation studies of KR, NE and LB algorithms.

In the application of optimal transport, we are interested in 
  balancing  
 a sequence of matrices
 \beqq \label{probA} A=\exp(-t \IT(c))\eeqq for a sequence of $t$-sequence, 
 i.e., $\diag( \zeta^{(1)})A \diag( \zeta^{(2)})$ is doubly stochastic under some scaling vectors $\zeta^{(1)},  \zeta^{(2)}$. The first technique is that  
to avoid the numerical inaccuracy caused by the large entries in scaling vectors, we should execute matrix balancing algorithms  in the Log-Domain. For instance, in the LB method, we shall  avoid computing/storing  $\zeta^{(1)},  \zeta^{(2)}$ in matrix balancing algorithms. Instead, by expressing   $\zeta^{(1)}, \zeta^{(2)}$ as $ \zeta^{(1)}=\exp(t 
\nu^{(1)}) $ and $ \zeta^{(2)}=\exp(t \nu^{(2)})$ for some $\nu=[
\nu^{(1)}, \nu^{(2)}]$, we should conduct matrix balancing in terms of  $\nu^{(1)}$ and $\nu^{(2)}$. Hence, the LB iteration in  (\ref{eq_14''}) should be rewritten as 
\beqq
\nu_{k+1}= \nu_k+ t^{-1}\log (\1_{2n} -\alpha_k (C_k+B_k)^\dagger (B_k-I_{2n}) \1_{2n}),
\eeqq
and $C_k+B_k$ can be expressed as  \beqq
C_k+B_k=\left(\begin{array}{cc}
\diag( \1_\Sigma \odot \exp(- t \IT(c-M^\top \nu_k ))) \1_n & \1_\Sigma \odot \exp(- t\IT(c-M^\top \nu_k ))\\
(\1_\Sigma \odot \exp(- t\IT(c-M^\top \nu_k)) )^\top  & \diag( (
\1_\Sigma\odot \exp(- t\IT(c-M^\top \nu_k ) ))^\top \1_n) \end{array}\right).
\eeqq

 The second technique is  to
use  $\mu$-translation  to  reduce  numerical errors in matrix balancing computation. 
Suppose  the
scaling vectors $\{\exp(t\nu^{(1)} ),\exp(t\nu^{(2)})\}$
for the squared matrix $A=\exp(- t \IT(c))\in \IR^{n,n}$ is available. 
 Then  the squared (shifted) matrix \beqq \exp(-t(\IT(c)-M^\top \mu))=\diag(t\mu^{(1)})A \diag(t\mu^{(2)})
\eeqq can be balanced by translated  scaling vectors $\{\exp(t(\nu^{(1)} -\mu^{(1)} )),\exp(t(\nu^{(2)}-\mu^{(2)}) )\}$.
How should we choose 
 $\{\mu^{(1)},\mu^{(2)}\}$?
Suppose $\exp(- t_{k-1} \IT(c))$ can be balanced by scaling vectors $\{\exp(t_{k-1}\nu^{(1)} ),\exp(t_{k-1}\nu^{(2)})\}$.  When $t_{k-1}$ is sufficiently large,  $\{\exp(t_{k}\nu^{(1)} ),\exp(t_{k}\nu^{(2)})\}$ provides a  good  approximation  for scaling vectors of $\exp(- t_{k} \IT(c))$.  Thus, one good empirical choice is   $\mu^{(1)}=\nu^{(1)}$
and $\mu^{(2)}=\nu^{(2)}$. Once the scaling vectors
 $\{\exp(t_{k}\xi^{(1)}), \exp(t_{k}\xi^{(2)})\}$ 
 of the shifted matrix  \beqq \exp(-t_{k} \IT(c-M^\top \nu))\eeqq
are computed, we know that  the original matrix $\exp(-t_{k} \IT(c))$ in (\ref{probA}) can be balanced by scaling vectors 
$\{\exp(t_{k}(\nu^{(1)}+\xi^{(1)}) ),\exp(t_{k}
(\nu^{(2)}+\xi^{(2)} )
)\}$.
 In summary, we have the following algorithm for the problem in (\ref{probA}).
\begin{alg}\label{Alg3.8} Input: a matrix $\IT(c)\in \IR^{n\times n}$ and a sequence $t_1, t_2, \ldots, t_{\max}$ in $\IR$.
\begin{itemize}
\item Initialize $\nu_0=0_{2n}$. For $k=1,2,\ldots, k_{\max}$, repeat the following two steps:
\item Compute  a scaling vector $\exp(t\mu)\in \IR^{2n}$ which  balances the matrix $\exp(-t_{k}(\IT(c)-M^\top \nu_{k-1}) )$.
\item Update $\nu_{k}=\nu_{k-1}+\mu \in \IR^{2n}$. 
\end{itemize}
Output: $\nu_{k_{\max}}$. Here, the vector 
 $\exp(t_{\max} \nu_{k_{\max}}) $ balances the matrix  $\exp(-t_{\max} \IT(c))$. 
\end{alg}

\section{Numerical simulations}
We provide three experiments in the section: (i)  Comparison of  matrix balancing schemes; (ii)Comparison experiments of matrix balancing  in solving discrete optimal transport; (iii) Application of   sparse support algorithms on large data-sets.  

\subsection{Matrix balancing}

\subsubsection{ Comparison in matrix balancing }

We compare four matrix balancing methods, including
\begin{itemize}
\item   Sinkhorn-Knopp algorithm(SK) in (\ref{SKb});
\item three Newton method based algorithms:
\begin{itemize}
\item Knight-Ruiz method(KR) in (\ref{eq_7'});
\item Negative entropy  method(NE) in (\ref{eq_79});
\item Logarithmic barrier  method(LB) in (\ref{eq_14''}).
\end{itemize}

\end{itemize} 
We select
 three    matrices,   $A=\exp(-magic(20)/20)$ of size $20\times 20$, 
 $A=\exp(-magic(50)/20)$ of size $50\times 50$,  and  $A=\exp(-magic(200)/50)$ of size $200\times 200$. 
 Here magic($n$) produces  an $n\times n$ matrix  from the integers $1,2,\ldots, n^2$ with 
 with equal row/column/diagonal sums.
 See the top row of Fig.~\ref{magic} for the pattern visualization of  matrices $magic(20)$, $magic(50)$ and $magic(200)$.
  \begin{figure}
   \includegraphics[width=0.3\textwidth]{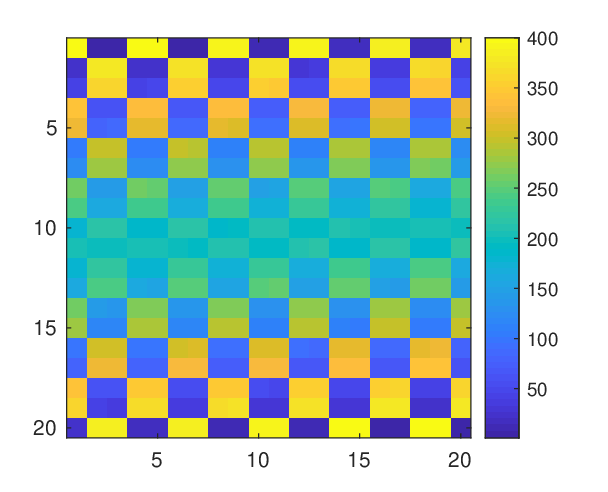} 
   \includegraphics[width=0.3\textwidth]{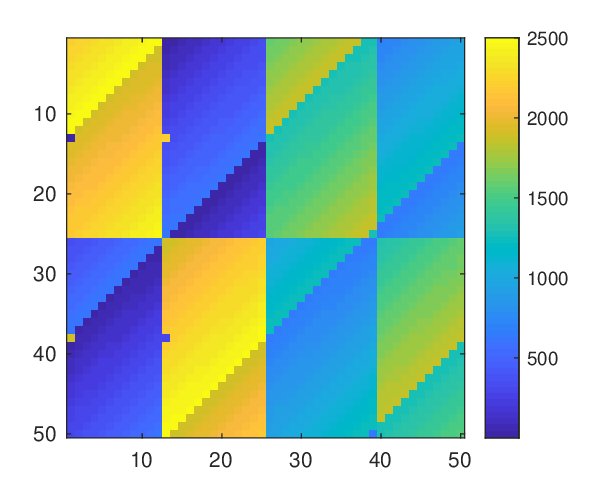} 
   \includegraphics[width=0.3\textwidth]{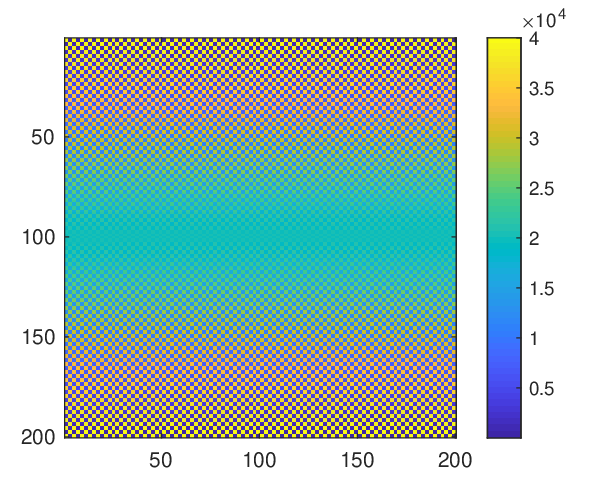} \\
   \includegraphics[width=0.3\textwidth]{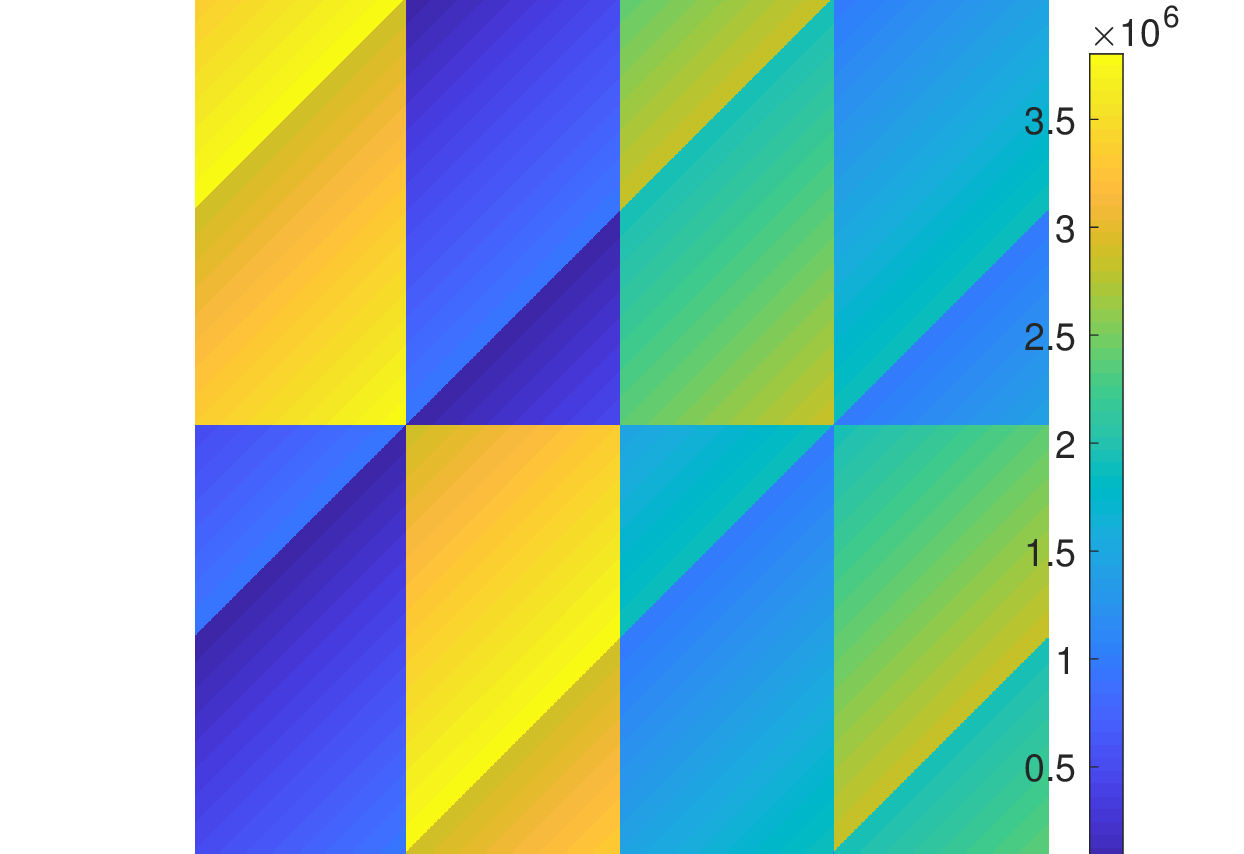} 
   \includegraphics[width=0.3\textwidth]{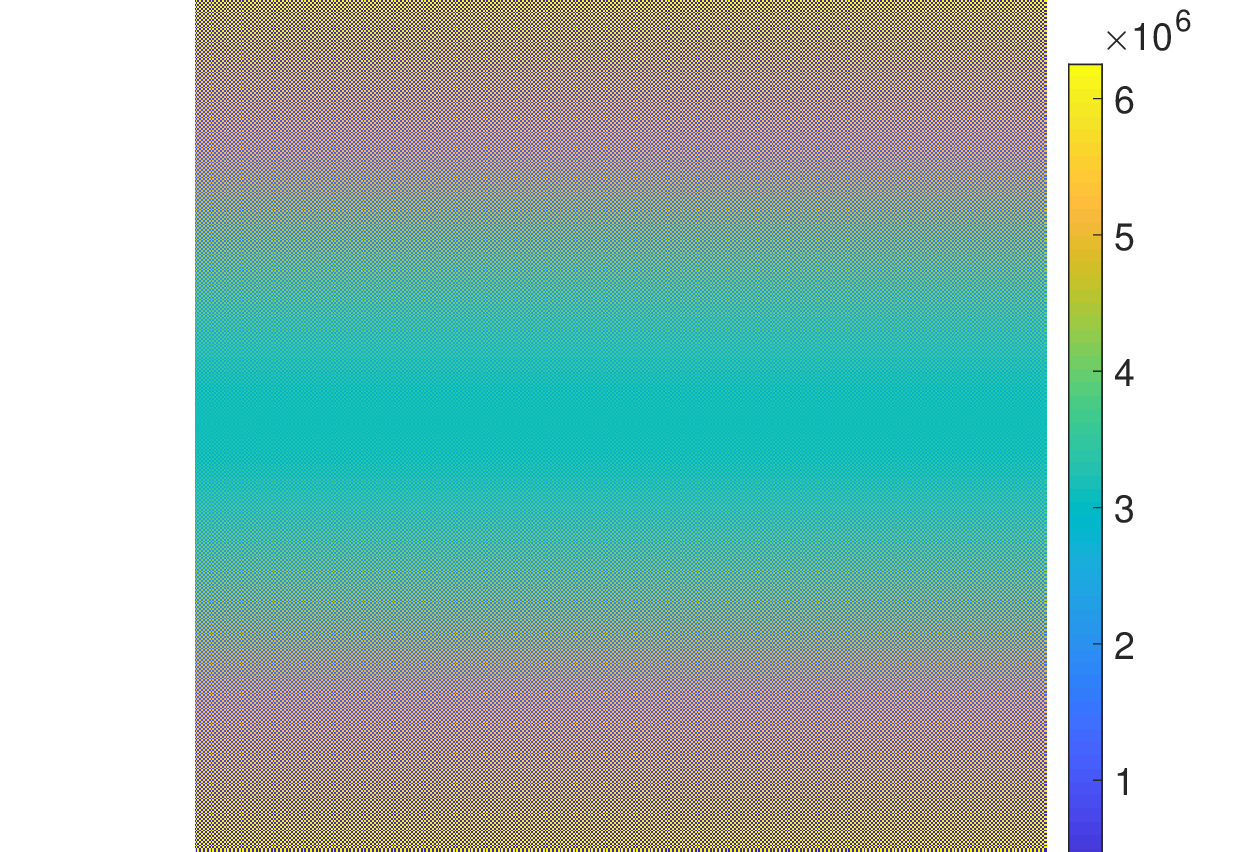} \\
 \caption{  ( Left to right subfigures in the top row show matrices
 $magic(20)$, $magic(50)$ and $magic(200)$, respectively. Left and right subfigures in the bottom row show $magic(1950)$ and $magic(2500)$, respectively.
}  \label{magic} 
 \end{figure}
At the $k$-th iteration, let $B_k:=\diag( \zeta^{(1)}_k) A \diag( \zeta^{(2)}_k)$ be the matrix corresponding to scaling vectors $\{ \zeta^{(1)}_k,  \zeta^{(2)}_k\}$ produced from matrix balancing algorithms. Consider  
the performance metric to evaluate the matrix balancing error:  \beqq\label{Error}
Error:=\| B_k \1_n -\1_n\|_1+\| B_k^\top  \1_n -\1_n\|_1.\eeqq

\begin{itemize}
\item First, we start with  the same initial vector $\1_n$ in the four methods. Results are reported in 
Figure~\ref{figure3}, where 
KR  empirically    gives very  fast convergence in the perspective of  CPU time. 
 Sinkhorn-Knopp algorithm, one popular algorithm,  typically  requires more iterations than  Newton methods.  However, thanks to its low complexity  in each iteration, SK can produce acceptable results  economically.
For instance, as shown in 
$A=\exp(-magic(20)/20)$  and $A=\exp(-magic(200)/50)$,
SK reaches a solution with error less than $10^{-2}$, much  faster than NE and LB.
On the other hand, SK has very poor convergence  in handling   $A=\exp(-magic(50)/20)$. 
 This case with $n=50$ is actually a challenging problem.
  Optimal scaling vectors $ \zeta^{(1)}, \zeta^{(2)}$ have norm  both greater than $10^{12}$, which suggest that  $\exp(-magic(50)/20)$  nearly does not have total  support. Under the circumstance,  all Newton methods give relatively slow convergence.

%
 
\item Second, we further  examine the case   $A=\exp(-magic(50)/20)$ from  the framework of negative entropic barrier functions. 
Consider a sequence of matrices $\exp(-t\cdot  magic(50) )$ with $t=1/160, 1/80, 1/40$ and $ 1/20$, respectively.
 The CPU time of these balancing tasks is reported in Table~\ref{tab:ex1_01}. 
Matrix balancing  task with small $t$ is easier  than those tasks with large $t$. For  $t=1/160, 1/80, 1/40, 1/20$, the geometric mean of the norm of  the  scaling vectors is 
\beqq
\|  \zeta^{(1)}\|^{1/2} \| \zeta^{(2)}\|^{1/2}= 2.31\times 10^1,\; 8.05\times 10^2,\;  1.61\times 10^6, \; 1.08\times 10^{13},
\eeqq
respectively.\footnote{As one reference,  $\|\zeta^{(1)} \|^{1/2} \| \zeta^{(2)}\|^{1/2}$ is   $1.663$ and $137.8$
for the problems  $\exp(-magic(20)/20)$ and $\exp(-magic(200)/20)$, respectively. } From Remark~\ref{3.6}, the norm growth of scaling vectors reflects that the matrices to be balanced nearly do not have total support. In addition, 
we  examine   the scaling vectors \beqq\label{eq_rc}
 \zeta^{(1)}=\exp( t \nu^{(1)}),  \zeta^{(2)}=\exp(t\nu^{(2)}),
\eeqq
by plotting  those entries of  dual vectors $\nu^{(1)}$ and $\nu^{(2)}$  in Fig.~\ref{nu12}. Observe the similarity among these  vectors $\nu^{(1)}$ and vectors $\nu^{(2)}$.  Fast convergence of Newton methods relies on the proximity of the initialization to the attractive basin.
Thanks to the similarity, we can speed up these Newton methods, when   the optimal scaling vectors of matrices  with previous   $t$ are employed as   warm starts.
Notice that the  CPU time  with $t=1/20$ is improved significantly,  compared with CPU time reported   in Fig.~\ref{figure3}.

\begin{table}
    \centering
\begin{tabular}{|c|c|c|c|c|}
\hline
 \multicolumn{5}{|c|}{$\exp(-magic(50)\cdot t)$} \\
 \hline
\multirow{2}{*}{$t$ value} &\multicolumn{1}{c|}{NE} & \multicolumn{1}{c|}{LB} & \multicolumn{1}{c|}{KR} & \multicolumn{1}{c|}{SK}\\
& (s) & (s) & (s) & (s) \\
\hline
$1/160$  & 0.0032 & 0.0018 & 0.0007 & 0.025 \\  
\hline
$1/80$   & 0.0047 & 0.0021 & 0.0008 & 0.075 \\  
\hline
$1/40$   & 0.0071 & 0.0039 & 0.0013 & 0.140 \\  
\hline
$1/20$   & 0.0074 & 0.0042 & 0.0025 & 0.939 \\ 
\hline
\end{tabular}
\caption{Computational time (sec) in balancing  $\exp(-magic(50)\cdot t)$ under tolerance $10^{-5}$.}
\label{tab:ex1_01}
\end{table}

%
 \end{itemize}
\begin{figure}
 \includegraphics[width=0.32\textwidth]{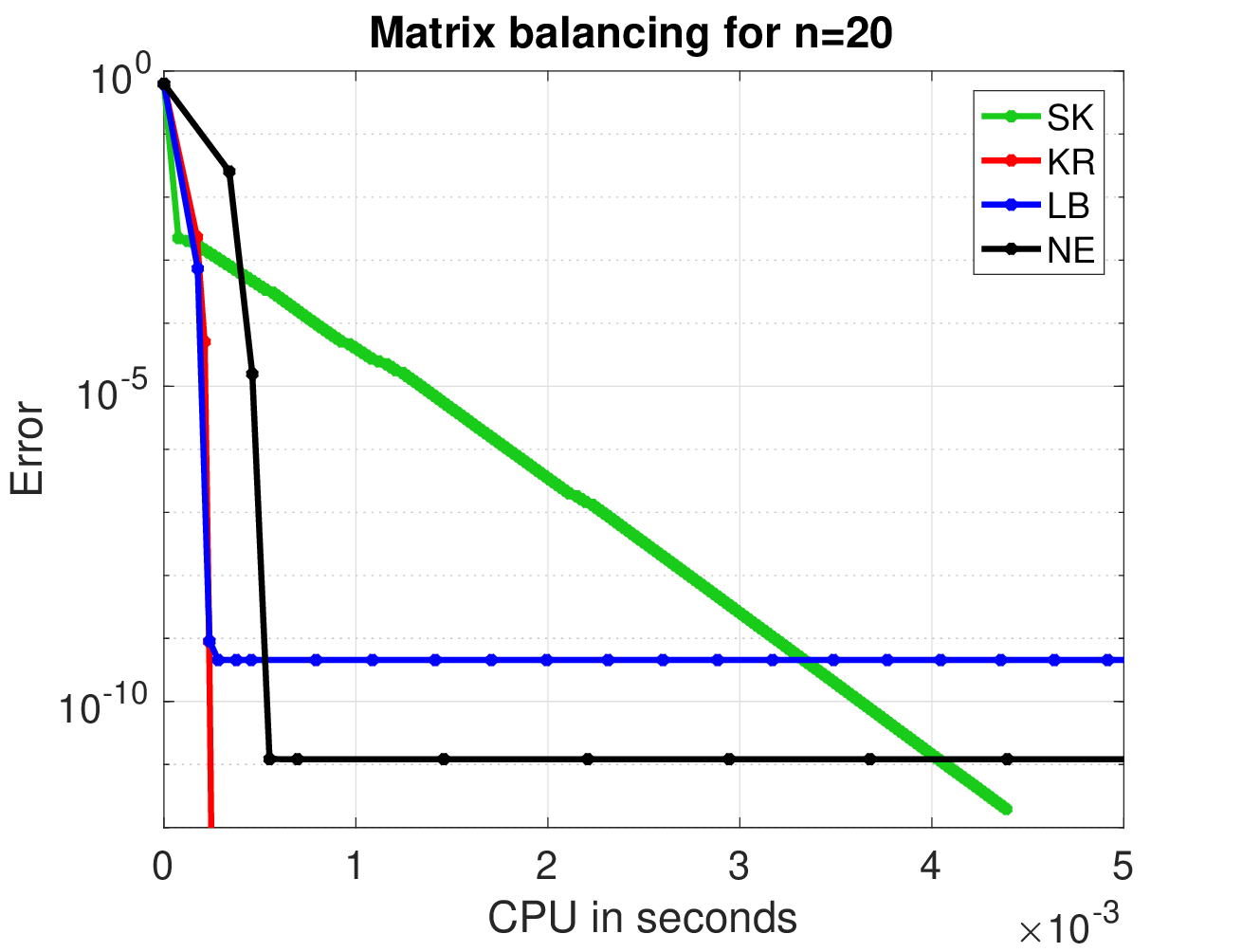}
 \includegraphics[width=0.32\textwidth]{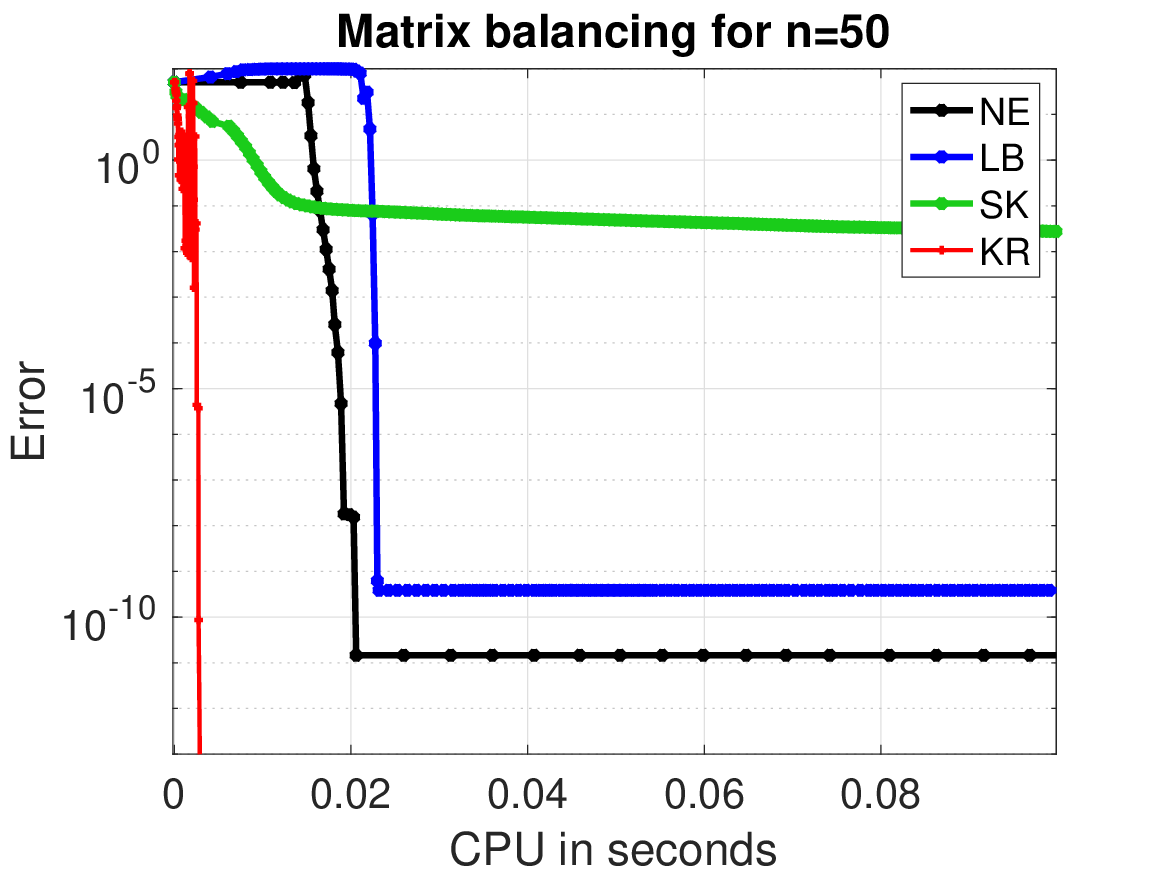}
    \includegraphics[width=0.32\textwidth]{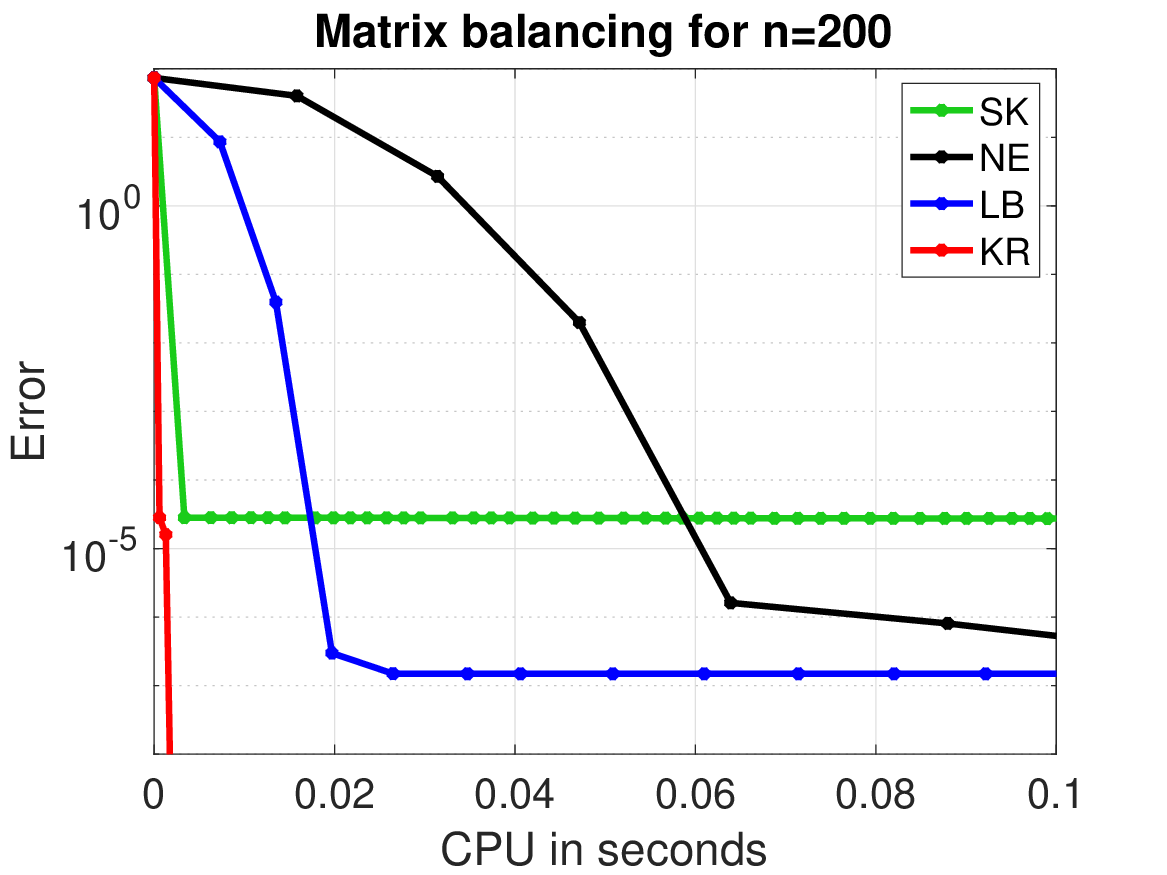}
 \caption{ 
 Comparison of SK with other Newton based matrix balancings, $n=20$(left), $n=50$(middle), and $n=200$ (right). The performance metric  is (\ref{Error}). }  \label{figure3} 
 \end{figure}

 \begin{figure}
\begin{center}
   \includegraphics[width=0.3\textwidth]{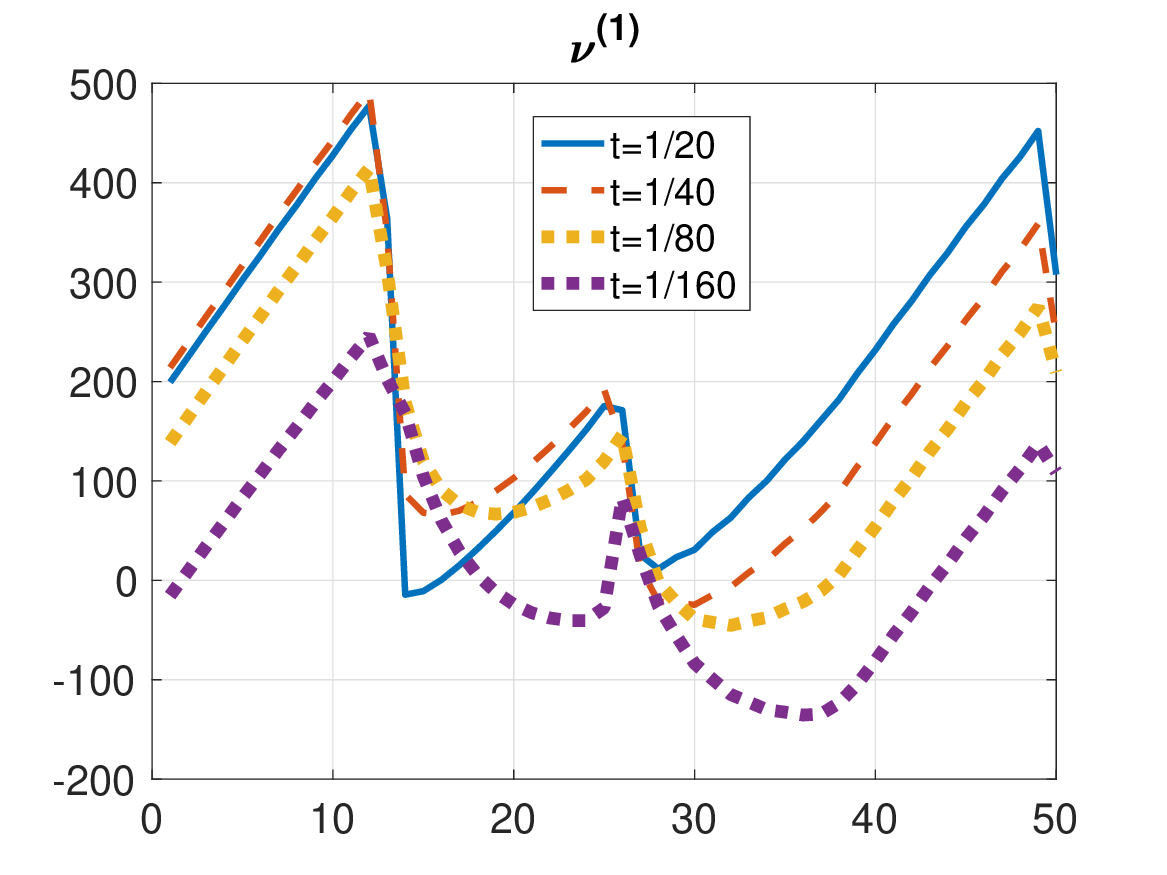} 
   \includegraphics[width=0.3\textwidth]{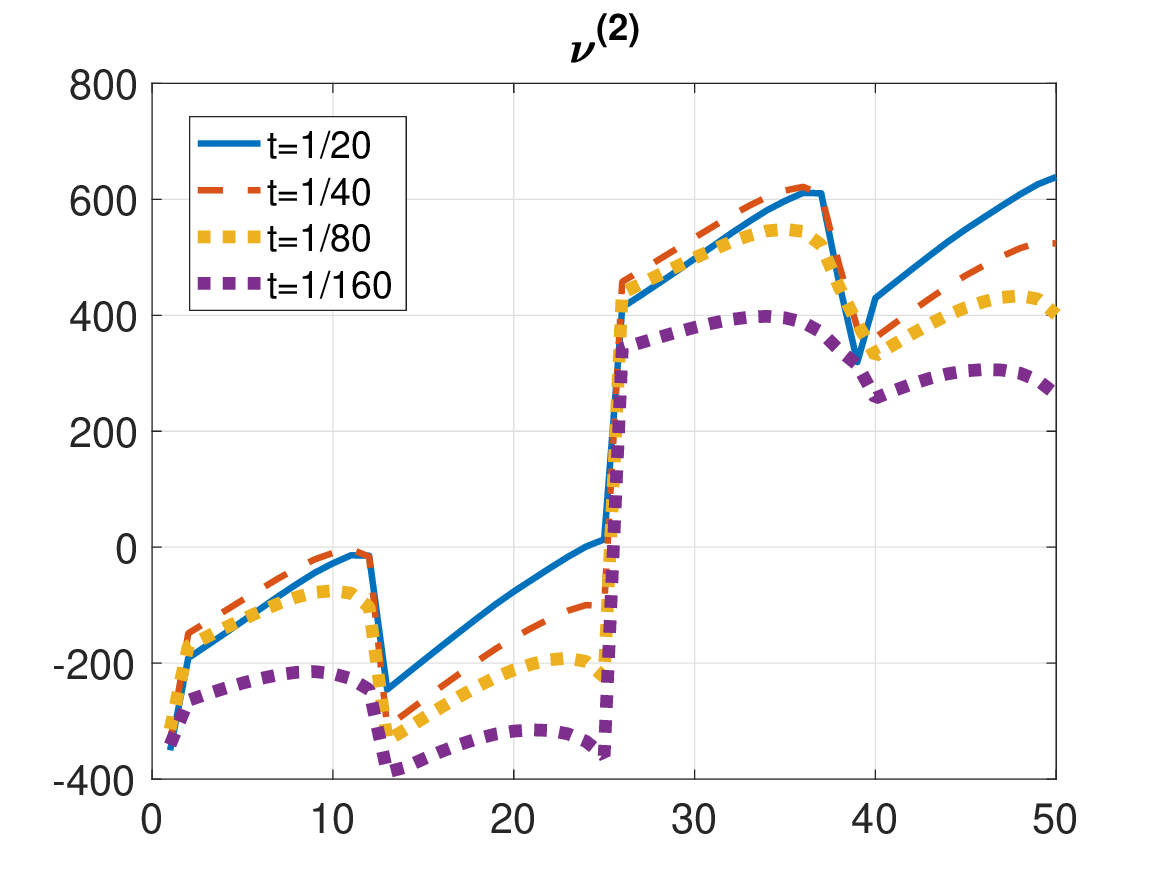} \\
   \includegraphics[width=0.3\textwidth]{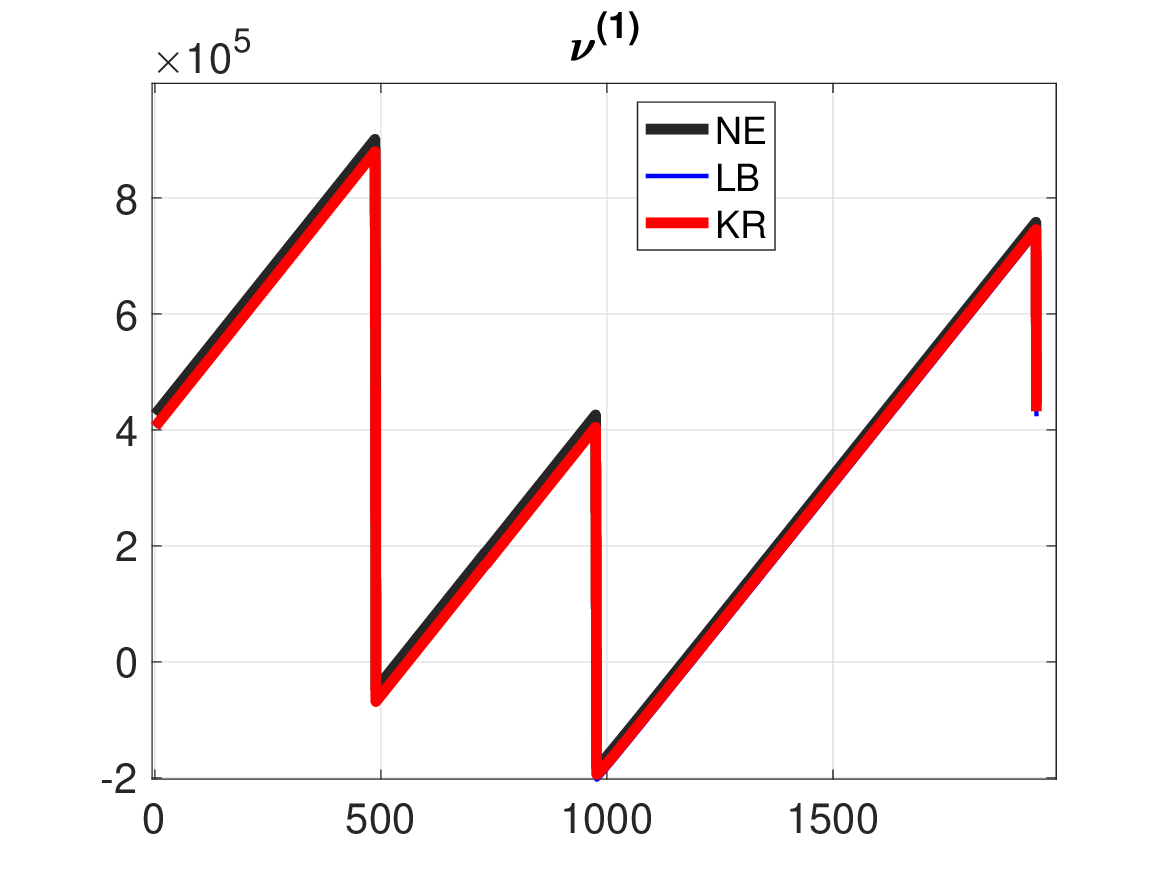} 
   \includegraphics[width=0.3\textwidth]{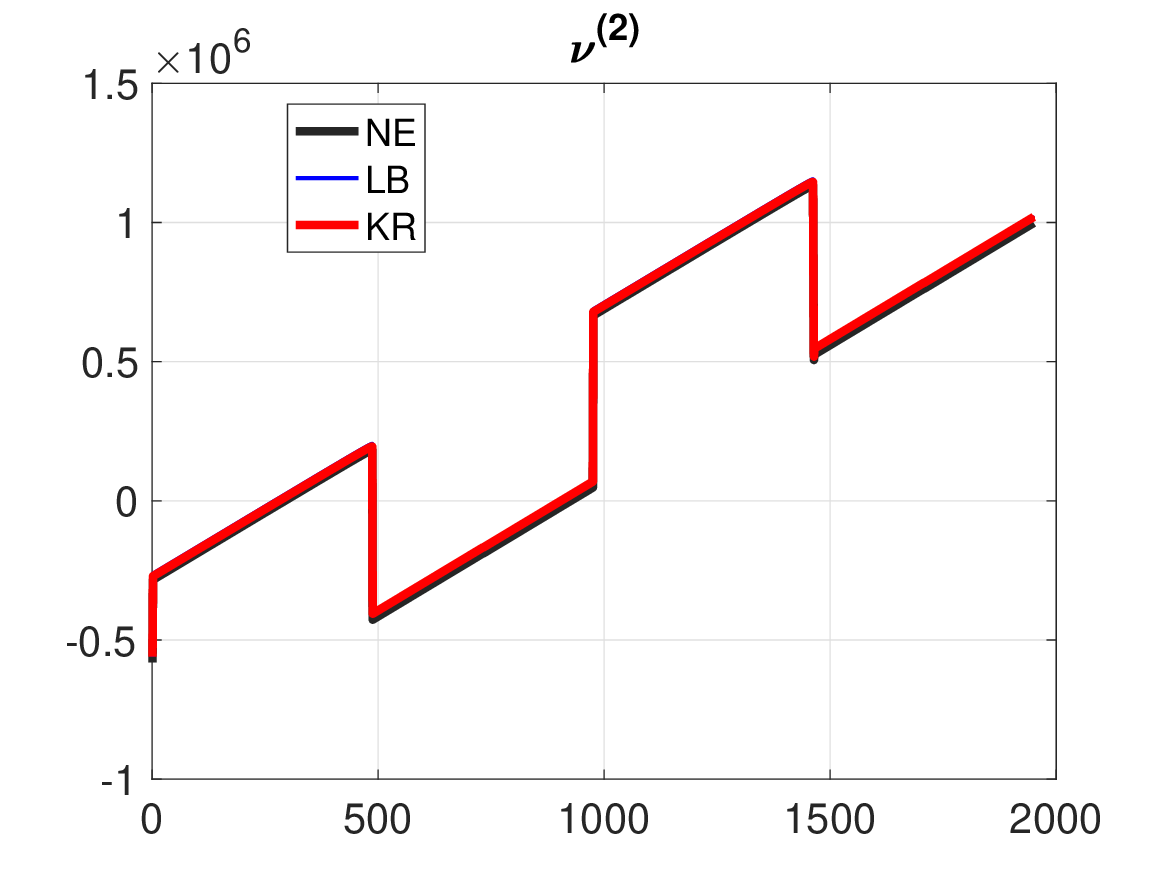} 
   \end{center}
 \caption{ Top:  Vectors $\nu^{(1)}$ and $\nu^{(2)}$ of (\ref{eq_rc}) in balancing $\exp(-magic(50)\cdot t)$.
 Bottom:  Vectors $\nu^{(1)}$ and $\nu^{(2)}$  in optimal transport with $\IT(c)=magic(1950)$.}  \label{nu12} 
 \end{figure}

\subsubsection{Comparison in solving optimal transport}
We demonstrate  the application of matrix balancing algorithms in 
solving optimal transport along a central path for $t=t_k=t_0 \eta^k$, $k=0,1,2,\ldots, t_{\max}$. 
We evaluate matrix balancing  algorithms in handling the  cases with assignment matrix
\begin{itemize}
\item  $\IT(c)= magic(2500)$;
\item  $\IT(c)= magic(1950)$;
\item  $\IT(c)$ has $n^2$ entries $\{c_{j,k}=\|y_j-z_k\|^2: i,j=1,\ldots, n\}$, where $\{y_k\}_{k=1}^{n}$ is one TLC point-set and  and $\{z_k\}_{k=1}^{n}$ is one translated FRC point-set, $n=254$.
\end{itemize}
Perform matrix balancing of a sequence of  matrices  for a few positive values $t=t_k$ until the relative duality gap $\epsilon$ is met, i.e, 
\beqq
\{
\exp(- t_i \IT(c)):  t_1< t_2<\ldots<t_{\max}
\}.
\eeqq Let $x_{opt}$ be an optimal primal vector. 
For  each algorithm,
we report the  computation time, 
when \textbf{ relative duality gap} falls within 
 a given tolerance level  $\epsilon$,
\beqq
\langle c, x_{opt}\rangle^{-1} (\langle c, x\rangle-\langle \nu, \1_{2n}\rangle )\le \epsilon.
\eeqq

Table~\ref{tab:ex1_0}-\ref{tab:ex1_1}
 report the CPU time and the corresponding $t_k$ value for  various tolerance level $\epsilon$.
 Notice that  when    identical sequences of $t_k$ are reported,  identical sequence of matrices are balanced in these methods. 
Consider a fixed matrix balancing tolerance $\epsilon_{MB}=10^{-5} \sqrt{n}$ as  the  stopping criterion. 
This criterion ensures that  the gradient has small norm, 
$\|\nabla \mbf(\nu)\|\le \epsilon_{MB}$, see   (\ref{gradf}).
 Experiment results show that all Newton methods work quite well in the three problems. In particular,  KR consistently gives the fastest convergence among these Newton methods. 
 However, a winner between Newton methods and SK
 usually  depends on the  difficulty of the problem  itself. Observe the pattern similarity between $magic(200)$ and $magic(2500)$ and observe 
 the pattern similarity between $magic(50)$ and $magic(1950)$ from  Fig.~\ref{magic}.
For   the problem $magic(2500)$, which  is relatively easy (compared with $magic(1950)$),
SK is a fast algorithm, which  produces acceptable results, much  faster than NE and LB as shown in $magic(200)$. On the other hand,  facing the challenging problem $magic(1950)$,   SK fails to produce acceptable results within $5000$ seconds.  
 As a result, we can see the similarity of the dual vectors $\nu^{(1)},\nu^{(2)}$  in  Fig.~\ref{nu12}.
As  in  $magic(50)$ and $magic(1950)$,  entries of  dual vectors in a point-set matching problem actually  vary a lot. From this viewpoint,  it is not so surprising that SK   has the worst convergence in solving the point set  matching problem, shown in  Table~\ref{tab:ex1_1}.

\begin{table}[h]
    \centering
\begin{tabular}{|c|c|c|c|c|c|c|c|c|}
\hline
 \multicolumn{9}{|c|}{ $magic(2500)$} \\
 \hline
\multirow{2}{*}{$\epsilon$} &\multicolumn{2}{c|}{NE} & \multicolumn{2}{c|}{LB} & \multicolumn{2}{c|}{KR} & \multicolumn{2}{c|}{SK}\\
\cline{2-9}
& time & $t_k$ & time & $t_k$ & time & $t_k$ & time & $t_k$\\
\hline
$1e-1$ & 22.35 & $291.9$ & 43.73 & 291.9      & 13.77   & $291.9$     & 2.99 & $291.9$ \\
\hline
$1e-2$ & 32.10 & $3325$ & 62.21  & $3325$     & 19.15 & $3325$      & 4.42  & $3325$ \\
\hline
$1e-3$ & 45.96 & $37876$ & 80.46 & $37876$    & 24.56 & $37876$     & 6.15  & $37876$ \\
\hline
$1e-4$ & 51.55 & $287627$ & 94.45 &  $287627$    & 28.73    & $287627$    & 7.00  & $287627$ \\
\hline
$1e-5$ & 116.87 & $2675044$ & 108.32 & $4914369$  & 34.00 & $3276247$   & 8.17  & $3276247$ \\
\hline
\end{tabular}
\caption{Computational time(sec) in solving optimal transport with $\IT(c)=magic(2500)$.}
\label{tab:ex1_0}
\end{table}

\begin{table}[h]
    \centering
\begin{tabular}{|c|c|c|c|c|c|c|c|c|}
\hline
 \multicolumn{9}{|c|}{ $magic(1950)$} \\
 \hline
\multirow{2}{*}{$\epsilon$} &\multicolumn{2}{c|}{NE} & \multicolumn{2}{c|}{LB} & \multicolumn{2}{c|}{KR} & \multicolumn{2}{c|}{SK}\\
\cline{2-9}
& time & $t_k$ & time & $t_k$ & time & $t_k$ & time & $t_k$\\
\hline
$1e-1$ & 25.29  & 437.9  & 66.42  & 437.9      & 15.30   & 437.9    & $11197$ & $388.2$ \\
\hline
$1e-2$ & 34.36  & $3325$   & 96.84  & 3325    & 19.89   & 3325      & 11626 & $4.42\times 10^3$ \\
\hline
$1e-3$ & 46.63  & $37877$  & 133.38 & 37877    & 25.71   & 37877     & -  & - \\
\hline
$1e-4$ & 65.09  & 287627 & 190.35 &  287627 & 41.22   & 287627    & -  & - \\
\hline
$1e-5$ & 129.21 & 2184164 & 2796 & 1531812  & 389.12 & 2184164   & -  & - \\
\hline
\end{tabular}
\caption{Computational time(sec) in solving optimal transport with $\IT(c)=magic(1950)$.}
\end{table}

\begin{table}
    \centering
\begin{tabular}{|c|c|c|c|c|c|c|c|c|}
\hline
 \multicolumn{9}{|c|}{Lung branch points 
 $(n=254)$} \\
 \hline
\multirow{2}{*}{$\epsilon$} &\multicolumn{2}{c|}{NE} & \multicolumn{2}{c|}{LB} & \multicolumn{2}{c|}{KR} & \multicolumn{2}{c|}{SK}\\
\cline{2-9}
& time & $t_k$ & time & $t_k$ & time & $t_k$ & time & $t_k$\\
\hline
$1e-1$ & 0.14 & $86.5$ & 0.32 & 86.5 & 0.11 & $86.5$  & 4.94  & $70.6$ \\
\hline
$1e-2$ & 0.22 & $4.37\times 10^{2}$ & 0.40 & $4.37\times 10^{2}$  & 0.13 & $4.37\times 10^{2}$  & $25.26$  & 509.8 \\
\hline
$1e-3$ & 0.39 & $1.478\times 10^{3}$ & 0.64 & $1.478\times 10^{3}$  & 0.28 & $1.478\times 10^{3}$ & $200.46$ &  $1.348 \times 10^{3}$ \\
\hline
$1e-4$ & 1.10 & $4.988\times 10^{3}$ & 5.19 & $4.988\times 10^{3}$  & 0.66 & $4.988\times 10^{3}$ &  $851.94$  & $5.247 \times 10^{3}$ \\
\hline
$1e-5$ & 1.25 & $2.5251\times 10^{4}$ & 5.89 & $2.5251\times 10^{4}$  & 0.80 & $2.5251\times 10^{4}$ &  $859.51$ & $2.1621\times 10^{4}$ \\
\hline
\end{tabular}
\caption{Computational time(sec) in solving optimal transport with $c=$ $L^2$-distance assignment. }
\label{tab:ex1_1}
\end{table}
%

%
%
%
\subsection{ Rigid-motion estimation}\label{sec_4.3}
One big advantage of SNNE over primal-dual methods is that SNNE updates multiplier vectors solely along the increase of $t$, i.e., no need to store/pass  $x$ between sub-problems. The memory requirement in SNNE can be much less than that in primal-dual methods, if the active support set  is  properly handled in  large-scale problems. 
  The following two experiments demonstrate the effectiveness  of   SNNE  in handling large-scale problems. 
In the first  study, we provide one comparison between SNNE, SNNE-sparse with  primal-dual methods, which are popularly used in solving linear programming. Here, we consider  two primal-dual methods: Mehrotra predictor-corrector  method, which is one widely-used primal-dual interior point method\cite{doi:10.1137/0802028}, and
 one commercial software solver, Gurobi, where the algorithm method is chosen to be  the barrier method. 
  In the first study, we  actually solve 
a number of optimal transport problems.
 For the second study, we demonstrate the flexibility of the entropic regularization.
 We   apply  entropic regularization, but  take $t$ as the outer loop variable to 
bypass the multiple optimal transport problems. The algorithms SNNE-t and SNNE-sparse are developed in this framework to optimize the computational time. 
 
We present a rigid motion experiment on a three-dimensional teapot   point cloud consisting of $41472$ points. 
 We subsample $1000/2500/5000$ point-sets $\{y_1,\ldots, y_n\}$  from  the teapot point cloud.  Select
one orthogonal matrix  $Q\in \IR^{3\times 3}$,
 and generate another set of  point-sets, $\{z_i=Qy_i: i=1,\ldots, n\}$,
 as shown  in Figure~\ref{figureRigid}.
 For simplicity, 
$\{y_i\}$  is  shifted  so that  $\sum_{i=1}^n y_i=0$.
Introducing a user-defined parameter $\eta>0$,  we estimate $Q\in \IR^{3\times 3}$ and $\IT(x)\in \Pi_n$ from the minimization, 
\beqq\label{Qx0}
\min_{Q}\min_x \left\{ \IF(Q, x):=\langle 
\mbc(Q),  x\rangle+ \eta \|Q-I_3\|^2_F\right\}, \eeqq
where the assignment $\mbc(Q)$ is a function of $Q$  with  $
 \IT(\mbc(Q))_{i,j}=\| y_i-Q z_j\|^2$. 
We can apply optimal transport for general non-rigid motion problems  via
introducing   regulation terms for splines. For instance,    see \cite{RefWorks:ChuiCVPR, Glaunes, JMIV}.

Here is one  \textit{ naive } algorithm, consisting of repeating the estimations of $Q$ and $x$:
\begin{itemize}
\item Fix $Q$. Estimate $x$ with $\IT(x)\in \Pi_n$, which is one optimal transport. 
\item Fix $x$. Solve  $Q$ from   the least squares problem,
 \beqq
\min_Q \{\IF(Q, x)=\sum_{i,j=1}^n\langle (y_i-Q z_j), x_{i,j}( y_i-Qz_j)\rangle+\eta \langle Q-I, Q-I \rangle\}.
\eeqq From the SVD property,  an optimal matrix  is $Q=UV^\top$, where  $U,V$ are unitary matrices in  the SVD,
\beqq\label{SVDq} UDV^\top=\sum_{i,j=1}^n\{x_{i,j} y_i z_j^\top +\eta I\}.\eeqq 
\end{itemize}   The performance metric is given by 
   \beqq\label{err}
   error:=\langle \mbc(Q),x\rangle\ge 0.
   \eeqq
   Note that $error=0$ if and only if $y_i=Qz_j$ holds for all $x_{i,j}>0$. 

For a fair  comparison,  we use  SNNE, Mehrotra  primal-dual method(PD) and Gurobi solver to solve optimal transport minimizer  $\IT(x)$ after each $Q$-update. 
Table~\ref{default} reports the computational time of  SNNE, Mehrotra  primal-dual method(PD) and Gurobi optimization software. We stop algorithms when $error$ reaches $10^{-4}$.
Figure~\ref{figureRigid2} shows the desired small error under  PD, Gurobi and SNNE, which indicates the successful  reconstruction of
 $x$ and $Q$ in  the cases  $n=2500$ and $n=5000$.  As  expected, when $n$ increases, the computational time increases accordingly. The computational time of PD is approximately proportional to $n^3$, while the computational time of SNNE or Gurobi is approximately proportional to $n^2$.
Clearly,  either Gurobi optimization software or Mehrotra predictor-corrector  method can deliver  an optimal solution of optimal transport in (\ref{P0}) very fast, when the  cardinality $n$ does not exceed  $1000$. 
However,  due to its advantage in low memory requirement, the inferior performance of SNNE becomes less apparent  in the case  $n=2500$ and $n=5000$.  See Figure~\ref{figureRigid2}.

\subsubsection{SNNE-t and SNNE-sparse}
 In SNNE,       after each  $Q$-update,  a sequence of matrices are balanced   to generate  one approximate  optimal transport minimizer for each assignment matrix $\mbc(Q)$.  Balancing these matrices   along multiple paths actually  makes SNNE  very  inefficient. 
  To alleviate the difficulty, we   introduce   the entropic regularization to (\ref{Qx0}) to  estimate $(Q,x)$ along  ``one"  inexact minimizer path  associated with a sequence  $\{t=t_1,\ldots, t=t_{\max}\}$,
\beqq\label{Qx}
\min_{Q}\min_{\Sigma}\min_x \left\{ \IF_t(Q, x, \Sigma):=\langle 
\mbc(Q),  x\rangle-t^{-1} \langle \IT^{-1}(\1_\Sigma), \log x\rangle+ \eta \|Q-I_3\|^2_F\right\}. \eeqq

At each $t$, we execute the following block coordinate steps  to approximate the  minimizer $(Q,x)$. \begin{itemize}
\item Fixing $Q$, use matrix balancing to  compute  an optimal $\IT(x)$, i.e., find $\nu$ to balance the matrix   $\exp(-t (c(Q)-M^\top \nu))$. Use $\nu$ to update $\Sigma$.
\item 
Fixing $x$, we update $Q$ by SVD computation in (\ref{SVDq}).
\end{itemize}
The convergence to  the exact minimizer $(Q,x)$ requires a sufficient number of  these block coordinate descent steps.  (See Prop. 2.7.1~\cite{Np}.)
As $t$ gets sufficiently large, $(Q,x)$ in (\ref{Qx}) is expected to  approach  one minimizer in (\ref{Qx0}).  We call  the new algorithm solving (\ref{Qx}) along one $t$-path as  \textbf{SNNE-t}. 
Note that the major difference from (\ref{Qx0})  is that
  the  parameter $t$  in (\ref{Qx})  is  an outer loop variable.
 Results are reported in  Table~\ref{default}.    Thanks to bypassing   multiple optimal transport problems, 
  SNNE-t actually   consumes much less computation time than  previous algorithms.

%
    
Next, we 
implement   SNNE-sparse to solve (\ref{Qx}), where the support of $x$ is 
 dynamically updated  reduce the memory load of SNNE-t. That is, 
   $x$ and $Q$ are updated alternately with initialization   
  $Q=I$. For each $Q$ fixed,  we compute $x$  via 
  one approximate multiplier vector $\nu_\xi$ subject to the approximate support set   $\Sigma_\xi$ for $\xi$ in $\{1,2,3,\ldots, \xi_{\max}\}$,   as in  Alg.~\ref{SNNEalg}. 
To have  a better  control on sparsity of $\Sigma_\xi$ in SNNE, we  select a sparse parameter  $ k=20$ to ensure  an upper bound $(2k+1)n$ for the  cardinality of $\Sigma_\xi$. 
The result of SNNE-sparse is reported in Table~\ref{default}  and Fig.~\ref{figureRigid2}.   Clearly, the introduction of  matrix sparsity  together with the usage of one $t$-path greatly  reduces the computational time of the implementation of  SNNE-sparse. Here,  $\xi_{\max}=3$ is used. 
The heuristic choice of $\xi_{\max}$  has a big influence on  the whole computational time. When $\xi_{\max}=2$ is used, the computation time can be further reduced. See the column of SNNE-sparse-2 in Table~\ref{default}.      \begin{figure}   
  \includegraphics[width=0.3\textwidth]{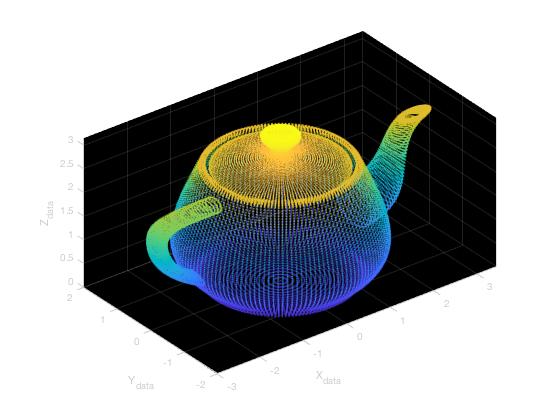} 
   \includegraphics[width=0.3\textwidth]{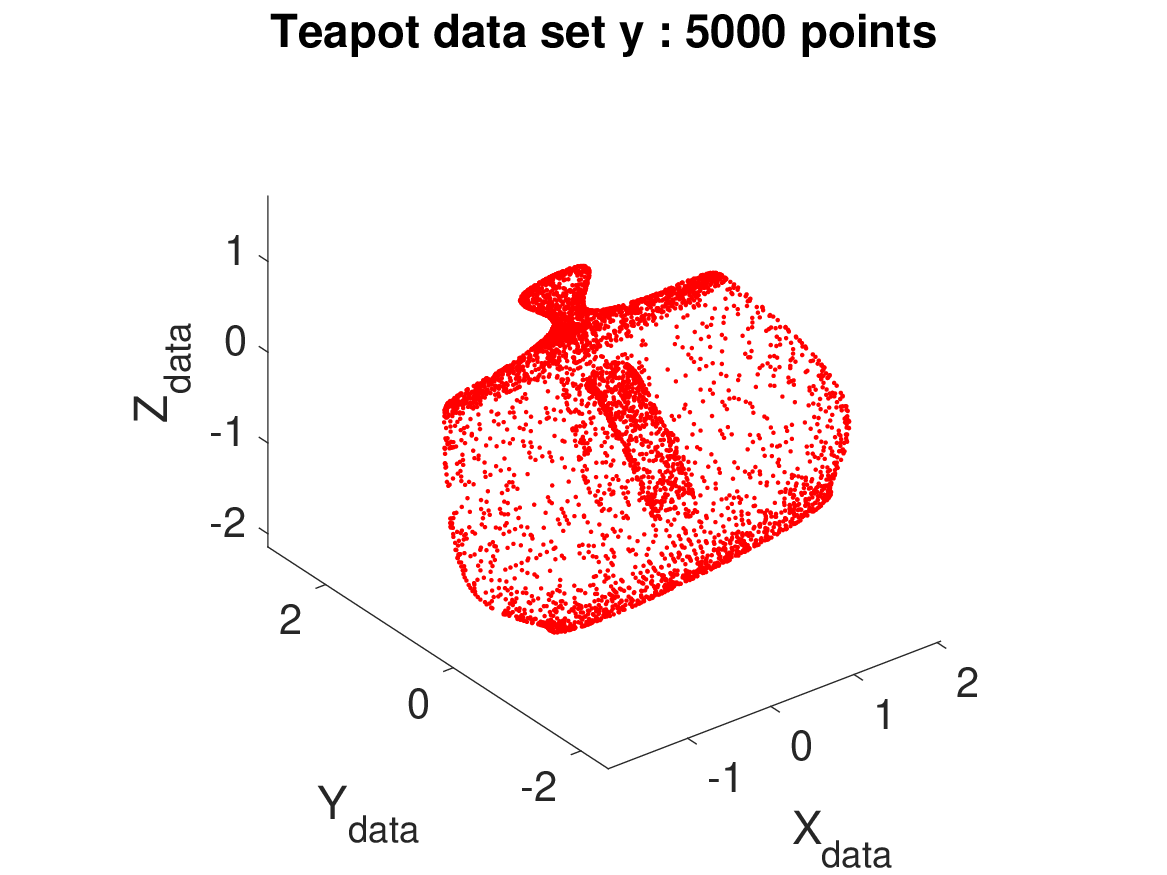} 
 \includegraphics[width=0.3\textwidth]{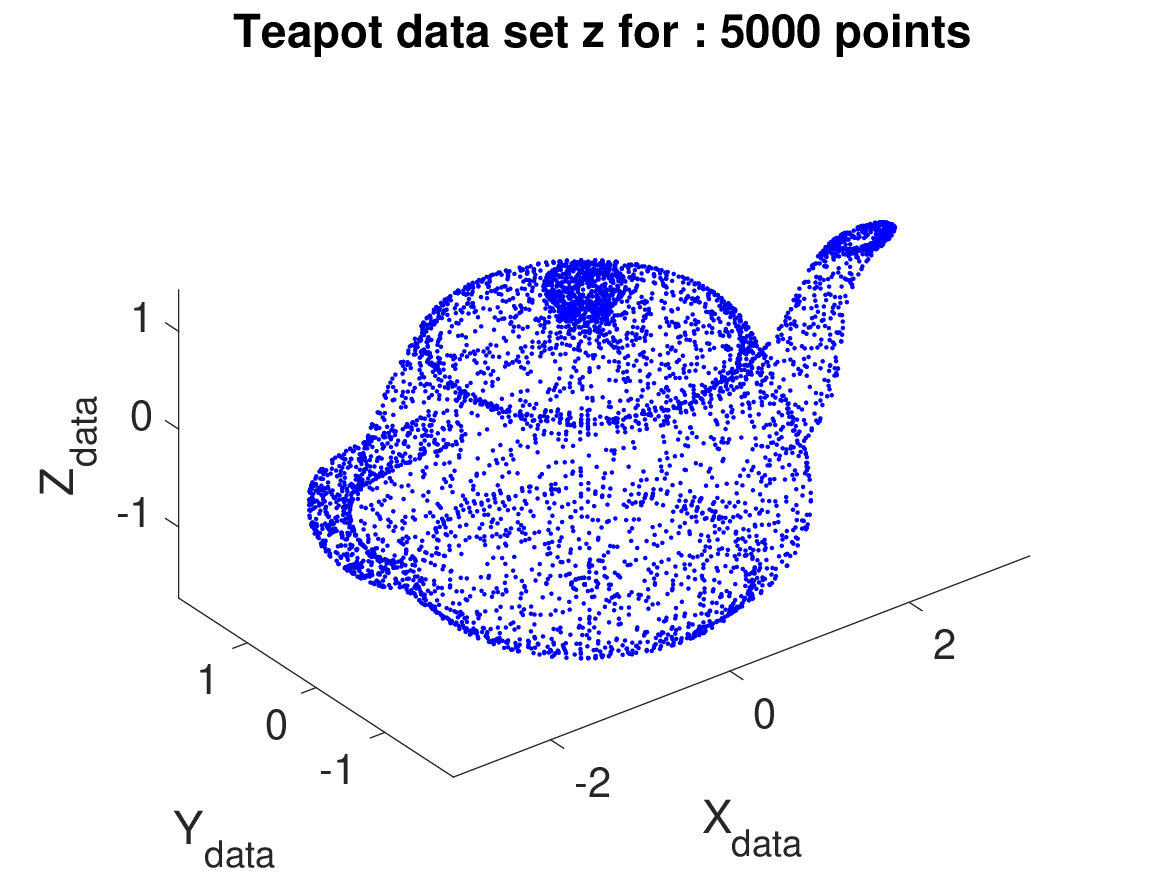}
 \caption{Left: the teapot point set(41472 points). Middle: the point set $\{y_i: i=1,\ldots, 5000\}$. Right: the point set $\{z_i: i=1,\ldots, 5000\}$. }  \label{figureRigid} 
 \end{figure}
%

    \begin{figure}
   \begin{center}
 \includegraphics[scale=0.3]{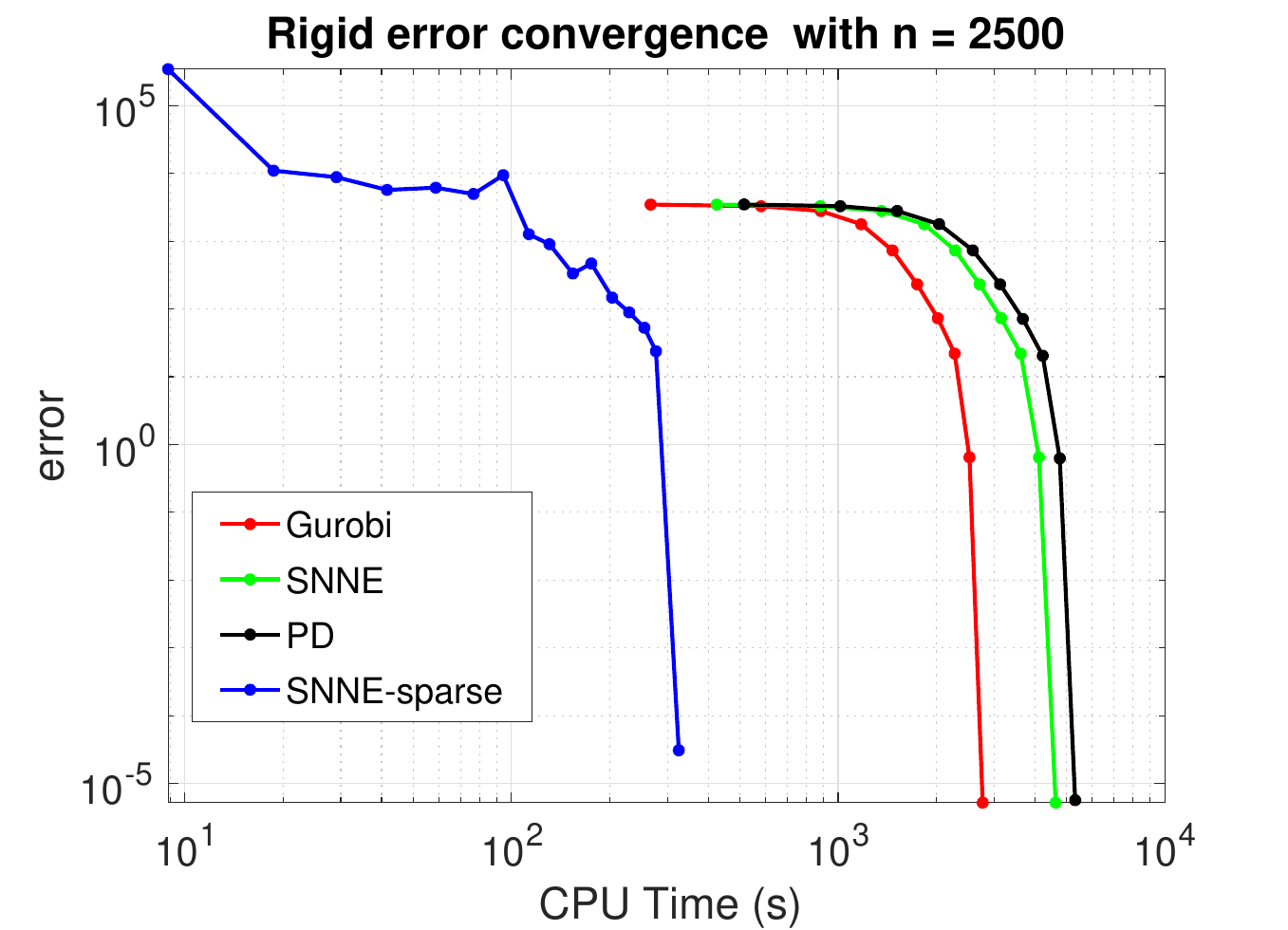}
 \includegraphics[scale=0.3]{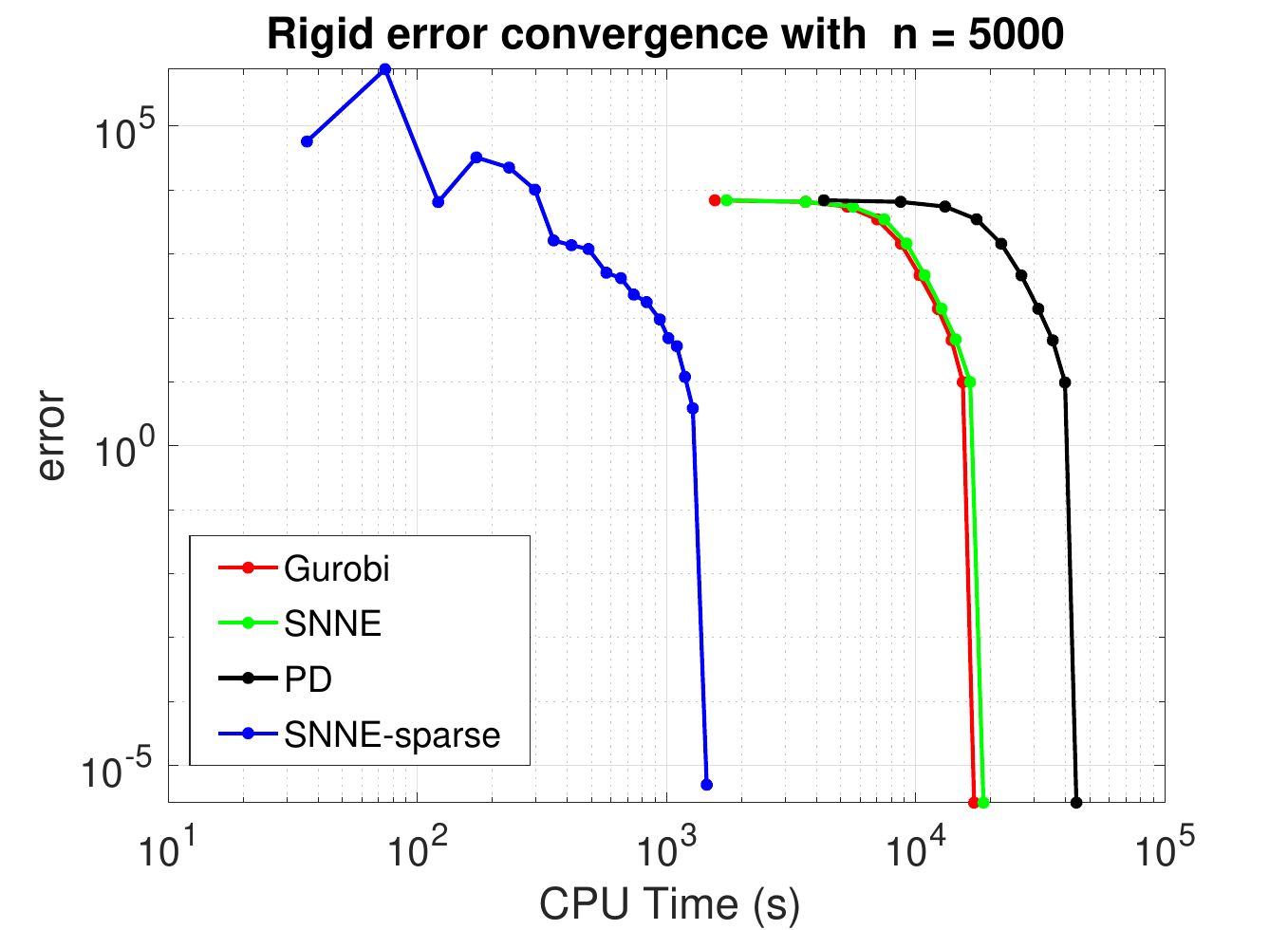}
 \end{center}
 \caption{ Computational time(sec) vs. error metric in (\ref{err}) under SNNE methods and Primal-dual methods.  }  \label{figureRigid2} 
 \end{figure}
 
%
%
%
%
%

 \begin{rem}[Multi-scale similarity]
Actually the multiplier vectors corresponding to  different cardinality $n$ resemble each other. 
The dual vector associated with coarser sampling  can be used as  one warm start to compute the dual vector associated with finer sampling. 
For instance, 
consider the application of SNNE on  the  problem with $n=2500$ and $n=5000$, respectively,
 \[
 \min_{\IT(x)\in \Pi_n}\langle c,x\rangle, \textrm{ with 
$ \IT(c)_{i,j}=\|y_i-z_j\|^2$, $i,j=1,\ldots, n$ }.
 \]
For the case $n=2500$, let $Y'=\{y_1,\ldots, y_{2500}\}$
 and $Z'=\{z_1,\ldots, z_{2500}\}$. Let $[{\nu^{(1)}}',{\nu^{(2)}}']$ be the multiplier vector in (\ref{eq_29'}).
For the case $n=5000$, let $Y=\{y_1,\ldots, y_{5000}\}$
 and $Z=\{z_1,\ldots, z_{5000}\}$. Let $[{\nu^{(1)}},{\nu^{(2)}}]$ be the multiplier vector in (\ref{eq_29'}).
The color distribution  in the top figures showing $({\nu^{(1)}}', {\nu^{(2)}}')$ resembles the color distribution in the bottom figures  showing $({\nu^{(1)}}, {\nu^{(2)}})$. Indeed,  $\nu^{(1)}\approx {\nu^{(1)}}'+160$ and $\nu^{(2)}\approx {\nu^{(2)}}'-160$. (Here the shift  is caused by  the one-dimension null space of $M$.)
Hence, we can employ $({\nu^{(1)}}', {\nu^{(2)}}')$  to produce a warm start $(\nu^{(1)}_{ini},\nu^{(2)}_{ini})$ (satisfying KKT conditions in (\ref{KKT})) to initialize $\nu$ (which initializes $\Sigma$) in the problem with $n=5000$. 
That is, 
\begin{itemize}
 \item let $\nu_{ini}^{(1)}$ be computed as follows: for $j=1,\ldots, 5000$
  \beqq
  \nu_{ini}^{(1)}(j)=\max_k \{ \|y_j-z_k\|^2-{\nu^{(2)}}' (k): z_k\in Z'\}.
  \eeqq
  \item Let $\nu_{ini}^{(2)}$ be computed as follows: for $k=1,\ldots, 5000$
  \beqq
  \nu_{ini}^{(2)}(k)=\max_j \{\|y_j-z_k\|^2-{\nu}_{ini}^{(1)}(j): y_j\in Y\}.
  \eeqq
\end{itemize}
 \end{rem}
    \begin{figure}   
   \includegraphics[width=0.4\textwidth]{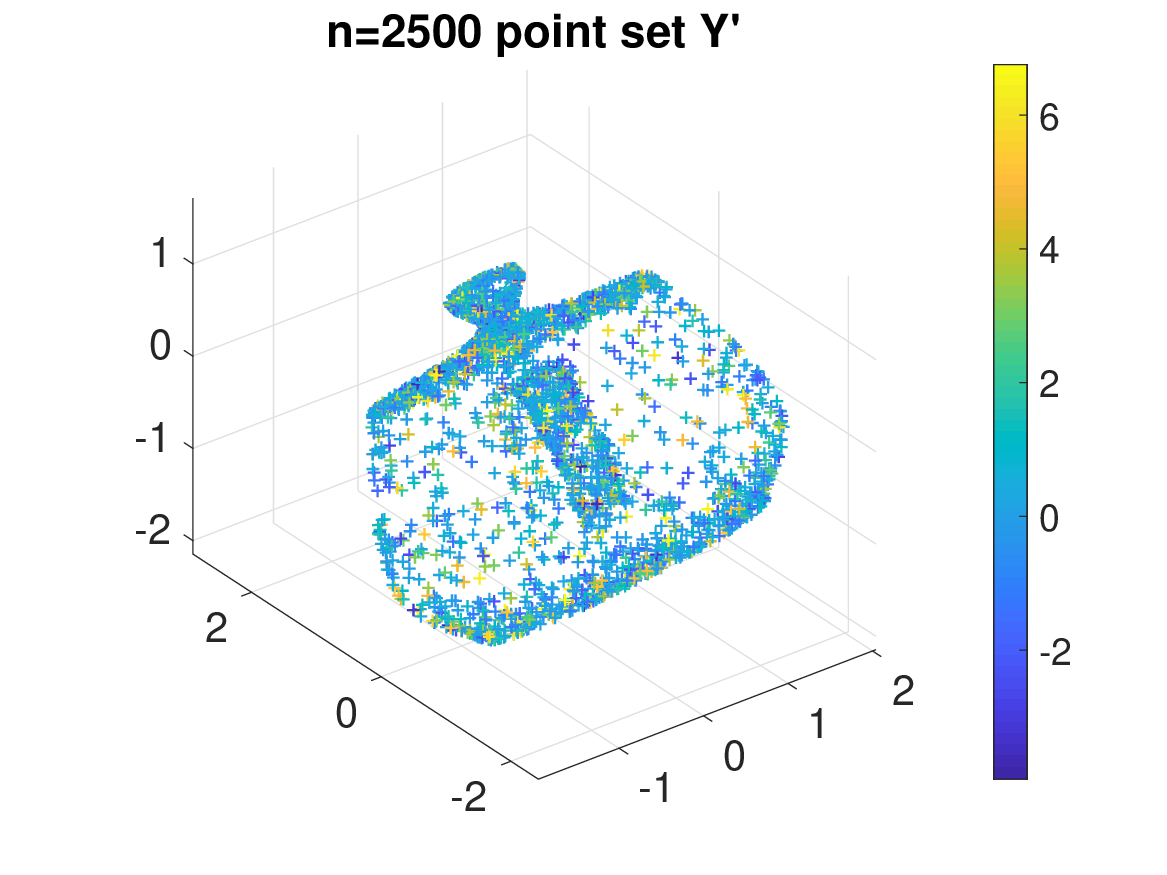} 
   \includegraphics[width=0.4\textwidth]{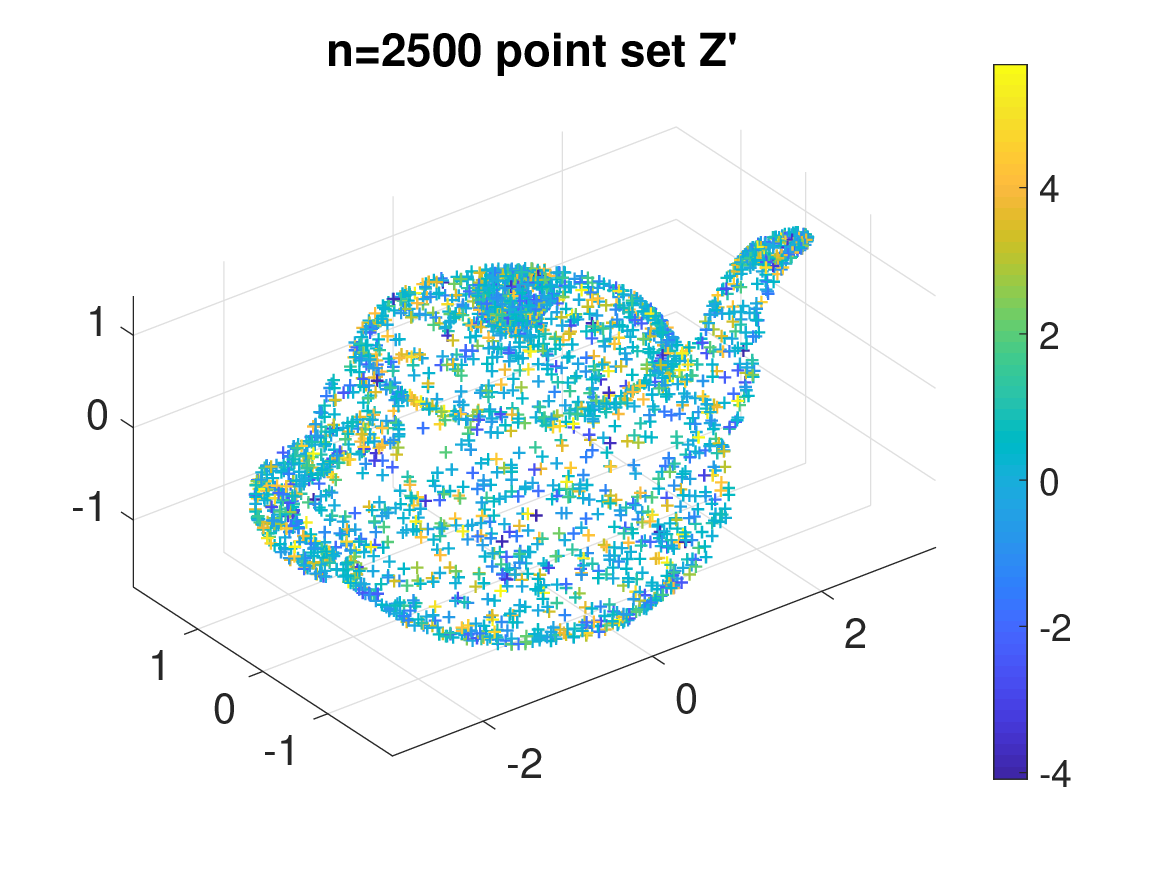} \\
        \includegraphics[width=0.4\textwidth]{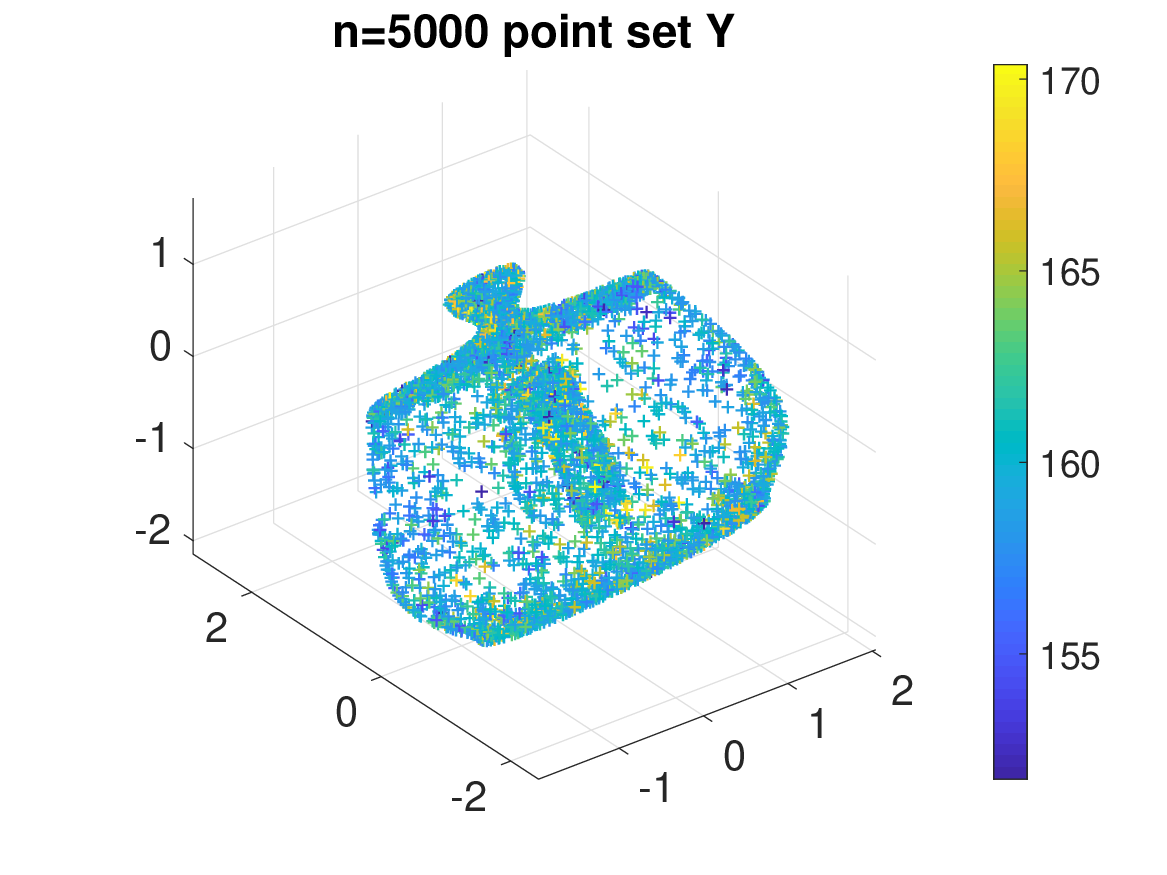} 
   \includegraphics[width=0.4\textwidth]{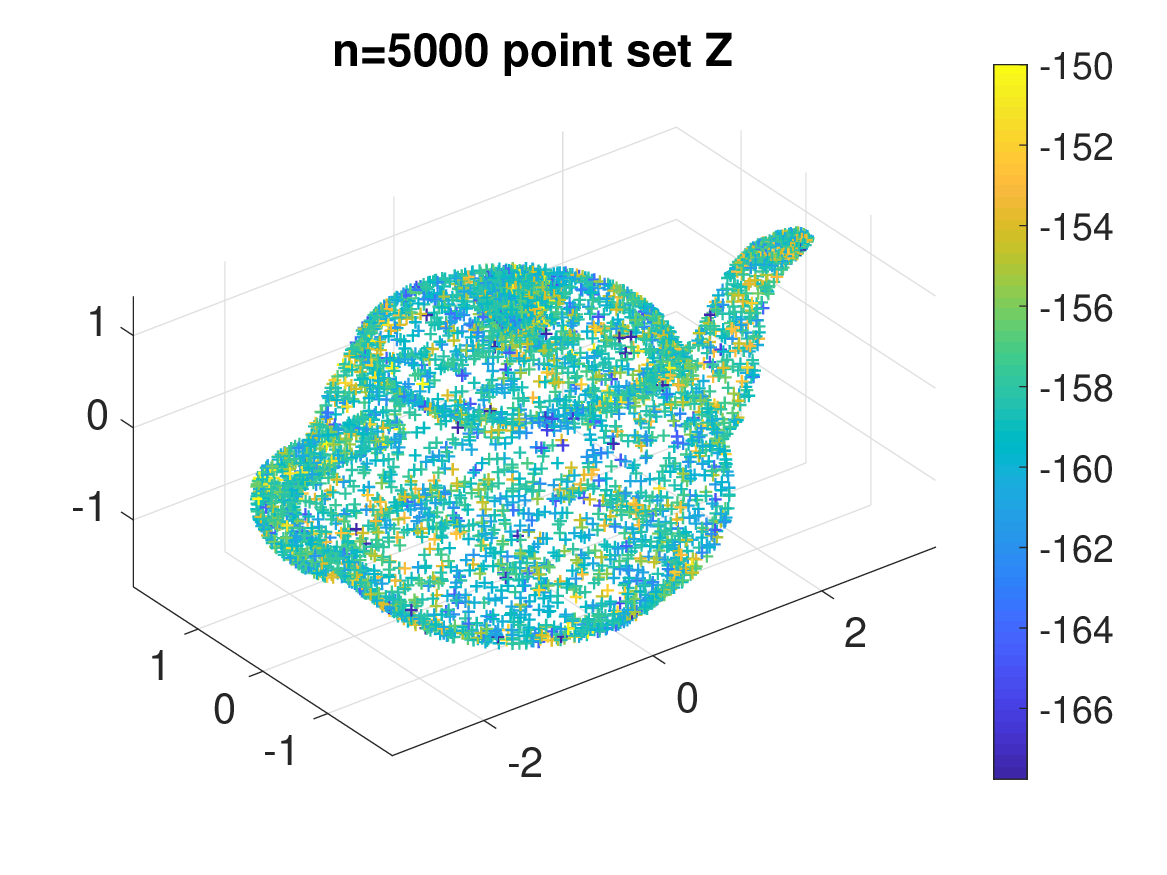}  
 \caption{Top: the point sets $Y', Z'$ with $n=2500$, respectively. 
 Bottom: the point sets $Y, Z$ with $n=5000$, respectively. The color on $Y,Y'$ illustrates  the values $\nu^{(1)}$ and ${\nu^{(1)}}'$.  The color on $Z,Z'$ illustrates  the values $\nu^{(2)}$ and ${\nu^{(2)}}'$. }  \label{Multiscale} 
 \end{figure}

 \begin{table}
\caption{Computational time(sec) based on   rigid error   reaching $10^{-4}$. }
\begin{center}

\begin{tabular}{|c|c|c|c||c|c|c|}\hline
n  &   Primal-dual & Gurobi-Barrier & SNNE & SNNE-t & SNNE-sparse & SNNE-sparse-2\\
\hline
250 &3.6 &  19.1& 17.3 & 1.3 & 2.2 & 1.0 \\
500 &40.6 &  77.3 & 134.3& 4.1 &  6.1 & 4.4\\
800 & 141.3 &236.2 & 380.7 & 11.8& 14.3 & 11.3\\
1000 & 280.0 &   375.5 & 605.3 & 24.3 &23.6 & 16.8\\
2500 & 5319 &  2834 & 4636 & 277.0 &325.3 & 106.3\\
5000 &   44230 & 17180 & 18740 & 1136 &1452 &          504.0\\
12500 & $>50000$ & MemoryError  & $>50000$  & 13959&  12001&   6502\\
\hline
\end{tabular}
\end{center}
\label{default}
\end{table}%
\subsubsection{Support sets without total support}
The following  provides one comparison between  the performance under  sparse support sets given by $\Sigma=\Sigma'\cup \Sigma''\cap \Sigma'''$ in (\ref{sigma1})  and   the performance under  sparse support sets given by $\Sigma=\Sigma''$ in  (\ref{Sigma_ep''}).  The purpose is to illustrate  the advantage of   sparse support sets with total support  over those  without total support.
As a reference, we also conduct the simulation with $\Sigma=\Pi_0$, i.e., the original complete index set as the support. 
 
 Consider the minimization in (\ref{Qx}) with $t=1$, $n=2500$. Fix $Q=I$. Use the NE method to compute $\nu$ and update $\Sigma$ according to the following five rules, including
 \begin{itemize}
 \item[(i)] Sparse  support set $\Sigma$ with $K=20$ in (\ref{sigma1});
 \item[(ii)] Sparse support set $\Sigma$ with $K=80$ in (\ref{sigma1});
 \item[(iii)]  Sparse support set $\Sigma=\Sigma''$ with $K=20$ in (\ref{Sigma_ep''});
 \item[(iv)]  Sparse support set $\Sigma=\Sigma''$ with $K=80$ in (\ref{Sigma_ep''};
  \item[(v)] The complete  set $\Pi_0:=\{(i,j): i=1,\ldots, n, j=1,\ldots, n\}$.
 \end{itemize}
   Repeat the above $(\nu, \Sigma)$-procedure $\xi_{\max}$ times to get  an approximate  minimizer $x$ for (\ref{Qx}).  Results are reported in Table~\ref{default1}-\ref{default3}.

    Table~\ref{default1} reports the matrix balancing error of $x$. 
    For (i),(ii) and (v), the support sets  have total support and  we can obtain accurate matrix balancing in these cases.  Since $t=1$ is used, the approximate solution $x$ is far from a permutation solution and we are not concerned with accurate objective values. Hence, it is not surprisingly to see some numerical gap in Table~\ref{default2}, when objective values
    in (i),(ii) and (v) are compared.  
     Indeed, as the size of support increases, more positive  terms  in $\langle \IT^{-1}(\1_\Sigma), \log x\rangle$ contribute to the increase of objective values.
      Lastly, Table~\ref{default3} reports the norm of the null vector $M^\top \nu$.   In these three cases, the norm of the corresponding dual vectors are of similar size $\sim 10^3$. 
  
     On the other hand,       since
 the set in (iii) or (iv) does not have total support,  we can not  get  accurate matrix balancing to produce acceptable objective values. 
High accurate matrix balancing here is  a very challenging task.
Due to lack of  total support, we also observe  the blow-up of the dual vector norm.  The norm is of size $\sim 10^5$.  See   (iii) and (iv) in Table~\ref{default3}.  
Under this circumstance, the vector $\nu$ with very large norm  can easily ruin the  computational accuracy  of  the  exponential functions  in $x$.

  \begin{table}
\caption{ Matrix balancing error   $\|M x-\1_{2n}\|$. }
\begin{center}
\begin{tabular}{|c|c|c|c|c|c|}\hline
$\xi_{\max}$  &  (i) $\Sigma$,  $K=20$ & (ii)$\Sigma$, $K=80$  & (iii) $\Sigma''$, $K=20$  & (iv)$\Sigma''$,  $K=80$ & (v) $\Pi_0$\\
\hline
2 & $2.27\times 10^{-5}$ &  $2.81\times 10^{-6}$ & inf & $2.55\times 10^2$ & $1.94\times 10^{-5}$ \\
4 & $5.56\times 10^{-6}$ &  $3.98\times 10^{-6}$ & inf &  inf & $1.55\times 10^{-5}$  \\
\hline
\end{tabular}
\end{center}
\label{default1}
\end{table}%

 \begin{table}
\caption{ Objective values   $\IF_t(Q,x)$.  Here ``NaN" stands for ``Not a number".}
\begin{center}
\begin{tabular}{|c|c|c|c|c|c|}\hline
$\xi_{\max}$  &  (i) $\Sigma$,  $K=20$ & (ii)$\Sigma$, $K=80$  & (iii) $\Sigma''$, $K=20$  & (iv)$\Sigma''$,  $K=80$ & (v) $\Pi_0$\\
\hline
2 & $5.81\times 10^{5}$ &  $1.57\times 10^{6}$ & NaN & $1.26\times 10^6$ & $1.02\times 10^{7}$ \\
4 & $6.35\times 10^{5}$ &  $1.48\times 10^{6}$ & NaN &  NaN & $1.02\times 10^{7}$  \\
\hline
\end{tabular}
\end{center}
\label{default2}
\end{table}%

  \begin{table}
\caption{The Frobenius norm  $\|M^\top \nu\|_F$ of dual vectors. }
\begin{center}
\begin{tabular}{|c|c|c|c|c|c|}\hline
$\xi_{\max}$  &  (i) $\Sigma$,  $K=20$ & (ii)$\Sigma$, $K=80$  & (iii) $\Sigma''$, $K=20$  & (iv)$\Sigma''$,  $K=80$ & (v) $\Pi_0$\\
\hline
2 & $3.30\times 10^3$ &  $3.42\times 10^3$ & $7.43\times 10^5$ & $7.96\times 10^5$ & $3.98\times 10^3$ \\
4 & $4.72\times 10^3$ &  $3.47\times 10^3$ & $7.42\times 10^5$ &  $7.63\times 10^5$ &  $3.98\times 10^3$\\
\hline
\end{tabular}
\end{center}
\label{default3}
\end{table}%

\subsection{Conclusion}
 Optimal transport, which is
 one assignment problem,  can be handled by many methods, including  the dual simplex method and the primal-dual methods.
With   negative entropy regularization, we can use matrix balancing algorithms to reach one approximate solution to optimal transport.  
 In the  study, we are concerned with  
 Newton method based matrix balancing algorithms to point-set matching problems,  i.e., SNNE and SNNE-sparse methods. 
 One advantage of SNNE is that 
    the method solely updates multiplier vectors along the increase of $t$, i.e., no need to store/pass  $x$ between each sub-problem.
   With the aid of sparse support,   SNNE-sparse  can be 
 a relatively convenient tool in solving  large-scale  point-set matching problems. 
  To  ensure the solution  quality from matrix balancing, 
 we employ one simple rule to update  these sparse support sets, in order to  meet    total support condition. With the aid of  total support assumption, 
 we can establish  the convergence of LB and its step size analysis, which sheds light on the convergence of KR. 

  \subsection{Data availability }
 The teapot dataset can be retrieved from  the matlab 3-D point cloud file,   `` pcread('teapot.ply')".
 The lung branch points of subject H6012 is available from the corresponding author upon request.
 \subsection{Acknowledgements}
We thank anonymous referees for helpful comments and suggestions that lead to improvement of the original manuscript.
 \appendix
 
\section{Appendix}

\subsection{Consistency  of (\ref{normal})}\label{A.1}

For $x>0$,   the null space of $M \diag(x)^2 M^\top$ has dimension $1$.

\begin{prop} Consider a positive vector $x\in \IR^{n^2}$ and a matrix $M$ in (\ref{M}). Then $M\diag(x)$ has rank $2n-1$ and \beqq\label{nulleq}
null(M \diag(x)^2 M^\top)=null(M^\top )=
span\{[\1_n; -\1_n ] \}.\eeqq
In addition, for each $r\in \IR^{n^2}$ and $x\in \IT^{-1}(\Pi_n)$,  the system \beqq\label{eq_24'} M \diag(x^2) M^\top u= M\diag(x^2)r\eeqq is consistent.
\end{prop}
\begin{proof}

Suppose $M \diag(x)^2 M^\top u=0$ for some $u\in \IR^{2n}$. Then 
\beqq
0=\langle u,  M \diag(x)^2 M^\top u\rangle=\|\diag(x)M^\top u\|^2
\eeqq
 implies $\diag(x)M^\top u=0$, i.e., $M^\top u=0$.
Hence, $null(M \diag(x)^2 M^\top)\subseteq null(M^\top )$.
Besides,  write $u=[v; w]$ with some vectors $v\in \IR^n$ and $w\in \IR^n$. Since $M^\top u=0=\1_n w^\top+v \1_n^\top=0 $, then  $u_i+w_j=0$ for all $i,j=1,\ldots, n$, i.e., 
 $u_i=u_1=-w_j$ for all $i,j$. This establishes  \[ null(M \diag(x)^2 M^\top)\subseteq null(M^\top )\subseteq span\{[\1_n; -\1_n ]\}.\]
 On the other hand,  consider a vector  in the form $u=c[\1_n; -\1_n]$ with $c\in\IR$. Then 
  $M^\top u=c(\1_n \1_n^\top-\1_n  \1_n^\top)=0$ and  $u\in null(M \diag(x)^2 M^\top)$. This completes the proof of the first part.
  Finally, note that  (\ref{eq_24'}) is the associated normal equation to the least squares problem \beqq
\min_u \|\diag(x) M^\top u-\diag(x) r\|^2.\eeqq 
Hence,  (\ref{eq_24'})  is consistent. 
\end{proof}

\subsection{Early termination}\label{sec_ET}
The following  rounding procedure could  quickly  provide a  KKT candidate  point before the degeneracy of Schur complement matrices occurs. 
Suppose that one diagonal in  $\IT(x^{(t)})$ dominates other diagonals  for some $t$. Then we have
early termination of the interior point method, i.e., 
   a permutation matrix can be identified as one optimal solution from $\IT(x^{(t)})$.
For simplicity, the following discussion does not involve  support constraints. 

\begin{prop}\label{early} Let $\gamma'\in (0,1)$ and   $\gamma''\in (1,\infty)$. Let $ \widehat \nu=[\widehat \nu^{(1)}, \widehat \nu^{(2)}]$.
Let $(\widehat x,\widehat  \nu)$ be one approximate  KKT point to (\ref{KKTt}) for some $t>0$
with the entry wise bounds \beqq\label{eq_35}
\gamma' t^{-1} \le \widehat  x \odot  \widehat  s\le \gamma'' t^{-1},\;  \widehat  s=c-M^\top\widehat  \nu, 
\eeqq 
 Let  $X:=\IT( \widehat  x)$.
Suppose that for
 some 
 permutation $\cJ:\{1,2,\ldots, n\}\to \{1,2,\ldots, n\}$,\beqq\label{eq_22'}
X_{i,j}\le  \frac{\gamma' }{\gamma''} X_{i,\cJ(i)} \textrm{  for all $j\neq \cJ(i)$,}\eeqq 
Let  $\nu:=[\nu^{(1)};\nu^{(2)}]\in \IR^{2n}$ be given by 
\beqq\label{eq_21}
\nu^{(1)}(i):= c_{i,\cJ(i)}-\nu^{(2)}(\cJ(i)),
\textrm{ where }
\nu^{(2)}:=\widehat  \nu^{(2)}.
\eeqq
Let $\widetilde X$ be the permutation, 
 \beqq \widetilde X_{i, \cJ(i)}=1 \textrm{  and  } \widetilde X_{i, j}=0, \textrm{ $j\neq \cJ(i)$ }.  \eeqq
Then   $(\widetilde x,\nu)$ is one KKT point to (\ref{KKT}), where $\widetilde x:=\IT^{-1}(\widetilde X)$.
\end{prop}

\begin{proof}

The  condition in (\ref{eq_35}) ensures that  for all $i,j=1,\ldots, n$, 
\beqq\label{KKTnu}\epsilon_{i,j}:= t X_{i,j}(
c_{i,j}-(M^\top \widehat  \nu)_{i,j})\in (  \gamma',\gamma'').
\eeqq
In particular,  for $j=\cJ(i)$,
\beqq\label{eq_50'}
c_{i,\cJ(i)}-\widehat \nu^{(1)}(i)- \widehat \nu^{(2)}(\cJ(i))\le \gamma'' (t X_{i,\cJ(i)})^{-1}.
\eeqq
We shall prove that  (\ref{KKT}) holds under this $\nu$. Let $s:=c-M^\top \nu$. From the definition in (\ref{eq_21}),  it suffices to show $s_{i,j}\ge 0$ for all  entries with $j\neq \cJ(i)$. 
From (\ref{KKTnu}) and (\ref{eq_21}), we have 
\begin{eqnarray}
s_{i,j}&:=&c_{i,j}-(M^\top \nu)_{i,j}=c_{i,j}-\nu^{(1)}(i)-\nu^{(2)}(j)\\
&\ge &c_{i,j}-c_{i,\cJ(i)}+\widehat \nu^{(2)}(\cJ(i))-\widehat \nu^{(2)}(j)- \widehat \nu^{(1)}(i)+ \widehat \nu^{(1)}(i)
\\
&\ge & (t X_{i,j})^{-1} \epsilon_{i,j}-\gamma'' \frac{X_{i,j}}{X_{i,\cJ(i)}} (t X_{i,j})^{-1}\\
&\ge & (t X_{i,j})^{-1}\left (\epsilon_{i,j}
-\gamma'\right)
\ge 0,
\end{eqnarray}
where we used the assumption in (\ref{eq_22'}) and (\ref{eq_50'}).
\end{proof}

\subsection{Proof of Theorem~\ref{LBstepsize}. }
We shall prove  Theorem~\ref{LBstepsize}.
Recall $B_k$ and $C_k$ in (\ref{eq_A},\ref{eq_c}). In addition to $u_k$ in (\ref{eq_88}),  introduce a few notations:
 \beqq
 \lambda_k^2:=\langle ( B_k-I_{2n} )\1_{2n}, ( C_k +B_k  )^{\dagger} ( B_k-I_{2n}) \1_{2n}
\rangle,\eeqq
\beqq\label{eq162}
v_k:=( B_k-I_{2n}) \1_{2n},\; 
y_k:=-( C_k +B_k  )^{\dagger}( B_k-I_{2n}) \1_{2n}.
 \eeqq
 The LB iteration $\zeta_{k+1}$ is given by 
 \beqq
 \zeta_{k+1}=\zeta_k\odot +\alpha_k u_k=\zeta_k\odot (1-\alpha_k y_k)
 \eeqq
 for some step size $\alpha_k$ within $(0, y_{\max}^{-1} (1-\epsilon_+))$, where the safeguard parameter $\epsilon_+$ and $y_{\max}$ are defined in Remark~\ref{safe}.

Before we  proceed, we  prove  one crucial property:  positive upper  bounds exist for  $\{\|C_k+B_k\|\}_{k=1}^\infty$ and $\{\|(C_k+B_k)^\dagger \|\}_{k=1}^\infty$.
Let $\zeta_*$ be one minimizer of $\mbg(\zeta)$ and $\zeta_1$ be one starting point of LB. 
We  introduce  a set of matrices, 
 \beqq
S=\{ (\zeta^{(1)} {\zeta^{(2)}}^\top )\odot \1_\Sigma \in \IR^{n\times n} : \mbg(\zeta_1)\ge \mbg(\zeta)\ge \mbg(\zeta_*), \zeta=[\zeta^{(1)} ; \zeta^{(2)}] >0 \}.
\eeqq
 Then $S$ is compact from   Prop.~\ref{prop5.5}.
Introduce $H(\zeta)$ in (\ref{Hdef}).
Note that $H(\zeta_k)=C_k+B_k$. Let $w=[w^{(1)}; w^{(2)}]$. 
We have norm estimates for $H(\zeta)$,
\begin{eqnarray}
&&\| H(\zeta)\| \le \max_w \|w\|^{-2}\left(\sum_{j=1}^n \sum_{i=1}^n A_{i,j}\zeta^{(1)}_i \zeta^{(2)}_j   ({w^{(1)}_i}-{w^{(2)}_j})^2\right)\\
& \le&
\max_w \|w\|^{-2}
\left( \sum_{i=1}^n \sum_{j=1}^n A_{i,j} \zeta^{(1)}_i \zeta^{(2)}_j (1^2+1^2) ({w^{(1)}_i}^2+{w^{(2)}_j}^2)\right)\\
&=&2 \max \{\|  (\diag(\zeta^{(1)}) A \diag(\zeta^{(2)}) ) \1_n\|_\infty, \| (\diag(\zeta^{(1)}) A \diag(\zeta^{(2)}) )^\top \1_n\|_\infty\}.
\end{eqnarray}
We have the following upper bound, 
 \beqq
\| H(\zeta)\| \le 2 \left(\max_{(i,j)\in \Sigma} A_{i,j} \right)
\max\left( \| (\zeta^{(1)} {\zeta^{(2)}}^\top \odot \1_\Sigma)\1_n \|_\infty, \| (\zeta^{(1)} {\zeta^{(2)}}^\top \odot \1_\Sigma )^\top \1_n \|_\infty 
\right).
\eeqq
Thanks to Prop.~\ref{prop5.5},  
a  constant $\mbM$  exists
as  a  upper bound for  $\|H(\zeta)\|$.
On the other hand, for each $w$, we can express
\beqq
\langle  w, H(\zeta) w\rangle=
\langle A\odot (\zeta^{(1)} {\zeta^{(2)}}^\top\odot \1_\Sigma),
(w^{(1)}  \1_n^\top -\1_n {w^{(2)}}^\top )\odot (w^{(1)}  \1_n^\top -\1_n {w^{(2)}}^\top )
\rangle
\eeqq
as one function defined on $S$.
The null space of $H(\zeta)$ is  $\cN$ for each $\zeta>0$ from Prop.~\ref{prop3.3}.
Consider the following  function to characterize  the smallest positive eigenvalue  of $H(\zeta)$,
 \beqq
 \widehat H(\zeta):=\min_w \left\{ \frac{\langle w, H(\zeta) w\rangle }{\|w\|^2}: w \textrm{ is orthogonal to } \cN\right\}>0.
 \eeqq
Since $S$ is compact, then 
a positive  constant $\mbm$  exists as
 a lower bound for  the smallest positive eigenvalue  of $H(\zeta)$. Hence, $\|H(\zeta)^\dagger \|\le \mbm^{-1}$.  In summary,   for all $k$, we have \beqq\label{Mm}
  \|C_k+B_k\|\le \mbM, \|(C_k+B_k)^\dagger \|\le \mbm^{-1}.
  \eeqq
In addition,  $\|B_k\|\le \mbM_1$ holds
for some   constant $\mbM_1$.

The  convergence of LB can be established by  standard arguments in  section 9.5 in\cite{Boyd}.
Introduce a function of $\alpha$, \beqq  \mbgt(\alpha):=\mbg(\zeta_k+ \alpha u_k).\eeqq
Let $\mbgt'$ and $\mbgt''$ denote the first derivate and the second derivate of $\mbgt$, respectively.
Calculus shows \beqq
\nabla \mbg(\zeta)= \widetilde A \zeta-\zeta^{-1},\; 
\nabla^2 \mbg(\zeta)= \widetilde A +\diag(\zeta^{-2}).
\eeqq

\begin{prop}[Damped Newton phase]\label{A6} Let $\epsilon_+$ be the safeguard parameter in Remark~\ref{safe}.
Then 
\beqq
\lim_{k\to \infty}\lambda_k=0,\; \lim_{k\to \infty }\|y_k\|=0,\; \lim_{k\to \infty }\|v_k\|= 0.\eeqq
  \end{prop}
\begin{proof}
First, we show that the limit of step size interval in Remark~\ref{safe} is not zero. 
Indeed,    since $y_{\max}\le \|y_k\|\le \|(C_k+B_k)^\dagger \| \|(B_k-I_{2n}) \1_{2n}\|\le \mbm^{-1} (\mbM_1+1) \|\1_{2n}\|$
 for each $\zeta_k$ with $\mbg(\zeta_k)\le c_0$, then   $y_{\max}^{-1} (1-\epsilon_+)$  stays away from $0$ for each iteration.
Second,  
 we show 
\beqq
\mbgt(\alpha)- \mbgt(0)\le (-\alpha+\frac{\alpha^2}{2} (\mbM_1+\epsilon_+^{-2}) \mbm^{-1}) \lambda_k^2.
\eeqq
Indeed,  Taylor's formula indicates that  for some scalar $\widetilde \alpha\in [0,\alpha]$,
\begin{eqnarray}
&&\mbgt(\alpha)=\mbgt(0)+\mbgt'(0)\alpha+\mbgt''(\widetilde  \alpha)\frac{\alpha^2}{2}\\
&\le&  \mbg(\zeta)+\alpha \nabla \mbg(\zeta)^\top u_k+\frac{\alpha^2}{2} \|\diag(\zeta_k) (\nabla^2 \mbg(\zeta_k+\widetilde  \alpha u_k))\diag( \zeta_k)\| \|y_k\|^2\\
&\le &\mbg(\zeta)+\alpha (-\lambda_k^2)+\frac{\alpha^2}{2} (\|B_k\|+\| (1+\widetilde  \alpha y_k)^{-2}\|_\infty) \|y_k\|^2\\
&\le &\mbgt(0)+\alpha (-\lambda_k^2)+\frac{\alpha^2}{2} (\mbM_1+\epsilon_+^{-2}) \mbm^{-1} \lambda_k^2.
\end{eqnarray}
Together,  the step size $\alpha_k$ in   backtracking line search is bounded below by some positive constant. Since $\mbg(\zeta)$ is bounded below, then $\lambda_k^2$ must tend to $0$, as $k\to \infty$.
From  (\ref{eq162}), we have 
 $\|v_k\|^2\le \lambda_k^2 \|C_k+B_k \|\le \mbM \lambda_k^2$ and $\|y_k\|^2\le \|(C_k+B_k)^\dagger \| \lambda_k^2\le \mbm^{-1}\lambda_k^2$, which completes the proof. 
\end{proof}

\begin{prop}[$\alpha_k=1$ phase]
As $k$ is sufficiently large, we have $\alpha_k=1$.
\end{prop}
\begin{proof}Let  $\epsilon_+>0$ be 
the safeguard parameter in Remark~\ref{safe}.  Let  $L=\epsilon_+ ^{-3}$.
Since $u_k=-\zeta_k\odot y_k$,
\begin{eqnarray}
&&|  \mbgt''(\alpha)-\mbgt''(0)
|
\le | u_k^\top (\nabla^2\mbg( \zeta_k+\alpha u_k) -\nabla^2 \mbg(\zeta_k)) u_k |\\
&\le &
| u_k^\top (\diag(\zeta_k+\alpha u_k)^{-2}-\diag(\zeta_k^{-2})) u_k|=
 \left| y_k^\top \{(\frac{1}{1-\alpha y_k})^2-1\} y_k \right|\le
 \alpha L \|y_k\|^3.
\end{eqnarray}
Hence, 
$
\mbgt''(\alpha)\le \mbgt''(0)+\alpha L \|y_k\|^3.
$
By integration, we have
$
\mbgt'(\alpha)\le \mbgt'(0)+ \alpha \mbgt''(0)+ \frac{\alpha^2}{2} L\|y_k\|^3$,
and
\begin{eqnarray}
\mbgt(\alpha)- \mbgt(0)&\le& \alpha \mbgt'(0)+ \frac{\alpha^2}{2}  \mbgt''(0)+ \frac{\alpha^3}{6}L \|y_k\|^3\\
&\le &- \alpha \lambda_k^2+ \frac{\alpha^2}{2}  ( \lambda_k^2+\| v_k\|_\infty \|y_k\|^2)+ \frac{\alpha^3}{6}L \|y_k\|^3\\
&\le &\lambda^2(- \alpha + \frac{\alpha^2}{2}  ( 1+\| v_k\|_\infty  \mbm^{-1})+ \frac{\alpha^3}{6}L  \mbm^{-3/2}\lambda_k )\label{eq182}
\end{eqnarray}
where we used 
  \begin{eqnarray}
\mbgt ''(0)&=& \langle u_k, \nabla^2 \mbg(\zeta_k) u_k \rangle\\
&=&
\langle 
(C_k+B_k)^\dagger (B_k-I_{2n} ) \1_{2n}, (I_{2n}+B_k) (C_k+B_k)^\dagger (B_k-I_{2n} ) \1_{2n}
\rangle
\\
&=&\lambda_k^2+
\langle
 (C_k+B_k)^\dagger (B_k-I_{2n} ) \1_{2n}, (I_{2n} -C_k) (C_k+B_k)^\dagger (B_k-I_{2n} ) \1_{2n} \rangle 
\\
&\le &
\lambda_k^2+ \| C_k-I_{2n}\| \|y_k \|^2=\lambda_k^2+\| B_k \1_{2n}-\1_{2n}\|_\infty  \|y_k \|^2=\lambda_k^2+\|  v_k \|_\infty  \|y_k \|^2.
\end{eqnarray}

Take $\alpha=1$ in (\ref{eq182}). Using $\|y_k\|^2\le \lambda_k^2 \mbm^{-1}$ and $\|v_k\|_\infty\le \|v_k\|\le \mbM^{1/2} \lambda_k$ from (\ref{Mm}), we have
\beqq\label{eq167}
\mbgt(1)- \mbgt(0)\le   \frac{\lambda_k^2}{2}  \left( -1+(\mbM^{1/2}   \mbm^{-1}+ \frac{\alpha^3}{3}L  \mbm^{-3/2}) \lambda_k \right)
\eeqq
Note that  $\lim_{k\to \infty }\lambda_k=0$ from Prop.~\ref{A6}.
 When $\lambda_k$ is sufficiently close to $0$, $\alpha_k=1$ is  accepted by the backtracking line search.  That is,  for $k$ sufficiently large, (\ref{eq167}) indicates that \beqq\mbgt(1)- \mbgt(0)\le 
\beta \nabla \mbg(\zeta_k)^\top u_k=-\beta\lambda_k^2\eeqq
holds with backtracking parameter $\beta\in (0,1/2)$. 

\end{proof}

\bibliographystyle{alpha}
\bibliography{IPMB0}

%
%
%
%
%
\end{document}